\documentclass[12pt]{article}
\usepackage{amsmath,amssymb,amsfonts,amsthm}
\newenvironment{myabstract}{\par\noindent
{\bf Abstract . } \small }
{\par\vskip8pt minus3pt\rm}
\newcounter{item}[section]
\newcounter{kirshr}
\newcounter{kirsha}
\newcounter{kirshb}
\newenvironment{enumroman}{\setcounter{kirshr}{1}
\begin{list}{(\roman{kirshr})}{\usecounter{kirshr}} }{\end{list}}
\newenvironment{enumarab}{\setcounter{kirshb}{1}
\begin{list}{(\arabic{kirshb})}{\usecounter{kirshb}} }{\end{list}}
\newenvironment{athm}[1]{\vskip3mm\par\noindent
{\bf #1 }. \slshape }
{\upshape\par\vskip10pt minus3pt}
\newtheorem{theorem}{Theorem}[section]

\newtheorem{lemma}[theorem]{Lemma}
\newtheorem{corollary}[theorem]{Corollary}
\newenvironment{demo}[1]{\noindent{\bf #1.}\upshape\mdseries}
{\nopagebreak{\hfill\rule{2mm}{2mm}\nopagebreak}\par\normalfont}
\theoremstyle{definition}

\newtheorem{example}[theorem]{Example}
\newtheorem{definition}[theorem]{Definition}
\def\R{\mathbb{R}}

\def\C{{\mathfrak{C}}}
\def\Fm{{\mathfrak{Fm}}}

\def\At{{\bf At}}
\def\Nr{{\mathfrak{Nr}}}

\def\Sg{{\mathfrak{Sg}}}
\def\Fm{{\mathfrak{Fm}}}
\def\A{{\mathfrak{A}}}
\def\B{{\mathfrak{B}}}
\def\C{{\mathfrak{C}}}
\def\D{{\mathfrak{D}}}
\def\M{{\mathfrak{M}}}
\def\N{{\mathfrak{N}}}
\def\Sn{{\mathfrak{Sn}}}
\def\CA{{\bf CA}}

\def\SC{{\bf SC}}
\def\QEA{{\bf QEA}}

\def\Df{{\bf Df}}

\def\Lf{{\bf Lf}}

\def\PA{{\bf PA}}
\def\PEA{{\bf PEA}}

\def\K{{\bf K}}
\def\K{{\bf K}}

\def\RCA{{\bf RCA}}

\def\Rd{{\ Rd}}
\def\(R)RA{{\bf (R)RA}}
\def\RA{{\bf RA}}
\def\RRA{{\bf RRA}}

\def\R{\mathbb{R}}

\def\Sc{{\bf Sc}}

\def\c #1{{\cal #1}}
 \def\CA{{\sf CA}}
\def\B{{\sf B}}
\def\G{{\sf G}}
\def\w{{\sf w}}
\def\y{{\sf y}}
\def\g{{\sf g}}

\def\r{{\sf r}}
\def\K{{\sf K}}
\def\tp{{\sf tp}}

 \def\Cm{{\mathfrak{Cm}}}
\def\Nr{{\mathfrak{Nr}}}

\def\restr #1{{\restriction_{#1}}}
\def\cyl#1{{\sf c}_{#1}}
\def\diag#1#2{{\sf d}_{#1#2}}

\def\Ra{{\mathfrak{Ra}}}
\def\Ca{{\mathfrak{Ca}}}
\def\set#1{\{#1\} }
\def\Ra{{\mathfrak{Ra}}}
\def\Nr{{\mathfrak{Nr}}}
\def\Tm{{\mathfrak{Tm}}}
\def\A{{\mathfrak{A}}}
\def\B{{\mathfrak{B}}}
\def\C{{\mathfrak{C}}}
\def\D{{\mathfrak{D}}}

\def\A{{\mathfrak{A}}}
\def\B{{\mathfrak{B}}}
\def\C{{\mathfrak{C}}}
\def\D{{\mathfrak{D}}}

\def\GG{{\mathfrak{GG}}}

\def\L{{\mathfrak{L}}}
\def\Rd{{\mathfrak{Rd}}}

\def\At{{\mathfrak{At}}}
\def\L{{\mathfrak{L}}}
\def\Bl{{\mathfrak{Bl}}}
\def\CA{{\bf CA}}
\def\RA{{\bf RA}}
\def\RRA{{\bf RRA}}
\def\RCA{{\bf RCA}}

\def\G{{\bf G}}

\def\F{{\mathfrak{F}}}
\def\At{{\sf{At}}}
\def\N{\mathbb{N}}
\def\R{\mathfrak{R}}

\def\CCA{{\bf CCA}}
\def\Cs{{\bf Cs}}
\def\RPEA{{\bf RPEA}}

\def\cyl#1{{\sf c}_{#1}}
\def\diag#1#2{{\sf d}_{#1#2}}

\def\c #1{{\cal #1}}

\def\pa{$\forall$}
\def\pe{$\exists$}

\def\nodes{{\sf nodes}}

\def\restr #1{{\restriction_{#1}}}

\def\Ra{{\mathfrak{Ra}}}
\def\Nr{{\mathfrak{Nr}}}
\def\Z{{\cal Z}}
\def\CA{{\bf CA}}
\def\RCA{{\bf RCA}}
\def\c#1{{\mathcal #1}}

\def\set#1{ \{#1\}}

\def\Ca{{\mathfrak Ca}}

\def\pe{$\exists$}
\def\pa{$\forall$}
\def\Cm{{\mathfrak Cm}}
\def\Sg{{\mathfrak Sg}}
\def\P{{\mathfrak P}}
\def\Rl{{\mathfrak Rl}}
\def\N{{\cal N}}

\def\At{{\sf At}}

\def\rng{{\sf rng}}
\def\dom{{\sf dom}}

\def\w{{\sf w}}
\def\g{{\sf g}}
\def\y{{\sf y}}
\def\r{{\sf r}}
\def\bb{{\sf b}} 
\def\tp{{\sf tp}}
\def\cyl#1{{\sf c}_{#1}}

\def\diag#1#2{{\sf d}_{#1#2}}

\def\ws{winning strategy}

\def\y{{\sf y}}
\def\g{{\sf g}}
\def\bb{{\sf b}}
\def\r{{\sf r}}
\def\w{{\sf w}}

\title{On Completions, neat atom structures, and omitting types}
\author{Tarek Sayed Ahmed \\
Department of Mathematics, Faculty of Science,\\ 
Cairo University, Giza, Egypt.
  }
\begin{document}
\maketitle
\begin{myabstract} This paper has a survey character, but it also contains several new results.
The paper tries to give a panoramic picture of the recent developments in algebraic logic.
We take a long magical tour in algebraic logic starting from classical notions due to Henkin Monk and Tarski
like neat embeddings, culminating in presenting sophisticated model theoretic constructions based on graphs, to solve problems on neat reducts.

We investigate several algebraic notions that apply to varieties of 
boolean algebras with operators in general $BAOs$, like canonicity, atom-canonicity and completions. We also show that in certain significant
special cases, when we have a Stone-like representabilty notion, like in cylindric, relation and polyadic algebras
such abtract notions turn out intimately related 
to more concrete notions, like complete representability, 
and the existence of weakly but not srongly representable atom structures.

In connection to the multi-dimensional  corresponding modal logic, we 
prove several omitting types theorem for the finite $n$ variable fragments of first order logic, 
the multi-dimensional modal logic corresponding to $CA_n$; the class of cylindric algebras of dimension $n$.

A novelty that occurs here is that
theories could be uncountable. Our constructions depend on deep model-theoretic results of Shelah. 

Several results mentioned in \cite{OTT} without proofs are proved fully here, one such result is:
There exists an uncountable atomic algebra in $\Nr_n\CA_{\omega}$ that is not completely representable. 
Another result: If $T$ is an $L_n$ theory (possibly uncountable), where $|T|=\lambda$, $\lambda$ is a regular cardinal, and $T$ admits elimination of 
quantifiers, then
$<2^{\lambda}$ non principal types {\it maximal} can be omitted. 

A central notion, that connects, representations, completions, complete representations for cylindric algebras is that of neat embedding,
which is an algebraic counterpart of Henkin constructions, and is a nut cracker in cylindric-like 
algebras for proving representation results and related results concerning various forms of the amalagmation property
for clases of representable algebras. 
For example, representable algebras are those algebras that have the neat embeding property, 
completey representable countable ones, are the atomic algebras that have the complete neat embedding 
property. We show that countability cannot be omitted which is sharp in view to our omitting types theorem mentiond above.
We show that the class $\Nr_n\CA_{\omega}$ is psuedo-elementary, not elementary, and its elementary closure is not finitely axiomatizable
for $n\geq 3$, We  characterize this class by games.

We give two constructions for weakly representable atom structures that is not 
strongly representable, that are simple variations on existing themes, and using fairly straightforward modificatons 
of constructons of Hirsch and Hodkinson, we show that
the latter class is not elementary for any reduct of polyadic algebras contaiting all cylindrifiers. 

We introduce the new notions of strongly neat, weakly neat and very weakly neat atom structures.
An $\alpha$ dimensional  atom strucure is  very weakly neat, $\alpha$
an ordinal,  if no algebra based on it in 
$\Nr_{\alpha}\CA_{\alpha+\omega};$ weakly neat if it has at least one algebra based on it that 
is in $\Nr_{\alpha}\CA_{\alpha+\omega}$, and finally  strongly neat if every algebra based on it is in 
$\Nr_{\alpha}\CA_{\alpha+\omega}.$ We give examples of the first two, show that they are distinct, 
and further show that the class of weakly 
neat atom structures is not elementary.
This is done for all dimensions $>2$, infinite included. For the third, we show that finite atom structures are strongly neat (in finite dimensions).

Modifying several constructions in the literature, as well as providing new ones, several 
results on complete representations and completions are presented, 
answering several questions posed by Robin Hirsch, and Ian Hodkinson, concerning relation algebras, and complete representabiliy
of reducts of plyadic algebras.
 
\end{myabstract}

\section{Introduction}

Atom canonicity, completions, complete representations and omitting types are  four notions that could appear at first glimpse unrelated. 
The first three, are algebraic notions that apply 
to varieties of Boolean algebras with operators $BAO$s. Omitting types is a metalogical notion that applies to the corresponding multi-modal logic. 

Canonicity is one of the most important concepts of modal 
completeness theory. From an algebraic perspective, canonical models are not abstract oddities, 
they are precisey the structure one is led to 
by underlying the idea in Stones representabilty theory for Boolean algebras.

The canonical extension of an algebra  has universe the power set algebra of 
the set of ultrafilters; that is its Stone space, and  the extra non-Boolean operations induced naturally from the original ones.
A variety is canonical if it is closed under taking canonical
extensions. 

This is typically a {\it persistence property}. Persistence properties refer to closure of a variety $V$  under 
passage from a given member in $V,$ to some 'larger' algebra .

The other  persistence property, namely, atom-canonicity, concerns the atom structure $\At\A$ of an atomic algebra $\A$. 
As the name suggests, $\At\A$ is a certain relational structure based on the set of atoms of $\A$. 
A variety $V$ is atom-canonical if it contains the complex algebra $\Cm\At(\A)$ whenever it contains $\A$.
If $\At V$ is the class of all atom structures of atomic algebras in $V$, 
then atom-canonicity amounts to the requirement that $\Cm\At V\subseteq V$.%

The (canonical) models of a multi-modal logic $\L_V$, corresponding to a canonical variety $V$,  are Kripke frames; 
these are the ultrafilter frames. The  atom structures, are special cases, we call these atomic 
models.

The canonical extension of an algebra is a complete atomic algebra that the original algebra embeds into, however it only preserves finite meets. 
Another completion is the Dedekind MacNeille completion, also known as the 
minimal completion. Every completely  additive $BAO$ has such a completion. It is uniquely 
determined by the fact that it is complete and the original algebra is dense in it; hence it preserves all meets existing in the original 
algebra.
The completion is atomic if and only if the original algebra is. The completion and canonical extension of an algebra 
only coincide when the algebra is finite.

Complete representations has to do with algebras that have a notion of representations involving - using jargon of modal logic - complex algebras 
of square frames or, using algebraic logic jargon, full set algebras having square units, and also having the 
Boolean operations of (concrete) intersections and complementation,
like relation algebras and cylindric algebras.  Unlike atom-canonicity, this notion is semantical.
Such  representations is a representation that carries existing (possibly infinite) meets to set theoretic 
intersections.

Atomic representability is also related to the Orey-Henkin omitting types theorem. Let $V$ be a variety of $BAO$'s which has a notion of representation,
like for example $CA_n$. The variety $CA_n$ corresponds to the syntactical part of $L_n$ 
first order logic restricted to the first $n$ variables, while $RCA_n$ corresponds to the semantical models 

Indeed given an $L_n$ theory $T$ and a model $\M$ of $T$, let $\phi^M$ denote the set of$n$-ary assignments satisfying
$\phi$ in $\M$, notationaly $\phi^{\M}=\{s\in {}^nM: \M\models \phi[s]\}$.
Then $\{\phi^M:\phi\in L\}$ is the universe of a cylindric set algebra $Cs_n$ with the operations of cyindrifiers corresponding to 
the semantics of existential quabtifiers
and diagonal elements to equality. The class of subdirect products of algebras in 
$Cs_n$ is the class $RCA_n$.

A set $\Gamma$ in the language of $T$ is said to be omitted by the model $\M$ of $T$, 
if $\bigcap \phi^{\M}=\emptyset$. This can be formulated algebraically as follows:
Let $\A\in \CA_n$ and  let there be given a family $(X_i: i\in I)$ of subsets of $\A$,  then there exists an injective homomorphism 
$f:\A\to \wp(V)$, $V$ a disjoint union of cartesian squares, that  omits the $X_i$'s, that is $\bigcap_{x\in X_i} f(x)=0$.

The Orey-Henkin omitting types theorem says that this always happens if $\A$ is countable and locally finite,
and when $I$ is countable, and  the $X_i'$ contain only finitely many 
dimensions (free variables). 
But it is clear that the above algebraic formulation lends itself to other contexts.

Note that if omitting types theorem holds for arbitrary cardinalities, then the 
atomic representable algebras are completely representable, by finding a representation that omits the 
non principal types $\{-x: x\in \At\A\}$. The converse is false. There are easy examples. Some will be provided below.

Now let us try to find a connection between such notions. 
Consider a variety$V$ of $BAO$s that is {\it not} atom-canonical. Ths means that there is an $\A\in V$, such that $\Cm\At\A\notin V$. 
If $V$ is completely aditive, then $\Cm\At\A$, is the completion of $\A$. So $V$ is not closed 
under completions. 

There are significant varieties that are not atom-canonical, like the variety of representable cylindric algebras 
$\RCA_n$ for $n\geq 3$ and representable relation algebras $\RRA$.
These have a notion of representability involving square units. 

Let $\A$ be an atomic representable such that $\Cm\At\A\notin \RCA_{n}$, for $n\geq 3$. Such algebras exist, the first of its kind was constructed
by Hodkinson. The term algera $\Tm\At\A$, which is the subalgebra of the complex algebra,
is contained in $\A$, because $\RCA_n$ is completely additiv; furthermore, it is 
representable but it {\it cannot} have a complete representation, for such a representation
would necessarily induce a representation of $\Cm\At\A$. That is to say, $\A$ is an example of an atomic theory in the cylindric modal logic of dimension
$n$, but it has no atomic model. 
In such a context, $\At\A$ is also an example of a {\it weakly representable} atom structure that is not {\it strongly representable.}

A weakly representable atom structure is an atom structure such that there is at least one algebra based on this atom structure that is 
representable \cite{w}. It is strongly representable, if every algebra based on it is representable.
Hodkinson, the first to construct a weakly representable cylindric atom structure that is not strongly representable, 
used a somewhat complicated construction depending on so called Rainbow 
constructions. His proof is model-theoretic.
In \cite{w} we use the same method to construct such an atom structure for both relation and cylindric algebras.
On the one hand, the graph we used substantially simplifies Hodkinson's construction, 
and, furthermore, we get our result for relation and cylindric algebras in one go.


Hirsch and Hodkinson show that the class of strongly representable atom structures of relation algebras 
(and cylindric algebras) is not elementary \cite{strong}.
The construction makes use of the pobabilistic method of  Erd\"os to show that there are finite graphs with arbitrarily large 
chromatic number and girth.   
In his pioneering paper of 1959, Erdos took a radically new approach to construct such graphs: for each $n$ he 
defined a probability space on the set of graphs with $n$ vertices, and showed that, for some carefully chosen probability measures, 
the probability that an $n$ vertex graph has these properties is positive for all large enough $n$.
This approach, now called the {\it probabilistic method} has since unfolded into a 
sophisticated and versatile proof technique, in graph theory and in other 
branches of discrete mathematics.
This method was used first in algebraic logic by Hirsch and Hodkinson to show that the class of strongly representable atom structures of cylindric 
and relation algebras is not elementary  
and that varieties of representable relation algebras are barely canonical \cite{hv}.
But yet again using these methods of Erd\"os in  \cite{ghv} it is shown 
that there exist continuum-many canonical equational classes of Boolean algebras 
with operators that are not generated by the complex algebras of any first-order 
definable class of relational structures. Using a variant of this construction  the authors resolve the long-standing question of Fine, 
by exhibiting a bimodal logic that is valid in its canonical frames, but is not sound and complete for any first-order definable class of 
Kripke frames. 

There is an ongoing interplay between algebraic logic on the one hand, and model theory and finite combinatorics particularly graph theory, 
on the other. 
Monk was the first to use Ramsey's theorems to construct what is known an Monk's algebras.
witnessing non finite axiomatizability for 
the class of representable cylindric algebras.
The key idea of the construction of a Monk's algebra is not so hard. 
Such algebras are finite, hence atomic, more precisely their Boolean reducts are atomic.
The atoms are given colours, and 
cylindrifications and diagonals are defined by stating that monochromatic triangles 
are inconsistent. If a Monk's algebra has many more atoms than colours, 
it follows from Ramsey's Theorem that any representation
of the algebra must contain a monochromatic triangle, so the algebra
is not representable. 

Later Monk-like constructions were substantially generalized by Andr\'eka N\'emeti \cite{ANT}, 
Maddux \cite{m}, and finally Hirsch and Hodkinson \cite{HHM}.
Constructing algebras from Erdos graphs have 
proved extremely rewarding \cite{strong}, \cite{hv}, \cite{ghv}. 
Another construction invented by Robin Hirsch and Ian Hodkinson is the so-called rainbow construction, which  is an ingenious technique 
that has been used to show that several classes are not elementary \cite{HH}, \cite{HHbook}, \cite{r}, and was used together with a 
lifting argument of Hodkinson of construction polyadic algebras from relation algebras to show that it is undecidable whether a finite relation or 
cylindric algebra is representable.
This shows that  ceratin important products modal logics are undecidable.
We will use the rainbow construction below to prove the $CA$ analogue of a deep result in \cite{r}.

Constructing cylindric algebras based on certain models satisfyig certain properties like homogeniouty, saturation, 
elimination of quantifiers, using model theory like in \cite {MLQ}, will be generalized below to answer a question of Hirsch \cite{r},
on relation algebra reducts of cylindric algebras.

Another model theoretic construction of Hodkinson, 
based on rainbow graphs, considerably simplified in \cite{w} will be further simplified here 
to prove  that several varieties approximating the class of representable cylindric algebras are not closed under completions.

\section*{ The main new results}

\begin{enumarab}

\item  Answering a question of Robin Hirsch in \cite{R} on complete representations for both relation and cylindric algebras 
using an example in the same paper. This example shows that in the characterization of countable completely representable algebras,
both relation and cylindric algebras, the condition of countability is necessary, it {\it cannot} be omitted.

\item Using an example by Andr\'eka et all, to show that unlike cylindric and polyadic equality algebras, 
atomic polyadic algebras, and Pinters substitution algebras, even without cylindrfiers, of 
dimension $2$ may  not be completely representable. However, the class of completely representable algebras 
is not so hard to characterize; it is finitely axiomatizable in first order logic.
This is contrary to a current belief for polyadic algebras, and is an answer to a result of Hodkinson for cylindifier free
Pinter's algebras.

\item Using the construction in Andr\'eka et all \cite{ANT}, 
showing that the omitting types theorem fails for finite first 
order definable extension of first order logic as defined by Tarski and Givant, and further pursued by others, like 
Maddux and Biro,  a result mentioned in the above cited paper without a proof.

\item Characterizing the class $\Nr_n\CA_{\omega}$ by games, and 
showing that the class $\Nr_n\CA_{\omega}$ is pseudo elementary, and its elementary closure is not finitely axiomatizable.

\item Characterizing the class of countable completely representable algebras of infinitely countable dimensions using weak representations 
(the question remains whether this class is elementary, 
the Hirsch Hodkinson example depending on a cardinality argument does not works when our units are weak spaces.)


\item Giving full proofs to three results mentioned in \cite{OTT} without proofs, referring to a pre-print, 
concerning omitting types in uncountable theories using finitely
many variables. This is the pre print, expanded, modified and polished containing proofs of these results and much more.
The results concerning omitting types depend on deep model-theoretic constructions of Shelah's.

\item We show that the class of weakly neat atom structures, as defined in the abstract, is not elementary for every dimension.

\item Unlike the cylindric case, we show that atomic polyadic algebras of infinite dimensions are completely representable.

\end{enumarab}

This paper also simplifies existing proofs in the literature, like the proof in \cite{ANT}, concerning complete representations of relation 
atom structures, having cylindric basis. Some classical results, like Monk's non-finitizability results for relation and cylindric algebras
are also re-proved.


\section{Omitting types}

In model theory we are typically encountered by questions not only of the form find a model  (prove consistency), but rather 
of the form find a model of a theory that 
does so and so. A typical example, in this regard, is finding a model 
omitting a given set of non-isolated types. In first order logic with $\omega$ variables, this problem is completely settled 
for the countable case. Another, which a variation on Vaught's conjecture is, {\it count} the number of such models.
Vaught's conjecture is regarded as one of  the most important open problem in model theory, if not the most.

In the first case, algebraically we find ourselves looking for non-trivial homomorphisms from locally finite cylindric algebras to set algebras 
that preserve infinitary meets (or types) carrying them to set theoretic intersection, and  by the Henkin-Orey classical 
omitting types theorem we can always find one if the meets considered are not isolated.
When we truncate the dimension to be finite, it is not always the case  algebra in $\RCA_n$ have this property, 
but indeed it is worthwhile characterizing those that does. This is one of the tasks carried out in this paper. 

In the second case, we  find ourselves counting special ultrafilters in locally finite cylindric algebras, that form a dense set of its 
Stone space (This will be discussed in some detail below).
Investigating, using the well developed machinery of algebraic logic, the number of types omitted in theories, 
and the number of pairwise non-isomorphic countable models
for a countable theory that omit a given set of types, 
this paper is interdisciplinary between model theory, set theory and algebraic logic.

There two typical types of investigations in set theory. Both of those will be adressed in this paper. 
The first type consists of theorems demonstrating the independence of mathematical statements. 
This type requires a thorough understanding of mathematics surrounding the 
source of problem in question, reducing the ambient mathematical constructions to combinatorial statements about sets, 
and finally using some method 
(primarily forcing) to show  that such combinatorial statements
are independent.
A second type involves delineating the edges of the independence proofs, giving proof in $ZFC$ of statements 
that on first sight would be suspected of being independent. 
Proofs of this kind is often extremely subtle and surprising; very similar 
statements are independent and it is hard to detect the underlying difference.

We present here results illustrating both types of investigations. 
We show that when we consider countable algebras, omitting $< {}^{\omega}2$ types turns 
out independent, whereas, if the types are maximal the statement is provable. 
Such results will be presented in the context of omitting types in $\L_n$ theories. We find this discrepancy interesting from the purely set theoretic point of 
view when maximality shifts us from the realm of independence to that of 
provability from $ZFC$.

Let us start with the classical omitting types theorem.
One direct consequence of the classical Henkin-Orey omitting types theorem states that if 
$T$ is a consistent theory in a countable language $\L$ and $\Gamma(x_1\ldots x_n)$ is 
realized in every model of $T$, then there is a formula $\phi\in \L$ such that
$\phi$ isolates $\Gamma$ in $T$. $\phi$ is called a witness of $\Gamma$. This follows directly from the contrapositive
of the following theorem, see also \cite{OTT}, \cite{OTT2}:

\begin{theorem}\label{infinite} Let $T$ be a countable consistent theory. Assume that 
$\kappa<covK$, where $covK$ is the least cardinal $\kappa$ such that the real line can be covered by
pairwise disjoint $\kappa$ nowhere-dense sets. Let $(\Gamma_i:i\in \kappa)$ be a set of non-principal types in $T$. 
Then there is a model countable $\M$ of $T$ that omits
all the $\Gamma_i$'s.
\end{theorem}
\begin{demo}{Proof} Standard Baire Category argument. Let $\A=\Fm_T$, then $\A\in \Lf_{\omega}$.
Let $X_i=\{\phi/T: \phi\in \Gamma_i\}$. Then we have $\prod X_i=0$ in $\A$. We have by \cite[1.11.6]{HMT1} that 
\begin{equation}\label{t1}
\begin{split} (\forall j<\alpha)(\forall x\in A)({\sf c}_jx=\sum_{i\in \alpha\smallsetminus \Delta x}
{\sf s}_i^jx.)
\end{split}
\end{equation}
Now let $V$ be the weak space $^{\omega}\omega^{(Id)}=\{s\in {}^{\omega}\omega: |\{i\in \omega: s_i\neq i\}|<\omega\}$.
For each $\tau\in V$ for each $i\in \kappa$, let
$$X_{i,\tau}=\{{\sf s}_{\tau}x: x\in X_i\}.$$
Here ${\sf s}_{\tau}$ 
is the unary operation as defined in  \cite[1.11.9]{HMT1}.
For each $\tau\in V,$ ${\sf s}_{\tau}$ is a complete
boolean endomorphism on $\B$ by \cite[1.11.12(iii)]{HMT1}. 
It thus follows that 
\begin{equation}\label{t2}\begin{split}
(\forall\tau\in V)(\forall  i\in \kappa)\prod{}^{\A}X_{i,\tau}=0
\end{split}
\end{equation}
Let $S$ be the Stone space of the Boolean part of $\A$, and for $x\in \A$, let $N_x$ 
denote the clopen set consisting of all
boolean ultrafilters that contain $x$.
Then from \ref{t1}, \ref{t2}, it follows that for $x\in \A,$ $j<\beta$, $i<\kappa$ and 
$\tau\in V$, the sets 
$$\bold G_{j,x}=N_{{\sf c}_jx}\setminus \bigcup_{i\notin \Delta x} N_{{\sf s}_i^jx}
\text { and } \bold H_{i,\tau}=\bigcap_{x\in X_i} N_{{\sf s}_{\bar{\tau}}x}$$
are closed nowhere dense sets in $S$.
Also each $\bold H_{i,\tau}$ is closed and nowhere 
dense.
Let $$\bold G=\bigcup_{j\in \beta}\bigcup_{x\in B}\bold G_{j,x}
\text { and }\bold H=\bigcup_{i\in \kappa}\bigcup_{\tau\in V}\bold H_{i,\tau.}$$
By properties of $covK$, it can be shown $\bold H$ is a countable collection of nowhere dense sets.
By the Baire Category theorem  for compact Hausdorff spaces, we get that $H(A)=S\sim \bold H\cup \bold G$ is dense in $S$.
Accordingly let $F$ be an ultrafilter in $N_a\cap X$.
By the very choice of $F$, it follows that $a\in F$ and  we have the following 
\begin{equation}
\begin{split}
(\forall j<\beta)(\forall x\in B)({\sf c}_jx\in F\implies
(\exists j\notin \Delta x){\sf s}_j^ix\in F.)
\end{split}
\end{equation}
and 
\begin{equation}
\begin{split}
(\forall i<\kappa)(\forall \tau\in V)(\exists x\in X_i){\sf s}_{\tau}x\notin F. 
\end{split}
\end{equation}
Next we form the canonical representation corresponding to $F$
in which satisfaction coincides with genericity. 
To handle equality. we define
$$E=\{(i,j)\in {}^2{\alpha}: {\sf d}_{ij}\in F\}.$$
$E$ is an equivalence relation on $\alpha$.   
$E$ is reflexive because ${\sf d}_{ii}=1$ and symmetric 
because ${\sf d}_{ij}={\sf d}_{ji}.$
$E$ is transitive because $F$ is a filter and for all $k,l,u<\alpha$, with $l\notin \{k,u\}$, 
we have 
$${\sf d}_{kl}\cdot {\sf d}_{lu}\leq {\sf c}_l({\sf d}_{kl}\cdot {\sf d}_{lu})={\sf d}_{ku}.$$
Let $M= \alpha/E$ and for $i\in \omega$, let $q(i)=i/E$. 
Let $W$ be the weak space $^{\alpha}M^{(q)}.$
For $h\in W,$ we write $h=\bar{\tau}$ if $\tau\in V$ is such that
$\tau(i)/E=h(i)$ for all $i\in \omega$. $\tau$ of course may
not be unique.
Define $f$ from $\B$ to the full weak set algebra with unit $W$ as follows:
$$f(x)=\{ \tau\in {}^{\omega}\omega:  {\sf s}_{\tau}x\in F\}, \text { for } x\in \A.$$ 
Then it can be checked that $f$
is a homomorphism 
such that $f(a)\neq 0$ and 
$\bigcap f(X_i)=\emptyset$ for all $i\in \kappa$, hence the desired conclusion.

\end{demo}

The above theorem can be proved using games. Indeed, the connection of games to the Baire category approach 
can be established using the famous Banach-Mazur theorem. An ultrafilter in the above proof is called a Henkn ultrfailter. 
Henkin ultrafilters give rise to representations (models) and carefully chosen one gives models that omit 
a given family of types

\section{Counting Henkin ultrafilers}

This is equivalent to countng models, and it can be used to count non isomorphic countable models for the locally finite
cylindric algebras representing a countable first order theory. This is a non-trivial topic in Model theory. This is Vaughts conjecture.

Let $\A$ be any Boolean algebra. The set of ultrafilters of $\A$ is denoted by $\mathcal{U}(\A)$. 
The Stone topology  makes $\mathcal{U}(\A)$ a compact Hausdorff space; we denote this space by $\A^*$. Recall that the Stone topology has as its basic open sets the sets $\{N_x:x\in A\}$, where
$$N_x=\{\mathcal F\in\mathcal{U}(\A):x\in\mathcal F\}.$$
It is easy to see that if $A$ is countable, then $\A^*$ is \emph{Polish}, (i.e., separable and completely metrizable).

Now, suppose $\A$ is a locally finite cylindric or quasi-polyadic $\omega$-dimensional algebra with a countable universe.
Note that if $T$ is a theory in a countable language with (without) equality, then $CA(T)$, (respectively $QPA(T)$), satisfies these requirements.
Let $$\mathcal{H}(\A)=\bigcap_{i<\omega,x\in A}(N_{-c_ix}\cup\bigcup_{j<\omega}N_{s^i_jx})$$ and, in the cylindric algebraic case, let

$$\mathcal{H}'(\A)=\mathcal{H}(\A)\cap\bigcap_{i\neq j\in\omega}N_{-d_{ij}}.$$
Note, for later use, that $\mathcal{H}(\A)$ and $\mathcal{H}'(\A)$ are $G_\delta$ subsets of $\A^*$, and are nonempty,  as a matter of face it is dense-- this latter fact can be seen,
for example, from Theorem \ref{th1} below -- and are therefore Polish spaces; (see \cite{Kechris}).
Assume $\mathcal F\in \mathcal{H}(\A).$ For any $x\in A$, define the function $\mathrm{rep}_{\mathcal F}$ to be $$\mathrm{rep}_{\mathcal F}(x)=\{\tau\in{}^\omega\omega:s^+_\tau x\in \mathcal F\}.$$
We have the following results due to G. S\'agi and D. Szir\'aki; (see \cite{Sagi}).
\begin{theorem}\label{th1}
If $\mathcal F\in \mathcal{H}'(\A)$, (respectively $\mathcal{H}(\A)$), then $\mathrm{rep}_{\mathcal F}$ is a homomorphism from $\A$ onto an element of $Lf_\omega\cap Cs_\omega^{reg}$, (respectively $LfQPA_\omega\cap Qs_\omega^{reg}$), with base $\omega.$ Conversely, if $h$ is a homomorphism from $\A$ onto an element of $Lf_\omega\cap Cs_\omega^{reg}$, (respectively $LfQPA_\omega\cap Qs_\omega^{reg}$), with base $\omega$, then there is a unique $\mathcal F\in \mathcal{H}'(\A)$, (respectively $\mathcal{H}(\A)$), such that $h=\mathrm{rep}_{\mathcal F}.$
\end{theorem}
\begin{theorem}
Let $T$ be a consistent first order theory in a countable language with (without) equality. Let $\mathcal{M}_0$ and $\mathcal{M}_1$ be two models of $T$ whose universe is $\omega$. Suppose $\mathcal{F}_0,\mathcal{F}_1\in\mathcal{H}'(CA(T))$, (respectively $\mathcal{H}(QPA(T))$), are such that $\mathrm{rep}_{\mathcal{F}_i}$ are homomorphisms from $CA(T)$, (respectively $QPA(T)$), onto $Cs(\mathcal{M}_i)$, (respectively $Qs(\mathcal{M}_i)$), for $i=0,1$. If 
$\rho:\omega\longrightarrow\omega$ is a bijection,
then the following are equivalent:
 \begin{enumerate}
 \item $\rho:\mathcal{M}_0\longrightarrow\mathcal{M}_1$ is an isomorphism.
 \item $\mathcal{F}_1=s^+_\rho \mathcal{F}_0=\{s^+_\rho x:x\in \mathcal{F}_0\}.$
\end{enumerate}
\end{theorem}

These last two theorems allow us to study models and count them via corresponding ultrafilters. This approach was initiated by S\'agi.
The main advantage of such an approach is that results proved
for locally finite cylindric algebras transfer {\it mutatis mutandis} to quasi-polyadic algebras (without diagonal elements).
So from the algebraic point of view we do the difficult task only once,
but from the model theoretic point of view we
obtain deep theorems for first order logic {\it without} equality, as well, which are more often than not, not obvious to prove without
the process of algebraisation. There are ceratin results that have an easy metalogical proof when we have equality.
but this proof does not work  in the absence of equality.
However, a purely algebraic proof works.

Now we define an equivalence relation on the Henkin utrafilters of a locally finite cylindric algebra
and  show that it is Borel subset of the product of the Stone space.
This implies that it satisfies the Glimm-Effros dichotomy, and so has either countably many or else continuum many equivalence classes.
The equivalence relation we introduce  corresponds to a non-trivial equivalence relation between models which is weaker than isomorphism.
\begin{definition}[Notation]
Let $\mathcal{F}$ be an ultrafilter of a locally finite (cylindric or quasi-polyadic) algebra $\A$.
For $a\in A$ define $$Sat_\mathcal{F}(a)=\{t|_{\Delta a}: t\in {}^{\omega}\omega,\;s^+_t a\in\mathcal{F}\}.$$
\end{definition}

Throughout, $\A$ is countable. We define an equivalence relation $\mathcal{E}$ 
on the space $\mathcal{H}'(\A)$ (or $\mathcal{H}(\A))$  that turns out to be Borel.

\begin{definition}
Let $\mathcal{E}$ be the following equivalence relation on $\mathcal{H}'(\A)$ (or $\mathcal{H}(\A))$ :\small
$$\mathcal{E}=\{(\mathcal{F}_0,\mathcal{F}_1): (\forall a\in A) (|Sat_{\mathcal{F}_0}(a)|=|Sat_{\mathcal{F}_1}(a)|)\}.$$
\end{definition}

We say that $\mathcal{F}_0,\mathcal{F}_1\in\mathcal{H}'(\A)$(or $\mathcal{H}(\A))$ are 
\emph{distinguishable } if $(\mathcal{F}_0,\mathcal{F}_1)\notin\mathcal{E}.$ 
We also say that two models of a theory $T$ are distinguishable if their corresponding 
ultrafilters in $\mathcal{H}'(CA(T))$(or $\mathcal{H}(QPA(T)))$ are distinguishable. 
That is, two models are distinguishable if they disagree in the number of realizations they have for some formula.
$\mathcal{E}$ is Borel. 
\begin{proof}
Here is an argument. Suppose we have a language $L$ with equality.
First note that if  $L^*=L_0\cup L_1$ where $L_0$ and $L_1$ are disjoint copies of $L$,
then $X_{L^*}\cong X_{L_0}\times X_{L_1}$ (where the spaces $X_L$'s are defined as in \cite{BeckerKechris} page 22).

For each formula $\varphi,$ let $\varphi^*$ be the
sentence $\bigwedge_{n\in\omega}(\exists^n \bar{x})\varphi_0(\bar{x})\leftrightarrow (\exists^n \bar{x})\varphi_1(\bar{x})$
where $\varphi_0,$ $\varphi_1$ are the copies of $\varphi$ in $L_0,L_1$ respectively, and $\exists^n$ is a shorthand for
``there exists at least $n$ tuples such that ..."

It is then immediate that two models $M_0,M_1$ of $L$ are not distinguishable iff the model $M$ of $L^*$ such
that $M|_{L_0}=M_0$ and $M|_{L_1}=M_1$ satisfies $\bigwedge_{\varphi\in L}\varphi^*.$
This means that our equivalence relation between models corresponds to the subset of $X_{L^*}$ of models of
the formula $\bigwedge_{\varphi\in L}\varphi^*.$ Such a subset is Borel by Theorem 16.8 in \cite{Kechris}.
\end{proof}

Vaught's conjecture has been confirmed when we restrict the action on certain subgroups of $G$.
But in this  case there might be isomorphic models that the group $G$ does not 'see' (the isomorphism witnessing this
can be outside $G$) so the equivalence relation is drastically different.

\begin{theorem} Let $G\subseteq S_{\infty}$ be a cli group. Then $|{\cal H}(\A)/E_G|\leq \omega$ or $|{\cal H}(\A)/E_G|=2^{\omega}$
\end{theorem}
\begin{demo}{Proof}  It is known that the number of orbits of $E_G$  satisfies the so-called Glimm Effros Dichotomy.
By known results in the literature on the topological version of
Vaught's conjecture, we have ${\cal H}(\A)/E_G$
is either at most countable or ${\cal H}(\A)/E_G$ contains continuum many non equivalent elements
(i.e non-isomorphic models).
\end{demo}
It is known that the number of orbits of $E=E_{S_{\infty}}$ {\it does not} satisfy the Glimm Effros Dichotomy.
We note that cli groups cover all natural extensions of abelian groups, like nilpotent and solvable groups.
Now we give a topological condition that implies Vaught's conjecture.
Let everything be as above with $G$ denoting a Polish subgroup of $S_{\infty}$.
Give ${\cal H}(\A)/E_G$ the qoutient topology and let $\pi: {\cal H}(\A)\to {\cal H}(\A)/E_G$ be the projection map.
$\pi$ of course depends on $G$, we somtimes denote it by $\pi_G$ to emphasize the dependence.


\begin{lemma} $\pi$ is open.
\end{lemma}
\begin{proof}
To show that $\pi$ is open it is enough to show for arbitrary $a\in \A$ that $\pi^{-1}(\pi(N_a))$ is open. For,
\begin{align*}
\pi^{-1}(\pi(N_a))&=\{F\in\A^*:(\exists F'\in N_a)(F,F')\in E\}\\
&=\{F\in\A^*:(\exists F'\in N_a)(\exists \rho\in G)s^+_\rho F'=F\}\\
&=\{F\in\A^*:(\exists F'\in N_a)(\exists \rho\in G)F'=s^+_{\rho^{-1} }F\}\\
&=\{F\in\A^*:(\exists \rho\in G)s^+_{\rho^{-1} }F\in N_a\}\\
&=\{F\in\A^*:(\exists \rho\in G)a\in s^+_{\rho^{-1} }F\}\\
&=\{F\in\A^*:(\exists \rho\in G)s^+_{\rho }a\in F\}\\
&=\bigcup_{\rho\in G}N_{s^+_\rho a}
\end{align*}
\end{proof}

\begin{theorem} If $\pi$ is closed, then Vaught's conjecture holds.
\end{theorem}
\begin{demo}{Proof} We have ${\cal H}(\A)$ is Borel subset of $\A^*$, the Stone space of $\A$ and ${\cal H}(\A)/E_G$ is a continuous image of
${\cal H}(\A)$.
Because $\pi$ is open, $H(\A)/E$  is second countable.
Now, since $H(\A)$ is metrizable and second countable, it is normal.
But $\pi$ is closed, and so $H(\A)/E$ is also normal, hence regular. Thus ${\cal H}(\A)/E_G$
can also be embedded as an open set in $\mathbb{R}^{\omega},$ hence it is Polish.
If ${\cal H}(\A)/E_G$ is uncountable, then being the continuous image under a map between two Polish spaces of a Borel set, it is analytic.
Then it has the power of the continuum.
\end{demo}

\subsection{ Number of models omitting a given family of types}

The way we counted the ultrafilters (corresponding to distinguishable models)
above gives a completely analogous  result when we count ultrafilters corresponding  to models {\it omitting} a countable set of non-isolated types.

Given a countable locally finite algebra $\A$, a non-zero $a\in A$ and a non-principal type
$X\subseteq \Nr_n\A$, so that $\prod X=0$, one constructs a model omitting $X$, by finding a Henkin ultrafilter preserving
the following set of infinitary joins and meets where $x\in A$, $i,j\in \omega$ and $\tau$ is a finite transformation:
$c_ix=\sum s_j^ix,$
and $\prod s_{\tau}X=0.$
Working in the Stone space, one finds an ultrafilter in $N_a$ outside the
nowhere dense sets
$N_{i,x}=S\sim \bigcup N_{s_j^i}$ and
$H_{\tau}=\bigcap_{x\in X} N_{\tau}x.$
Now suppose we want to count the number of distinguishable models omitting a
family $\Gamma=\{\Gamma_i:i<\lambda\}$ ($\lambda<covK)$
of non-isolated types of $T$.

Then $$\mathbb{H}=
\mathcal{H}(CA(T))\mbox{(or }\mathcal{H}'(QPA(T)))\cap\sim  \bigcup_{i\in\lambda,\tau\in W}\bigcap_{\varphi\in \Gamma_i}N_{s_\tau (\varphi/\equiv_T)}$$
(where $W=\{\tau\in{}^\omega\omega:|i:\tau(i)\neq i|<\omega\}$) is clearly (by the above discussion)
the space of ultrafilters corresponding to models of $T$ omitting $\Gamma.$

But then by properties of $covK$ union  $\bigcup_{i\in\lambda}$
can be reduced to a countable union.
We then have $\mathbb{H}$ a $G_\delta$ subset of a Polish space. So $\mathbb{H}$ is
Polish and moreover, $\mathcal{E}'=\mathcal{E}\cap (\mathbb{H}\times \mathbb{H})$
is a Borel equivalence relation on $\mathbb{H}.$ It follows then that the number of distinguishable models omitting $\Gamma$
is either countable or else $2^\omega.$
We readily obtain:
\begin{corollary}
Let $T$ be a first order theory  in a countable language (with or without equality).
If  $T$ has an uncountable set of countable models that omit $< covK$ many non principal types that
are pairwise distinguishable, then actually it has such a set of size $2^{\aleph_0}$.
\end{corollary}
Using the same reasoning as above conjuncted with Morleys theorem, we get
\begin{theorem}
The number of countable models of a countable theory that omits $< covK$ many types is either $\leq \omega$ or
$\omega_1$ or $^{\omega}2$.
\end{theorem}
We can also formulate a version of Vaughts conjecture for s called rich languages.
\begin{theorem}\label{2} Let $\A\in Dc_{\alpha}$ be countable simple and finitely generated.
Then the number of non-base isomorphic representations of $\A$ is $2^{\omega}$.
\end{theorem}
\begin{proof} Let $V={}^{\alpha}\alpha^{(Id)}$ and let $\A$ be as in the hypothesis. Then $\A$ cannot be atomic \cite{HMT1} corollary 2.3.33,
least hereditary atomic. Then it has $2^{\omega}$ ultrafilters.
For an ultrafilter $F$, let $h_F(a)=\{\tau \in V: s_{\tau}a\in F\}$, $a\in \A$.
Then $h_F\neq 0$, indeed $Id\in h_F(a)$ for any $a\in \A$, hence $h_F$ is an injection, by simplicity of $\A$.
Now $h_F:\A\to \wp(V)$; all the $h_F$'s have the same target algebra.
Then  $h_F(\A)$ is base isomorphic to $h_G(\A)$ iff there exists a finite bijection $\sigma\in V$ such that
$s_{\sigma}F=G$.
Define the equivalence relation $\sim $ on the set of ultrafilters by $F\sim G$, if there exists a finite permutation $\sigma$
such that $F=s_{\sigma}G$. Then any equivalence class is countable, and so we have $^{\omega}2$ many orbits, which correspond to
the non base isomorphic representations of $\A$.
\end{proof}
\begin{theorem}
Let $T$ be a countable theory in a rich language, with only finitely many relation symbols,
and $\Gamma =\{\Gamma_i: i\in covK\}$ be non isolated types.
Then $T$ has $2^{\omega}$ weak models that omit $\Gamma$.
\end{theorem}

Notice that this theorem substantially generalized the main theorem in \ref{OTT}, 
the latter shows that there exists at least one model omiytting non principal types, this theorem says that there are continuum many of 
them.

\subsection{Omitting types for finite variable fragments}

For finite variable fragments $\L_n$ for $n\geq 3$, the situation turns out to be drastically different.
But first a definition.
\begin{definition}
Assume that $T\subseteq \L_n$. We say that $T$ is $n$ complete iff for all sentences $\phi\in \L_n$ 
we have either $T\models \phi$ or $T\models \neg \phi$. We say that $T$
is $n$ atomic iff for all $\phi\in \L_n$, there is $\psi\in \L_n$ such that $T\models \psi\to \phi$ and for all $\eta\in \L_n$ either $T\models \psi\to \eta$
or $T\models \psi\to \neg \eta.$
\end{definition}

The next theorem \ref{finite} is proved using algebraic logic in \cite{ANT}, 
using combinatorial techniques depending on Ramsey's theorem. 

\begin{theorem}\label{finite} Assume that $\L$ is a countable first order language containing a binary relation symbol. For $n>2$ and 
$k\geq 0$,  there are a consistent $n$ complete
and $n$ atomic theory $T$ using only $n$ variables, and a set $\Gamma(x_1)$ using only $3$ variables
(and only one free variable) such that $\Gamma$ is realized in all models of $T$ but each $T$-witness for $T$ uses
more that $n+k$ variables
\end{theorem}
\begin{demo}{Proof} \cite{ANT} for a full proof and \cite{OTT} for an easier sketch. 
Below we give a simplified version of the proof, but as an expense we obtain a somewaht 
weaker result.
\end{demo}

Algebraisations of finite variable fragments of first order logic with $n$ variables is 
obtained from locally finite algebras by truncating the dimensions at $n$.
Expressed, formally this corresponds to the operation of forming $n$ neat reducts.

\begin{definition}\cite{HMT1} Let $\A\in \CA_{\beta}$ and $\alpha<\beta$, then the {\it $\alpha$ neat reduct of $\A$} is the algebra obtained 
from $\A$ by discarding operations in $\beta\sim \alpha$ and restricting the remaining operations 
to the set consisting only of $\alpha$ dimensional elements. An element is $\alpha$ dimensional if its dimension set, $\Delta x=\{i\in \beta: c_ix\neq x\}$ 
is contained in $\alpha$. Such an algebra is denoted by $\Nr_{\alpha}\A$.
\end{definition}
For a class $K\subseteq \CA_{\beta}$, $\Nr_{\alpha}K=\{\Nr_{\alpha}\A: A\in K\}$. It is easy to verify that $\Nr_{\alpha}K\subseteq \CA_{\alpha}$.

A class of particular importance, is the class $S\Nr_{\alpha}\CA_{\alpha+\omega}$ where $\alpha$ is an arbitrary ordinal; 
here $S$ stands for the operation of forming subalgebras. 
This class turns out to be a variety which coincides with
the class of {\it representable} algebras of dimension $\alpha$.

Another class that is of significance is the class $S_c\Nr_{\alpha}CA_{\alpha+\omega}$. Here $S_c$ 
is the operation of forming {\it complete} subalgebras. 
(A Boolean algebra $\A$ is a complete subalgebra of $\B$, if for all $X\subseteq \A$, whenever $\sum X=1$ in $A$, then 
$\sum X=1$ in $\B)$. This class is important because a countable cylindric algebra of dimension $n$ is completey representable if and only if
$\A\in S_c\Nr_n\CA_{\omega}$, for any $\alpha\geq \omega$ and $\A$ is atomic. 
This characterization even works for countable dimensions, 
by modifying the notion of complete representation' relativizing the unit to so-called weak units.

And it turns out for finite  variable fragments, that for a theory $T$ to omit types, 
whether countably or uncountably many, a sufficient condition is that 
the cylindric algebra of formulas $\Fm_T$ is in the class $S_c\Nr_n\CA_{\omega}$. 
Furthermore, the condition of  {\it complete} subalgebras, cannot be omitted.

\subsection{Omitting types for finite first order definable extensions of $\L_n$}

A question raised by Jonsson is whether non finite axiomatizability-results for the class $\RCA_n$ $n\geq 3$, can be circumvented by 
adding finitely many {\it first order definable} operations. This has to do with seeking a completeness theorem for $\L_n$.
The was answered negatively by Biro. Now we ask the same question for omitting types, 
and we also get a negative answer. But first we recall what we mean by first order definable.

\begin{definition} 
Let $\Lambda$ be a first order language with countably many 
relation symbols, $R_0, \ldots R_i,\ldots : i\in \omega$ 
each of arity $n$. 
Let $\Cs_{n,t}$ denote the following class of similarity type $t$:
\begin{enumroman}
\item $t$ is an expansion of $t_{\CA_n}.$
\item  $S\Rd_{ca}Cs_{n,t}=\Cs_n.$ In particular, every algebra in $\Cs_{n,t}$ is a boolean 
field of sets with unit $^nU$ say, 
that is closed under cylindrifications and contains diagonal elements.   
\item For any $m$-ary operation $f$ of $t$, there exists a first order formula
$\phi$ with free variables among $\{x_0,\ldots, x_n\}$ 
and having exactly $m,$ $n$-ary relation symbols
$R_0, \ldots R_{m-1}$ such that,
for any set algebra ${\cal A}\in Cs_{n,t}$ 
with base $U$, and $X_0, \ldots X_{m-1}\in {\cal A}$,
we have:
$$\langle a_0,\ldots a_{n-1}\rangle\in f(X_0,\ldots X_{m-1})$$
if and only if  
$${\cal M}=\langle U, X_0,\ldots X_{n-1}\rangle\models \phi[a_0,\ldots a_{n-1}].$$
Here $\cal M$ is the first order structure in which for each $i\leq m$, 
$R_i$ is interpreted as $X_i,$ and $\models$ is the usual satisfiability relation.
Note that cylindrifications and diagonal elements are so definable.
Indeed for $i,j<n$,  $\exists x_iR_0(x_0\ldots x_{n-1})$ 
defines $C_i$ and $x_i=x_j$ defines $D_{ij}.$ 
\item With $f$ and $\phi$ as above,   
$f$ is said to be a first order definable operation with $\phi$ defining $f$, 
or simply a first order definable 
operation, and $\Cs_{n,t}$ is said to be a first order definable
expansion of $\Cs_n.$
\item $\RCA_{n,t}$ denotes the class $SP\Cs_{n,t}$, i.e. the class of all subdirect products 
of algebras
in $\Cs_{n,t}.$ We also refer to 
$\RCA_{n,t}$ as a first order definable expansion of $\RCA_n.$
\end{enumroman}
\end{definition}

\begin{theorem} Let $t$ be a finite expansion of the $CA$ type.
Then there are atomic algebras in $\RCA_{n,t}$ that are not completely representable
\end{theorem}
\begin{proof} Let $k\in \omega$ such that all first order definable operations are built up using at most $n+k$ varialbes. 
This $k$ exists, because we have only finitely many of those operations. Let 
$\A\in RCA_n\cap \Nr_nCA_{n+k}$ be an atomic cylindric algebra that
is not completely representable. 
Then $\A$ is closed under all the first order definable operations. But the cylindric reduct of $\A$, is not completely representable, then
$\A$ itself is not.
\end{proof}
 
$\RCA_{n,t}$ is a variety, this can be proved exactly like the $\RCA_n$ case, and in fact it is a 
discriminator variety with discrimintor term $c_0\ldots c_{n-1}$.
(So it is enough to show that it is closed under ultraproducts.)
A result of Biro says that is not finitely axiomatizable. 

The logic corresponding to $\RCA_{n,t}$ has $n$ variables, and it has the usual first order connectives;
(here we view cylindrifiers as unary connectives) and it has one connective for each first order definable operation.
In atomic formulas the variables only appear in their natural order; so that they are {\it restricted}, in the sense of \cite{HMT1} sec 4.3.

Semantics are defined as follows: For simplicity we assume that we have an at most  countable collection of $n$-ary relation symbols
$\{R_i: i\in \omega\}$. 
(We will be considering countable languages anyway, to violate the omitting types theorem).

The inductive definition for the first order (usual) connectives is the usual.
Now given a model $\M=(M, R_i)_{i\in I}$ and a formula $\psi$ of $L_n^+$; with a corresponding connective 
$f_{\phi}$ which we assume is unary to simplify 
matters, and $s\in {}^nM$, then for any formula $\psi$: 
$$\M\models f_{\phi}(\psi)[s]\Longleftrightarrow (\M, \phi^{\M})\models \psi[s].$$

$_rL_n$ is the logic correponding to $\CA_n$, which has only restricted formulas, 
and $L_n$ is that correponding to $\RPEA_n$; the latter of course is a first order extension of the former 
because the substitutions corresponding to transpositions
are first order definable.

\begin{corollary}  
No first order finite extension of $L_n$  enjoys an omitting types theorem. 
\end{corollary}

\begin{proof} Let $\A$ be as above, and let $\Fm_T$ be the corresponding Tarski-Lindenbaum algebra. 
We assume that $\A=\Fm_T$ (not just isomorphic) Let $X=\At\A$ and let $\Gamma=\{\phi: \phi/\equiv\in X\}.$
Then $\Gamma$ cannot be omitted.
\end{proof}

\subsection{Omitting uncountably many types}

Here we give an independence result, concerning omitting a number of non-isolated types.
The we delineating the edges of the independence proofs by considering {\it complete} types,
giving a proof in $ZFC$ of the resulting statement 
that on first sight would be suspected of being independent. 
Proofs of this kind is often extremely subtle and surprising.

Inthe proof of the second item in our next theorem, we will be constructing countable models that omit
possibly uncountable types in $\L_n$, $n\in \omega$. This is theorem is formulated in \cite{OTT} without proof.
Now we provide one:
\begin{definition} 
\begin{enumroman}
\item Let $\kappa$ be a cardinal. Let $OTT(\kappa)$ be the following statement. 
If $\A\in S_c\Nr_n\CA_{\omega}$ is countable and for $i\in \kappa$, $X_i\subseteq A$ are such that $\prod X_i=0$, then for all $a\neq 0$, there exists
a set algebra $\C$ with countable base, $f:\A\to C$ such that $f(a)\neq 0$ and for all $i\in \kappa$, $\bigcap_{x\in X_i} f(x)=0.$
\item Let $OTT$ be the statement that 
$$(\forall k<{}^{\omega}2)(OTT(\kappa))$$
\item Let $OTT_m(\kappa)$ be the statement obtained from $OTT(\kappa)$ by replacing $X_i$ with ``nonprincipal ultrafilter $F_i$"
and $OTT_m$ be the statement 
$$(\forall k<{}^{\omega}2)(OTT_m(\kappa))$$
\end{enumroman}
\end{definition}
Now we formulate and prove our main theorem (this theorem is stated in \cite{OTT}):
\begin{theorem}\label{OT} 
\begin{enumroman}
\item $OTT$ is independent from $ZFC+\neg CH$. 
\item  $OTT_m$ is provable in $ZFC$.
\end{enumroman}
\end{theorem}
\begin{proof}
\begin{enumroman} 
\item Let $covK$ be the least cardinal $\kappa$ such that the real line can be covered by $\kappa$ many closed disjoint 
nowhere dense sets. It is known that 
$\omega<covK\leq 2^{\omega}$  By theorem \ref{infinite} $< covK$ types can be omitted.
But Martin's axiom implies that $covK =2^{\omega}$. 
So we have consistency. 
The independence is proved using standard iterated forcing to show that it is consistent that $covK$ could be literally anything greater 
than  $\omega$ and $\leq {}^{\omega}2$, and then show that $OTT(covK)$ is false, see example \ref{cov}. 

\item  The idea is that one can build several models such that they overlap only on isolated types.
One can build {\it two} models so that every maximal type which is realized in both is isolated. Using the jargon of Robinson's finite forcing
implemented via games, the idea
is that one distributes this job of building the two models among experts, each has a role to play, and that all have 
winning strategies. There is no difficulty in stretching the above idea to make the experts build three, four or any finite number of models 
which overlap only at principal types. 
With a pinch of diagonalisation we can extend the number to $\omega$. 

To push it still further to $^{\omega}2$ needs an entirely new idea (due to Shelah), which we will implement.

Algebraically, we first construct two representations of $\B$ such that if $F$ is an ultrafilter in $B$ that is realized in both representations, then $F$ is
necessarily principal, that is $\prod F$ is an atom generating $F$, then we sketch the idea of how to obtain $^{\omega}2$ many.
We construct two ultrafilters $T$ and $S$ of $\B$ such that (*)
$\forall \tau_1, \tau_2\in {}^{\omega}\omega( G_1=\{a\in \B: {\sf s}_{\tau_1}a\in T\},
G_2=\{a\in \B: s_{\tau_1}a\in S\})\\\implies G_1\neq G_2 \text { or $G_1$ is principal.}$
Note that $G_1$ and $G_2$ are indeed ultrafilters.
We construct $S$ and $T$ as a union of a chain. We carry out various tasks as we build the chains.
The tasks are as in (*), as well as

(**) for all $a\in A$, if $c_ka\in T$, then $s_l^kx\in T$ for $l\notin \Delta x$.

(***) for all $a\in A$ either $a\in T$ or $-a\in T$, and same for $S$.

We let $S_0=T_0=\{1\}$.
There are countably many tasks. Metaphorically we hire countably many experts and give them one task each.
We partition $\omega$ into infinitely many sets and we assign one of these tasks to each expert.
When $T_{i-1}$ and $S_{i-1}$ have been chosen and $i$ is in the set assigned to some expert $E$, then $E$ will construct
$T_i$ and $S_i$.

Let us start with  task $(**)$. The expert waits until she is given a set $T_{i-1}$ which contains $c_ka$ for some $k<\omega$.
 Every time this happen she look for a {\it witness} $l$ which is outside elements in $T_{i-1}$; 
this is possible since the latter is finite, then she sets $T_i=T_{i-1}\cup \{s_k^la\}$. 
Otherwise, she does nothing. This strategy works because her subset
of $\omega$ is infinite, hence contains arbitrarily large numbers. Same for $S_i$.

Now consider the expert who handles task (**). Let $X$ be her subset of $\omega$. Let her list as $(a_i: i\in X)$ all elements of $X$.
When $T_{i-1}$ has been chosen with $i\in X$, she should consider whether $T_{i-1}\cup \{a_i\}$ is consistent. If it is she puts
$T_i=T_{i-1}\cup \{a_i\}$. If not she puts $T_i=T_{i-1}\cup \{-a_i\}$. Same for $S_i$.

Now finally consider the crucial tasks in (*). Suppose that $X$ contains $i,$ and $S_{i-1}$ and $T_{i-1}$ have been chosen.
Let $e=\bigwedge S_{i-1}$ and $f=\bigwedge T_{i-1}$. We have two cases.
If $e$ is an atom in $B$ then the ultrafilter $F$ containing $e$ is principal so our expert can put $S_i=S_{i-1}$ and $T_i=T_{i-1}$.
If not, then let $F_1$ , $F_2$ be distinct ultrafilters containing $e$. Let $G$ be an ultrafilter containing $e$, and assume that $F_1$ 
is different from $G$.
Let $\theta$ be in $F_1-G$. Then put $S_i=S_{i-1}\cup \{\theta\}$ and $T_i=T_{i-1}\cup \{-\theta\}.$
It is not hard to check that the canonical models, defined the usual way,  corresponding to $S$ and $T$ are as required.

To extend the idea, we allow experts at any stage to introduce a new chain of theories which is a duplicate copy of one of the chains being
constructed. The construction takes the form of a tree where each branch separately will give a chain of conditions.
By splitting the tree often enough the experts can guarantee that there are continuum many branches and hence continuum many
representations. This is a well know method, in model theory, when one gets $^{\omega}2$ many models for the price of two.
There is one expert whose job is to make sure that this property 
is enforcable for each pair of branches.
But she can do this task, because at each step the number of branches is still finite.

\item Assume not. Let ${\bold F}$ be the given set of non principal ultrafilters. Then for all $i<{}^{\omega}2$,
there exists $F$ such that $F$ is realized in $\B_i$. Let $\psi:{}^{\omega}2\to \wp(\bold F)$, be defined by
$\psi(i)=\{F: F \text { is realized in  }\B_i\}$.  Then for all $i<{}^{\omega}2$, $\psi(i)\neq \emptyset$.
Furthermore, for $i\neq j$, $\psi(i)\cap \psi(j)=\emptyset,$ for if $F\in \psi(i)\cap \psi(j)$ then it will be realized in
$\B_i$ and $\B_j$, and so it will be principal.  This implies that $|\bold F|={}^{\omega}2$ which is impossible.

\end{enumroman}
\end{proof}

\begin{example}\label{cov} Let $n>1$. To show that $OTT(cov K)$ could be false, we adapt 
an example in \cite{CF} p.242. Another example can be found in \cite{N} but also needs a bit of a modification, 
since we have only finitely many 
variables. Fix $n\in \omega$. There the example is constructed for $L_{\omega,\omega}$ to adapt 
to $L_n$ some care is required. 
Let $T$ be a theory such that for this given $n$, in $S^n(T)$, the stone space of $n$ types,  the isolated points are not dense.
(In \cite{CF}, a theory $T$ is chosen which does not have a prime model. This implies that there is an $n$ such that the isolated types in 
$S_n(T)$ are not dense; in here we need a fixed $n$, so not any theory without a 
prime model will do). 
It is easy to construct such theories, for any fixed $n$.(For example the theory of random graphs the isolated types are not dense for any $n$). 

Let $X$ be the space $S^0(T)$ of all complete $0$ types which are consistent with $T$. For an ordinal $\alpha$, let $X^{(\alpha)}$ 
be the $\alpha$-iterated Cantor-Bendixon derivative of $X$. Recall that for ordinal numbers $\alpha$ the $\alpha$ 
Cantor-Bendixon derivative of a topological space is defined by 
transfinite induction 
\begin{itemize}
\item $X^0=X$
\item $X^{\alpha+1}=[X^{\alpha}]'$
\item $X^{\beta}=\bigcap_{\alpha\in \beta} X_{\alpha}$. 
\end{itemize}
The transfinite sequence of Cantor-Bendixson derivatives of $X$ must eventually stop. 
The smallest ordinal $\alpha$ such that $X^{\alpha+1}=X^{\alpha}$
is called the Cantor-Bendixon  rank of $X$.The language is countable, there is some $\alpha<\omega_1$ 
such that $X^{(\alpha)}=X^{(\alpha+1)}$
and $X\setminus X^{\alpha}$ is countable. $X^{\alpha}$ is a perfect set and 
therefore it is homeomorphic to the Cantor space $^{\omega}2$ or it is empty.
We associate a set $P_{\infty}$ of $\leq covK$ many types with $X^{\alpha}$. Assume that $X^{\alpha}$ is non-empty, since it is a 
closed set in $X$, 
there is some extension $T_{\infty}$ of $T$
such that in $X$ 
$$X^{\alpha}=\bigcap_{\sigma\in T^{\infty}}[\sigma].$$
Hence the space $S^0(T_{\infty})$ is homeomorphic to $X^{\alpha}$ and to $^{\omega}2$. Then there are $Y_{\beta} (\beta<covK)$ 
closed nowhere dense sets in $S^0(T_{\infty})$
such that
$$S^0(T_{\infty})=\bigcup_{\beta<covK}Y_{\beta}.$$type $p_{\beta}$ such that in $S^{0}(T_{\infty})$
$$Y_{\beta}=\bigcap_{\sigma\in p_{\beta}}[\sigma].$$
As $Y_{\beta}$ is nowhere dense $p_{\beta}$ is non principal in $T_{\infty}$. Assuming, without loss, that $T_{\infty}\subseteq p_{\beta}$ we get that 
$p_{\beta}$ is non 
principal in $T$.
Set
$$P_{\infty}=\{p_{\beta}:\beta<covK\}.$$
Let us consider the $0$ types in $X\setminus X^{\alpha}$. These are complete consistent extensions of $T$.
For every $T'\in X\setminus X^{\alpha}$ we shall define a set $P_{T'}$ of $\leq covK$ many $n$ types that are not omitted in $T'$.
If $T'$ is not a finite extension of $T$, set $P_{T'}=\{T'\}$. Otherwise, in $S^n(T')$ the isolated types are not dense. 
Hence there is some non-empty $Y\subseteq
S^n(T')$ clopen and perfect. Now we can cover $Y$ with a family of $covK$ many closed nowhere dense sets of $n$ types. 
Since $Y$ is clo-open in $S^{n}(T')$, these sets are closed nowhere dense sets in $S^{n}(T')$, 
so we obtain a family of $covK$ many non principal $n$ types that cannot be omitted. We may assume that 
$T'\subseteq p$ for every $p\in P_T'$ and therefore every type in $P_T'$ is non principal in $T$.
Define
$$P=P_{\infty}\cup\bigcup\{P_{T'}: T'\in X\setminus X^{\alpha}\}.$$  
Now $P$ is a family of non-principal types $|P|=covK$ that cannot be omitted.

Let $\A=\Fm/T$ and for $p\in P$ let $X_p=\{\phi/T:\phi\in p\}$. Then $X_p\subseteq \Nr_n\A$, and $\prod X_p=0$.
However for any $0\neq a$, there is no set algebra $\C$ with countable base 
$M$ and $g:\A\to \C$ such that $g(a)\neq 0$ and $\bigcap_{x\in X_i}f(x)=\emptyset$.
But in principle, if we take the $n$ neat reduct, representations preserving meets can exist. We exclude this possibility by showing that 
such representations necessarily lift to 
all $\omega$ dimensions. Let $\B=\Nr_n\A$. Let $a\neq 0$. Assume, seeking a contradiction, that there exists 
$f:\B\to \D'$ such that $f(a)\neq 0$ and $\bigcap_{x\in X_i} f(x)=\emptyset$. We can assume that
$B$ generates $\A$ and that $\D'=\Nr_n\D$ where $\D\in \Lf_{\omega}$. Let $g=\Sg^{\A\times \D}f$. We will show  
that $g$ is a one to one function with domain $\A$ that preserves the $X_i$'s which is impossible (Note that by definition $g$ is a homomorphism). 
We have
$$Dom g=Dom\Sg^{\A\times \D}f=\Sg^{\A}Domf=\Sg^{\A}\Nr_{n}\A=\A.$$
By symmetry it is enough to show that $g$ is a function.  We first prove the following (*)
 $$ \text { If } (a,b)\in g\text { and }  {\sf c}_k(a,b)=(a,b)\text { for all } k\in \omega\sim n, \text { then } f(a)=b.$$
Indeed,
$$(a,b)\in \Nr_{n}\Sg^{\A\times \D}f=\Sg^{\Nr_{n}(\A\times \D)}f=\Sg^{\Nr_{n}\A\times \Nr_{n}\D}f=f.$$
Here we are using that $\A\times \D\in \Lf_{\omega}$, so that  $\Nr_{n}\Sg^{\A\times \D}f=\Sg^{\Nr_{n}(\A\times \D)}f.$
Now suppose that $(x,y), (x,z)\in g$.
Let $k\in \omega\sim n.$ Let $\Delta$ denote symmetric difference. Then
$$(0, {\sf c}_k(y\Delta z))=({\sf c}_k0, {\sf c}_k(y\Delta z))={\sf c}_k(0,y\Delta z)={\sf c}_k((x,y)\Delta(x,z))\in g.$$

Also,
$${\sf c}_k(0, {\sf c}_k(y\Delta z))=(0,{\sf c}_k(y\Delta z)).$$ 
Thus by (*) we have  $$f(0)={\sf c}_k(y\Delta z) \text { for any } k\in \omega\sim n.$$
Hence ${\sf c}_k(y\Delta z)=0$ and so $y=z$.
We conclude that there exists a countable $\B\in \Nr_n\CA_{\omega}$ and $(X_i:i<covK)$ such that $\prod X_i=0$ but there is no representation that
preserves the $X_i$'s. In more detail. Give any $a\in \B$, if $a$ is non zero, $\C$ is a set algebra with countable base 
and $f:\B\to \C$ is a homomorphism such 
that $f(a)\neq 0$, then there exists $i<covK$, such that $\bigcap_{x\in X_i} f(x)\neq \emptyset.$
Therefore $OTT$ is false in a model of $ZFC+\neg CH$.
\footnote{Another model in which $covK=\omega_2$ and $^{\omega}2=\omega_3$ is due to Bukovski \cite{Miller} 
who does it by starting with a model of $2^{\omega}=\omega_3$, and then doing an $\omega_2$ iteration.
At each step of the iteration he does an $\omega_3$ iteration making $MA$ true. 
Alternatively we could start with a model of $2^{\omega}=\omega_3$ and then do an $\omega_2$ iteration with $\bold D$ 
where $\bold D=\{(n,f): n<\omega, f\in {}^{\omega}\omega\}$
the order for forcing an eventually dominant real.}
\end{example}


Now we omit uncountably many types in uncountable languages.
We use the following result of Shelah's. But first a piece of notation:
For a model $M$, and $\bar{a}\in {}^{<\omega}M$, $\tp(\bar{a})$ is the set of formulas satified by $\bar{a}$ 
In our theorem, addressing $L_n$, all types considered will have of cource a uniform upper bound of free variables, namely $n$.

\begin{lemma} Suppose that $T$ is a theory, 
$|T|=\lambda$, $\lambda$ regular, then there exist models $M_i: i<\chi={}^{\lambda}2$, each of cardinality $\lambda$, 
such that if $i(1)\neq i(2)< \chi$, $\bar{a}_{i(l)}\in |M_{i(l)}|$, $l=1,2,$, $\tp(\bar{a_{l(1)}})=\tp(\bar{a_{l(2)}})$, 
then there are $p_i\subseteq \tp(\bar{a_{l(i)}}),$ $|p_i|<\lambda$ and $p_i\vdash \tp(\bar{a_ {l(i)}}).$
\end{lemma}
\begin{proof} \cite{Shelah} Theorem 5.16
\end{proof}

\begin{theorem}\label{uncountable} Let $\A=S_c\Nr_n\CA_{\omega}$. Assume that $|A|=\lambda$, where $\lambda$ is an ucountable 
cardinal. assume that  $\kappa< {}^{\lambda}2$,
and that $(F_i: i<\kappa)$ is a system of non principal ultrafilters.
Then there exists
a set algebra $\C$ with base $U$ such that $|U|\leq \lambda$, $f:\A\to \C$ such that $f(a)\neq 0$ and for all $i\in \kappa$, $\bigcap_{x\in X_i} f(x)=0.$

\end{theorem}

\begin{proof} Let $\A\subseteq_c \Nr_n\B$, where $\B$ is $\omega$ dimensional, locally finite and has the same cardinality as $\A$. 
This is possible by taking $\B$ to be the subalgebra of which $\A$ is a strong neat reduct generated by  $A$, and noting that we gave countably many 
operations. 
The $F_i$'s correspond to maximal $n$ types in the theory $T$ corresponding to $\B$, that is, the first order theory $T$ such that $\Fm_T\cong \B$. 
Applying Shelah result of the existence of $^{\lambda}2$ representations of $\B$ and restricting to ultrafilters (maximal types)  in $\Nr_n\B$, 
together  with argument (ii) above, gives a a representation with base $\M$, of the big algebra $\B$, via an injective homomorphism $g$,
omitting the given maximal types. For a sequence $s$ with finite length let $s^+=s\cup id$.
Define $f:\A\to \wp({}^{n}\M)$ via $a\mapsto \{s\in {}^nM: s\cup Id\in f(a)\},$ then clearly $f$ is as desired.

\end{proof}
Since $\Nr_n\CA_{\omega}\subseteq S_c\Nr_n\CA_{\omega}$, we immediately get:
\begin{corollary}
Let $\A\in \Nr_n\CA_{\omega}$ be infinite such that $|A|=\lambda$, $\lambda$ is  a regular cardinal.
Let $\kappa<2^{\lambda}$. Let $(X_i:i\in \kappa)$ be a family of non-principal ultrafilters of $\A$.
Then there exists a representation $f:\A\to \wp(^nX)$ such that $\bigcap_{x\in X_i}f(x)=\emptyset$
for all $i\in \kappa$.
\end{corollary}
One way of classifying  formulas is by the amount of quantification. 
Formulae with less depth of quantifier alternation are thought of as simpler and the quantifier free formulae as the simplest. 
Recall that a  theory  has quantifier elimination if for every formula $\alpha$  there exists a formula $\alpha_{QF}$
without quantifiers which is equivalent to it (modulo the theory). 

 
Now one metalogical reading of the last two theorems is 

\begin{theorem} Let $T$ be an $\L_n$ consistent theory that admits elimination of quantifiers. 
Assume that $|T|=\lambda$ is a regular cardinal.
Let $\kappa<2^{\lambda}$. Let $(\Gamma_i:i\in \kappa)$ be a set of non-principal maximal types in $T$. Then there is a model $\M$ of $T$ that omits
all the $\Gamma_i$'s
\end{theorem}
\begin{demo}{Proof} If $\A=\Fm_T$ denotes the cylindric algebra corresponding to $T$, then since $T$ admits elimination of quantifiers, then
$\A\in \Nr_n\CA_{\omega}$. This follows from the following reasoning. Let $\B=\Fm_{T_{\omega}}$ be the locally finite cylindric algebra
based on $T$ but now allowing $\omega$ many variables. Consider the map $\phi/T\mapsto \phi/T_{\omega}$. 
Then this map is from $\A$ into $\Nr_n\B$. But since $T$ admits elimination of quantifiers the map is onto.
The Theorem now follows.
\end{demo}

We now give another natural omitting types theorem for certain uncountable languages.
Let $L$ be an ordinary first order language with 
a list $\langle c_k\rangle$ of individual constants
of order type $\alpha$. $L$ has no operation symbols, but as usual, the
list of variables is of order type $\omega$. Denote by $Sn^{L_{\alpha}}$ the set
of all $L$ sentences, the subscrpt $\alpha$ indicating that we have $\alpha$ many constants Let $\alpha=n\in \omega$. 
Let $T\subseteq Sn^{L_0}$ be consistent. Let $\M$  be an $\L_0$  model of $T$.  Then any $s:n\to M$ defines an expansion of $\M$ to $L_n$ which we denote by $\M{[s]}$.  
For $\phi\in L_n$ let $\phi^{\M}=\{s\in M^n: \M[s]\models \phi\}$. Let 
$\Gamma\subseteq Sn^{L_{n}}$. 
The question we adress is: Is there a model $\M$ of $T$ such that for no expansion $s:n\to M$ we have 
$s\in  \bigcap_{\phi\in \Gamma}\phi^{M}$. 
Such an $\M$ omits $\Gamma$. Call $\Gamma$ principal over $T$ if there exists $\psi\in L_n$ consistent with $T$ such that $T\models \psi\to \Gamma.$ 
Other wise $\Gamma$ is non principal over T.

\begin{theorem}  Let $T\subseteq Sn^{L_0}$ be consistent and assume that $\lambda$ is a regular cardinal, and $|T|=\lambda$. 
Let $\kappa<2^{\lambda}$. Let $(\Gamma_i:i\in \kappa)$ be a set of non-principal maximal types in $T$. 
Then there is a model $\M$ of $T$ that omits
all the $\Gamma_i$'s
That is, there exists a model $\M\models T$ such that is no $s:n\to \M$
such that   $s\in \bigcap_{\phi\in \Gamma_i}\phi^{\M}$. 
\end{theorem}
\begin{demo}{Proof}
Let $T\subseteq Sn^{L_0}$ be consistent. Let $\M$ be an $\L_0$  model of $T$. 
For $\phi\in Sn^L$ and $k<\alpha$
let $\exists_k\phi:=\exists x\phi(c_k|x)$ where $x$ is the first variable
not occuring in $\phi$. Here $\phi(c_k|x)$ is the formula obtained from $\phi$ by 
replacing all occurences of $c_k$, if any, in $\phi$ by $x$.
Let $T$ be as indicated above, i.e $T$ is a set of sentences in which no constants occur. Define the
equivalence relation $\equiv_{T}$ on $^2Sn^L$ as follows
$$\phi\equiv_{T}\psi \text { iff } T\models \phi\equiv \psi.$$
Then, as easily checked $\equiv_{T}$ is a 
congruence relation on the algebra 
$$\Sn=\langle Sn,\land,\lor,\neg,T,F,\exists_k, c_k=c_l\rangle_{k,l<n}$$  
We let $\Sn^L/T$ denote the quotient algebra.
In this case, it is easy to see that $\Sn^L/T$ is a $\CA_n$, in fact is an $\RCA_n$.
Let $L$ be as described in  above. But now we denote it $L_n$, 
the subsript $n$ indicating that we have $n$-many 
individual constants. Now enrich $L_{n}$ 
with countably many constants (and nothing else) obtaining
$L_{\omega}$. 
Recall that both languages, now, have a list of $\omega$ variables. 
For $\kappa\in \{n, \omega\}$ 
let $\A_{\kappa}=\Sn^{L_{k}}/{T}$. 
For $\phi\in Sn^{L_n}$, let $f(\phi/T)=\phi/{T}$. 
Then, as easily checked $f$ is an embedding  of $\A_{n}$ 
into $\A_{\omega}$. Moreover $f$ has the additional property that
it maps $\A_{n}$, into (and onto) the neat $n$ reduct of $\A_{\beta}$,
(i.e. the set of $\alpha$ dimensional elements of $A_{\beta}$). 
In short, $\A_{n}\cong \Nr_{n}\A_{\omega}$. Now agin putting $X_i=\{\phi/T: \phi\in \Gamma_i\}$ and using that the 
$\Gamma_i$'s are maximal non isolated, it follows that the 
$X_i's$ are non-principal ultrafilters
Since $\Nr_n\CA_{\omega}\subseteq S_c\Nr_n\CA_{\omega}$, then our result follows, also from Theorem 1.    
\end{demo}


\section{Complete representability}

The notion of complete representation has turned out to be an intensely interesting topic in algebraic logic \cite{OTT}, \cite{HH}, for an overview,
as well as for undefined notions in what follows. It is known that $L_2$ enjoys a Vaughts theorem, but if we delete equality, 
and introduce substitutions, we lose this property; there are atomic theories, with no atomic models.
This is, to the best of our knowledge a new result. We start with providing a proof of algebraic version of 
the positive result  inspired by modal logic, the proof is due to Hirsch and Hodkinson:

\begin{theorem} Any atomic $\RCA_2$ is completely representable
\end{theorem}

\begin{proof} This is a nice proof inspired by modal logic taken from Hirsch and Hodkinson. 
Let $\A\in RCA_2$ be atomic. As $\RCA_2$ is conjugated and can be defined by Sahlqvist equations it is closed under completions. So
the equations are valid in $At\A$. But $At\A$ is a bounded morphic image of a disjoint union of square frames $F_i$ $(i\in I)$. 
Each $\Cm F_i$ is a full 2 dimension cylindric set
algebra. The inverse of the bounded morphism is an embedding from $\A$ into
$\prod \Cm F_i$ that preserves all meets
\end{proof}
The same can be proved for polyadic {\it equality} algebras of dimension $2$. 
Furthermore, there is a common belief that the same results holds for quasi polyadic algebras without equality.
Here we show that it is {\it not} the case; for indeed such a variety is not aditive, least conjugated.
We show that in the absence of diagonal elements, atomic algebras of dimension $2$ atomic algebras might not be completely representable.
However, we first give a simpler example showing that the omiting types theorem fails in the corresponding logic, namely, 
single non-principal types in countable languages may not be 
omitted. The proof works for every dimension, infinite included.

\begin{example}

It suffices to show that there is an algebra $\A$, and a set $S\subseteq A$, such that $s_0^1$ does not preserves $\sum S$.
For if $\A$ had a representation as stated in the theorem, this would mean that $s_0^1$ is completely additive in $\A$.

For the latter statement, it clearly suffices to show that if $X\subseteq A$, and $\sum X=1$,
and there exists  an injection $f:\A\to \wp(V)$, such that $\bigcup_{x\in X}f(x)=V$,
then for any $\tau\in {}^nn$, we have $\sum s_{\tau}X=1$. So fix $\tau \in V$ and assume that this does not happen.
Then there is a $y\in \A$, $y<1$, and
$s_{\tau}x\leq y$ for all $x\in X$.
(Notice that  we are not imposing any conditions on cardinality of $\A$ in this part of the proof).
Now
$$1=s_{\tau}(\bigcup_{x\in X} f(x))=\bigcup_{x\in X} s_{\tau}f(x)=\bigcup_{x\in X} f(s_{\tau}x).$$
(Here we are using that $s_{\tau}$ distributes over union.)
Let $z\in X$, then $s_{\tau}z\leq y<1$, and so $f(s_{\tau}z)\leq f(y)<1$, since $f$ is injective, it cannot be the case that $f(y)=1$.
Hence, we have
$$1=\bigcup_{x\in X} f(s_{\tau}x)\leq f(y) <1$$
which is a contradiction, and we are done.
Now we turn to constructing the required  counterexample, which is an easy adaptation of a 
construction due to Andr\'eka et all in \cite{AGMNS} to our present situation.
We give the detailed construction for the reader's conveniance.

Let $\B$ be an atomless Boolean set algebra with unit $U$, that has the following property:

For any distinct $u,v\in U$, there is $X\in B$ such that $u\in X$ and $v\in {}\sim X$.
For example $\B$ can be taken to be the Stone representation of some atomless Boolean algebra.
The cardinality of our constructed algebra will be the same as $|B|$.
Let $$R=\{X\times Y: X,Y\in \B\}$$
and
$$A=\{\bigcup S: S\subseteq R: |S|<\omega\}.$$
Then indeed we have $|R|=|A|=|B|$. 

We claim that $\A$ is a subalgebra of $\wp(^2U)$.

Closure under union is obvious. To check intersections, we have:
$$(X_1\times Y_1)\cap (X_2\times Y_2)=(X_1\cap X_2) \times (Y_1\cap Y_2).$$
Hence, if $S_1$ and $S_2$ are finite subsets of $R$, then
$$S_3=\{W\cap Z: W\in S_1, Z\in S_2\}$$
is also a finite subset of $R$ and we have
$$(\bigcup S_1)\cap (\bigcup S_2)=\bigcup S_3$$
For complementation:
$$\sim (X\times Y)=[\sim X\times U]\cup [U\times \sim Y].$$
If $S\subseteq R$ is finite, then
$$\sim \bigcup S=\bigcap \{\sim Z: Z\in S\}.$$
Since each $\sim Z$ is in $A$, and $A$ is closed under intersections, we conclude that
$\sim \bigcup S$ is in $A$.
We now show that it is closed under substitutions:
$$S_0^1(X\times Y)=(X\cap Y)\times U, \\ \ S_1^0(X\times Y)=U\times (X\cap Y)$$
$$S_{01}(X\times Y)=Y\times X.$$
Let
$$D_{01}=\{s\in U\times U: s_0=s_1\}.$$
We claim that the only subset of $D_{01}$ in $\A$ is the empty set.

Towards proving this claim, assume that $X\times Y$ is a non-empty subset of $D_{01}$.
Then for any $u\in X$ and $v\in Y$ we have $u=v$. Thus $X=Y=\{u\}$ for some $u\in U$.
But then $X$ and $Y$ cannot be in $\B$ since the latter is atomless and $X$ and $Y$ are atoms.
Let
$$S=\{X\times \sim X: X\in B\}.$$
Then
$$\bigcup S={}\sim D_{01}.$$
Now we show that
$$\sum{}^{\A}S=U\times U.$$
Suppose that $Z$ is an upper bound different from $U\times U$. Then $\bigcup S\subseteq Z.$ Hence
$\sim D_{01}\subseteq Z$, hence $\sim Z=\emptyset$, so $Z=U\times U$.
Now $$S_{0}^1(U\times U) =(U\cap U)\times U=U\times U.$$
But
$$S_0^1(X\times \sim X)=(X\cap \sim X)\times U=\emptyset.$$
for every $X\in B$.
Thus $$S_0^1(\sum S)=U\times U$$
and
$$\sum \{S_{0}^1(Z): Z\in S\}=\emptyset.$$

For $n>2$, one takes $R=\{X_1\times\ldots\times X_n: X_i\in \B\}$ and the definition of $\A$ is the same. Then,
in this case, one takes $S$ to be
$X\times \sim X\times U\times\ldots\times U$
such that $X\in B$. The proof survives verbatim.

By taking $\B$ to be countable, then $\A$ can be countable, and so it violates the omitting types theorem.
\end{example}

Now let $SA_n$ be the cylindrifier free reducts of polyadic algebras. Then:

\begin{theorem}\label{counter2} For every finite $n\geq 2$, there exists an atomic algebra in 
$SA_n$ that is not completely representable.
\end{theorem}
\begin{proof} It is enough, in view of the previous theorem,  to construct an atomic algebra such that $\Sigma_{x\in \At\A} s_0^1x\neq 1$.
In what follows we produce such an algebra. 
(This algebra will be uncountable, due to the fact that it is infinite and complete, so it cannot be countable.
In particular, it cannot be used to violate the omitting types theorem, the most it can say is that the omitting types theorem fails 
for uncountable languages, which is not too much of a surprise).

Let $\mathbb{Z}^+$ denote the set of positive integers.
Let $U$ be an infinite set. Let $Q_n$, $n\in \omega$, be a family of $n$-ary relations that form  partition of $^nU$ 
such that $Q_0=D_{01}=\{s\in {}^nU: s_0=s_1\}$. And assume also that each $Q_n$ is symmetric; for any $i,j\in n$, $S_{ij}Q_n=Q_n$. 
For example one can $U=\omega$, and for $n\geq 1$, one sets 
$$Q_n=\{s\in {}^n\omega: s_0\neq s_1\text { and }\sum s_i=n\}.$$
These even works for weak set algebras of infinite dimension. In this case, one takes
Let $(Q_n: n\in \omega)$ be a sequence $\alpha$-ary relations such that

\begin{enumroman}

\item $(Q_n: n\in \omega)$ is a partition of
$$V={}^{\alpha}\alpha^{(\bold 0)}=\{s\in {}^{\alpha}\alpha: |\{i: s_i\neq 0\}|<\omega\}.$$



\item Each $Q_n$ is symmetric.
\end{enumroman}
Take $Q_0=\{s\in V: s_0=s_1\}$, and for each $n\in \omega\sim 0$, take $Q_n=\{s\in {}^{\alpha}\omega^{({\bold 0})}: s_0\neq s_1, \sum s_i=n\}.$
(Note that this is a finite sum).
Clearly for $n\neq m$, we have $Q_n\cap Q_m=\emptyset$, and $\bigcup Q_n=V.$
Furthermore, obviously each $Q_n$ is symmetric, that is  $S_{[i,j]}Q_n=Q_n$ for all $i,j\in \alpha$
Now fix $F$ a non-principal ultrafilter on $\mathcal{P}(\mathbb{Z}^+)$. For each $X\subseteq \mathbb{Z}^+$, define
\[
 R_X =
  \begin{cases}
   \bigcup \{Q_k: k\in X\} & \text { if }X\notin F, \\
   \bigcup \{Q_k: k\in X\cup \{0\}\}      &  \text { if } X\in F
  \end{cases}
\]

Let $$\A=\{R_X: X\subseteq \mathbb{Z}^+\}.$$
Notice that $\A$ is uncountable. Then $\A$ is an atomic set algebra with unit $R_{\mathbb{Z}^+}$, and its atoms are $R_{\{k\}}=Q_k$ for $k\in \mathbb{Z}^+$.
(Since $F$ is non-principal, so $\{k\}\notin F$ for every $k$).
We check that $\A$ is indeed closed under the operations.
Let $X, Y$ be subsets of $\mathbb{Z}^+$. If either $X$ or $Y$ is in $F$, then so is $X\cup Y$, because $F$ is a filter.
Hence
$$R_X\cup R_Y=\bigcup\{Q_k: k\in X\}\cup\bigcup \{Q_k: k\in Y\}\cup Q_0=R_{X\cup Y}$$
If neither $X$ nor $Y$ is in $F$, then $X\cup Y$ is not in $F$, because $F$ is an ultrafilter.
$$R_X\cup R_Y=\bigcup\{Q_k: k\in X\}\cup\bigcup \{Q_k: k\in Y\}=R_{X\cup Y}$$
Thus $A$ is closed under finite unions. Now suppose that $X$ is the complement of $Y$ in $\mathbb{Z}^+$.
Since $F$ is an ultrafilter exactly one of them, say $X$ is in $F$.
Hence,
$$\sim R_X=\sim{}\bigcup \{Q_k: k\in X\cup \{0\}\}=\bigcup\{Q_k: k\in Y\}=R_Y$$
so that  $\A$ is closed under complementation (w.r.t $R_{\mathbb{Z}^+}$).
We check substitutions. Transpositions are clear, so we check only replacements. It is not too hard to show that
\[
 S_0^1(R_X)=
  \begin{cases}
   \emptyset & \text { if }X\notin F, \\
   R_{\mathbb{Z}^+}      &  \text { if } X\in F
  \end{cases}
\]

Now
$$\sum \{S_0^1(R_{k}): k\in \mathbb{Z}^+\}=\emptyset.$$
and
$$S_0^1(R_{\mathbb{Z}^+})=R_{\mathbb{Z}^+}$$
$$\sum \{R_{\{k\}}: k\in \mathbb{Z}^+\}=R_{\mathbb{Z}^+}=\bigcup \{Q_k:k\in \mathbb{Z}^+\}.$$
Thus $$S_0^1(\sum\{R_{\{k\}}: k\in \mathbb{Z}^+\})\neq \sum \{S_0^1(R_{\{k\}}): k\in \mathbb{Z}^+\}.$$
Our next theorem gives a plathora of algebras that are not completely representable. Any algebra which shares the atom structure of $\A$ 
constructed above cannot have a complete representation. Formally:
\end{proof}

\begin{theorem} There are atomic $PA_2$s that are not completely representable; 
furthermore thay can be countable. However, the class of atomic $PA_2$ is elementary
\end{theorem}
\begin{proof}

We modify the above proof as follows. Take each $Q_n$ so it has both domain and range equal to $U$. 
This is possible; indeed it is easy to find such a 
partition of ${}^2U$.
Let $\At(x)$ be the first order formula expressing that $x$ is an atom. That is $\At(x)$ is the formula
$x\neq 0\land (\forall y)(y\leq x\to y=0\lor y=x)$. For distinct $i,j<n$ let $\psi_{i,j}$ be the formula:
$y\neq 0\to \exists x(\At(x)\land s_i^jx\neq 0\land s_i^jx\leq y).$ Let $\Sigma$ be obtained from $\Sigma_n$ by adding $\psi_{i,j}$
for every distinct $i,j\in n$.
We show that $CSA_n={\bf Mod}(\Sigma)$. Let $\A\in CSA_n$. Then, by theorem \ref{converse}, we have
$\sum_{x\in X} s_i^jx=1$ for all $i,j\in n$. Let $i,j\in $ be distinct. Let $a$ be non-zero, then $a.\sum_{x \in X}s_i^jx=a\neq 0$,
hence there exists $x\in X$, such that
$a.s_i^jx\neq 0$, and  so $\A\models \psi_{i,j}$.
Conversely, let $\A\models \Sigma$. Then for all $i,j\in n$, $\sum_{x\in X} s_i^jx=1$. Indeed, assume that for some distinct
$i,j\in n$, $\sum_{x\in X}s_i^jx\neq 1$.
Let $a=1-\sum_{x\in X} s_i^jx$. Then $a\neq 0$. But then there exists $x\in X$, such that $s_i^jx.\sum_{x\in X}s_i^jx\neq 0$
which is impossible. But for distinct $i, j\in n$, we have  $\sum_{x\in X}s_{[i,j]}X=1$ anyway, and so $\sum s_{\tau}X=1$, for all
$\tau\in {}^nn$, and so it readily follows that $\A\in CRA_n.$ So if $\A$ models $\Sigma$, then it is completely additive, and the above modal theoretic
proof works.
\end{proof}
When we do not have cylindrifiers the above example works for all dimension {\it infinite} included, using instead weak set algebras.
Indeed, we have

\begin{theorem}\label{counterinfinite2} There is an atomic $\A\in SA_{\alpha}$ such that $\A$ is not completely representable.
\end{theorem}

\begin{proof} First it is clear that if $V$ is any weak space, then $\wp(V)\models \Sigma$.
Let $(Q_n: n\in \omega)$ be a sequence $\alpha$-ary relations such that

\begin{enumroman}

\item $(Q_n: n\in \omega)$ is a partition of
$$V={}^{\alpha}\alpha^{(\bold 0)}=\{s\in {}^{\alpha}\alpha: |\{i: s_i\neq 0\}|<\omega\}.$$



\item Each $Q_n$ is symmetric.
\end{enumroman}
Take $Q_0=\{s\in V: s_0=s_1\}$, and for each $n\in \omega\sim 0$, take $Q_n=\{s\in {}^{\alpha}\omega^{({\bold 0})}: s_0\neq s_1, \sum s_i=n\}.$
(Note that this is a finite sum).
Clearly for $n\neq m$, we have $Q_n\cap Q_m=\emptyset$, and $\bigcup Q_n=V.$
Furthermore, obviously each $Q_n$ is symmetric, that is  $S_{[i,j]}Q_n=Q_n$ for all $i,j\in \alpha$.

Now fix $F$ a non-principal ultrafilter on $\mathcal{P}(\mathbb{Z}^+)$. For each $X\subseteq \mathbb{Z}^+$, define
\[
 R_X =
  \begin{cases}
   \bigcup \{Q_n: n\in X\} & \text { if }X\notin F, \\
   \bigcup \{Q_n: n\in X\cup \{0\}\}      &  \text { if } X\in F
  \end{cases}
\]
Let $$\A=\{R_X: X\subseteq \mathbb{Z}^+\}.$$
Then $\A$ is an atomic set algebra, and its atoms are $R_{\{n\}}=Q_n$ for $n\in \mathbb{Z}^+$.
(Since $F$ is non-principal, so $\{n\}\notin F$ for every $n$.
Then one proceeds exactly as in the finite dimensional case, theorem \ref{counter2}.
\end{proof}
The latter answers an implicit question of Hodkinson's.
Another result, where weak models offer solace is:
This is a new theorem, giving a characterization of completely representable 
algebras on weak unit, and raises a natural question, namely, is the class of infinite dimensional representable 
on weak units elementary.

\begin{theorem} Let $\A\in S_c\Nr_{\alpha}\CA_{\alpha+\omega}$ be countable. 
Then $\A$ is atomic iff $\A$ has a complete representation on weak models
\end{theorem}
\begin{proof} Let $\A\subseteq_c\Nr_{\omega}\B'$, where $\B'\in CA_{\omega+\omega}$
Let $X=\At\A$. Then $\sum ^{\B'}X=1$. Let $\B=\Sg^{\B'}\A$, then $\B\in Dc_{\omega}$, and
$\sum^{\B}X=1$. Let $V={}^{\omega}\omega^{(Id)}=\{\tau \in {}^{\omega}\omega: |\{i\in \omega: \tau(i)\neq i\}|<\omega\}$.
Then for every $\tau\in V$, $s_{\tau}^{\B}X=1$. Also, in $\B$  we have $c_kx=\sum s_k^lx$. Let $F$ be an ultrafilter that preserves
these joins, and define
$$f:\A\to \wp(V)$$ via 
$$a\mapsto \{\tau \in V: s_{\tau}a\in F\}.$$
Clearly $f$ is a complete representation.
\end{proof}


We will show in example \ref{countable} that maximality cannot be omitted by showing that 
there exists an uncountable atomic simple algebra $\A\in \Nr_n\CA_{\omega}$ which is not completely representable. 
This example is sketched in a remark 31, p. 688, 
by Robin Hirsch; here we give the detailed proof, and also 
show that this construction answers a question raised by Hirsch in the same paper, see  p.674, we quote 
' Whether this characterization works for uncountable algebras remain unknown.'(End of quote) 
Here we show it does not. This theorem was also mentioned in \cite{OTT} without proof. 
We shall need the following form of the Erdos-Rado cominatorial principle (Extending Ramsey's theorem to uncountable sets)

\begin{theorem}(Erdos-Rado)
If $r\geq 2$ is finite, $k$  an infinite cardinal, then
$$exp_r(k)^+\to (k^+)_k^{r+1}$$

where $exp_0(k)=k$ and inductively $exp_{r+1}(k)=2^{exp_r(k)}$.
\end{theorem}
The above partition symbol describes the following statement. If $f$ is a coloring of the $r+1$ 
element subsets of a set of cardinality $exp_r(k)^+$
in $k$ many colors, then there is a homogeneous set of cardinality $k^+$+ 
(a set, all whose $r+1$ element subsets get the same $f$-value).

Our cylindric algebra will be based on the atom structure of a relation algebra that has an $\omega$-dimensional cylindric base.
So we review some notions from relation and cylindric algebras, concerning atom structures and networks (which are graphs labelled by atoms).


A \textit{relation atom structure} $\Upsilon=(A, Id, \breve{}, C)$
consists of a non-empty set $A$ of atoms, a unary function
$\breve{}:A\to A$ and a ternary relation $C$ satisfying
\begin{itemize}
\item $a=b$ $\longleftrightarrow  \exists e\in A$ such that $
 e \in Id~ \land~ (a,e,b) \in C$,
\item $(a,e,b) \in C \implies (\breve{a}, c, b) \in C$ and $(c, \breve{b}, a) \in
C$,
\item $(a,b,c) \in C\land (c,d,g) \in C\implies \exists f\in A$ such that $(a,f,g) \in C\land
(b,d,f) \in C$.
\end{itemize}
Let $\A=(A, +, ., -, 0,1, Id, \breve{}, ;)$ be an atomic  relation
algebra. Let $At \A$ denote the set of atoms. Then this determines
an atom structure
$$\At \A=(At \A, Id, \breve{} ,C)$$
whose domain is the set of atoms of $A$, the identity $Id=\{e\in At
\A: e\leq 1'\}$, converse is the restriction of $~\breve{}~$ on the
atoms of $\A$, and the ternary relation $C$ is defined by
$$ (a,b,c) \in C\longleftrightarrow a;b\geq c, \text { for all atoms }a,b,c.$$
A triple in $C$ is said to be a \textit{consistent triple}.
\\

Conversely from an atom structure $\Upsilon=(A, Id, \breve{}, C)$ it
is possible to define an \textit{atomic relation algebra} $\Cm
\Upsilon$ whose domain is the power set of $A$, the boolean
operations $\cup$, complement in $A$, and whose non-boolean
operations are defined by
$$1'=Id,$$
$$\breve{r}=\{\breve{a}: a\in r\},$$
$$r;s=\{c: \exists a\in r, \exists b\in s, (a,b,c) \in C\}.$$
$\Cm\Upsilon$ is called the complex algebra. The set of forbidden
triples of $\Upsilon$ is $A\times A\times A\smallsetminus C$, i.e.,
it is the complement of the consistent triples. Listing the
forbidden triples specifies the composition part of the atom
structure.

\begin{definition}
Let $\A$ be a relation algebra. An $\A$-labelled graph $N=(N_1,
N_2)$ consists of a set $N_1$ of nodes, and a map $N_2:N_1\times
N_1\to \A.$ $N$ is called an $\A$ network, if for all nodes
$x,y,z\in N_1$:
\begin{itemize}
\item $N_2(x,x)\leq 1'$,
\item $N_2(x,y)\cdot (N_2(x,z);N_2(z,y))\neq 0$.
\end{itemize}
\end{definition}

\begin{definition} { \ }

\begin{enumerate}

\item An $\A$ network is atomic if $N(x,y)$ is an atom for all $x,y\in N.$
For an atomic $\A$ network it is not difficult to show that the
following hold for $x,y,z\in N$:

\begin{itemize}

\item $N(x,y)=N(y,x)$,

\item $N(x,y)\leq N(x,z);N(z,y)$,

\item $N(x,x)=1'.(N(x,y);N(y,x))$.

\end{itemize}
For networks $N$ and $M$, we write $M\equiv_{x_1\ldots x_m}N$ if
$M(\bar{a})=N(\bar{a})$ for all $\bar{a}\in {}^2(n-\{x_1,\ldots,
x_m\}).$

\item Let $n\leq \omega$. An $n$-dimensional cylindric basis for $\A$ is a set $\cal M$
of atomic networks $N$ with nodes $(N)=n$ such that

(a) For $a\in At \A$ there is $N\in \cal M$ with $N(0,1)=a.$

(b) If $N\in \cal M$ and $x,y<n$, and $a,b\in At \A$, $N(x,y)\leq
a;b$ and $x,y\neq z<n$, then there is some $M\in \cal M$ such that
$M\equiv_zN$, $M(x,z)=a$ and $M(z,y)=b$.

(c) If $M, N\in \cal M$ $x,y<n$, $x\neq y$ and $M\equiv_{xy}N$, then
there is some $L\in \cal M$ such that $M\equiv_x L\equiv_yN.$
\end{enumerate}
\end{definition}
Maddux calls networks basic matrices. 
A relation algebra can be specified either by specifying the consistent triples of atoms or the forbbiden triples, a forbidden triple is one that
is not consistent.

\begin{example}\label{countable}

\begin{enumarab}

\item We define an atomic relation algebra $\A$ with uncountably many
atoms. Let $\kappa$ be an infinite cardinal. This algebra will be used to construct cylindric algebras of dimension 
$n$ showing that countability is essential in the above characterization.  

The atoms are $1', \; a_0^i:i<2^{\kappa}$ and $a_j:1\leq j<
\kappa$, all symmetric.  The forbidden triples of atoms are all
permutations of $(1',x, y)$ for $x \neq y$, \/$(a_j, a_j, a_j)$ for
$1\leq j<\kappa$ and $(a_0^i, a_0^{i'}, a_0^{i^*})$ for $i, i',
i^*<2^{\kappa}.$  In other words, we forbid all the monochromatic
triangles.  

Write $a_0$ for $\set{a_0^i:i<2^{\kappa}}$ and $a_+$ for 
$\set{a_j:1\leq j<\kappa}$. Call this atom
structure $\alpha$.  

Let $\A$ be the term algebra on this atom
structure; the subalgebra of $\Cm\alpha$ generated by the atoms.  $\A$ is a dense subalgebra of the complex algebra
$\Cm\alpha$. We claim that $\A$, as a relation algebra,  has no complete representation.

Indeed, suppose $\A$ has a complete representation $M$.  Let $x, y$ be points in the 
representation with $M \models a_1(x, y)$.  For each $i<\omega_1$ there is a 
point $z_i \in M$ such that $M \models a_0^i(x, z_i) \wedge a_1(z_i, y)$.  

Let $Z = \set{z_i:i<2^{\kappa}}$.  Within $Z$ there can be no edges labelled by 
$a_0$ so each edge is labelled by one of the $\kappa$ atoms in 
$a_+$.  The Rado Erdos theorem forces the existence of three points 
$z^1, z^2, z^3 \in Z$ such that $M \models a_j(a^1, z^2) \wedge a_j(z^2, z^3) 
\wedge a_j(z^3, z_1)$, for some single $j<\kappa$.  This contradicts the 
definition of composition in $\A$.

Let $S$ be the set of all atomic $\A$-networks $N$ with nodes
 $\omega$ such that\\ $\set{a_i: 1\leq i<\omega,\; a_i \mbox{ is the label 
of an edge in }
 N}$ is finite.
Then it is straightforward to show $S$ is an amalgamation class, that is for all $M, N 
\in S$ if $M \equiv_{ij} N$ then there is $L \in S$ with 
$M \equiv_i L \equiv_j N.$  
Hence the complex cylindric algebra $\Ca(S)\in \CA_\omega$.

Now let $X$ be the set of finite $\A$-networks $N$ with nodes
$\subseteq\omega$ such that 
\begin{enumerate}
\item each edge of $N$ is either (a) an atom of
$\c A$ or (b) a cofinite subset of $a_+=\set{a_j:1\leq j<\kappa}$ or (c)
a cofinite subset of $a_0=\set{a_0^i:i<2^{\kappa}}$ and
\item $N$ is `triangle-closed', i.e. for all $l, m, n \in nodes(N)$ we
have $N(l, n) \leq N(l,m);N(m,n)$.  That means if an edge $(l,m)$ is
labelled by $1'$ then $N(l,n)= N(m,n)$ and if $N(l,m), N(m,n) \leq
a_0$ then $N(l,n).a_0 = 0$ and if $N(l,m)=N(m,n) =
a_j$ (some $1\leq j<\omega$) then $N(l,n).a_j = 0$.
\end{enumerate}
For $N\in X$ let $N'\in\Ca(S)$ be defined by 
\[\set{L\in S: L(m,n)\leq
N(m,n) \mbox{ for } m,n\in nodes(N)}\]
For $i,\omega$, let $N\restr{-i}$ be the subgraph of $N$ obtained by deleting the node $i$.
Then if $N\in X, \; i<\omega$ then $\cyl i N' =
(N\restr{-i})'$.
The inclusion $\cyl i N' \subseteq (N\restr{-i})'$ is clear.

Conversely, let $L \in (N\restr{-i})'$.  We seek $M \equiv_i L$ with
$M\in N'$.  This will prove that $L \in \cyl i N'$, as required.
Since $L\in S$ the set $X = \set{a_i \notin L}$ is infinite.  Let $X$
be the disjoint union of two infinite sets $Y \cup Y'$, say.  To
define the $\omega$-network $M$ we must define the labels of all edges
involving the node $i$ (other labels are given by $M\equiv_i L$).  We
define these labels by enumerating the edges and labelling them one at
a time.  So let $j \neq i < \omega$.  Suppose $j\in nodes(N)$.  We
must choose $M(i,j) \leq N(i,j)$.  If $N(i,j)$ is an atom then of
course $M(i,j)=N(i,j)$.  Since $N$ is finite, this defines only
finitely many labels of $M$.  If $N(i,j)$ is a cofinite subset of
$a_0$ then we let $M(i,j)$ be an arbitrary atom in $N(i,j)$.  And if
$N(i,j)$ is a cofinite subset of $a_+$ then let $M(i,j)$ be an element
of $N(i,j)\cap Y$ which has not been used as the label of any edge of
$M$ which has already been chosen (possible, since at each stage only
finitely many have been chosen so far).  If $j\notin nodes(N)$ then we
can let $M(i,j)= a_k \in Y$ some $1\leq k < \omega$ such that no edge of $M$
has already been labelled by $a_k$.  It is not hard to check that each
triangle of $M$ is consistent (we have avoided all monochromatic
triangles) and clearly $M\in N'$ and $M\equiv_i L$.  The labelling avoided all 
but finitely many elements of $Y'$, so $M\in S$. So
$(N\restr{-i})' \subseteq \cyl i N'$.

Now let $X' = \set{N':N\in X} \subseteq \Ca(S)$.
Then the subalgebra of $\Ca(S)$ generated by $X'$ is obtained from 
$X'$ by closing under finite unions.
Clearly all these finite unions are generated by $X'$.  We must show
that the set of finite unions of $X'$ is closed under all cylindric
operations.  Closure under unions is given.  For $N'\in X$ we have
$-N' = \bigcup_{m,n\in nodes(N)}N_{mn}'$ where $N_{mn}$ is a network
with nodes $\set{m,n}$ and labelling $N_{mn}(m,n) = -N(m,n)$. $N_{mn}$
may not belong to $X$ but it is equivalent to a union of at most finitely many 
members of $X$.  The diagonal $\diag ij \in\Ca(S)$ is equal to $N'$
where $N$ is a network with nodes $\set{i,j}$ and labelling
$N(i,j)=1'$.  Closure under cylindrification is given.
Let $\c C$ be the subalgebra of $\Ca(S)$ generated by $X'$.
Then $\A = \Ra(\c C)$.
Each element of $\A$ is a union of a finite number of atoms and
possibly a co-finite subset of $a_0$ and possibly a co-finite subset
of $a_+$.  Clearly $\A\subseteq\Ra(\c C)$.  Conversely, each element
$z \in \Ra(\c C)$ is a finite union $\bigcup_{N\in F}N'$, for some
finite subset $F$ of $X$, satisfying $\cyl i z = z$, for $i > 1$. Let $i_0,
\ldots, i_k$ be an enumeration of all the nodes, other than $0$ and
$1$, that occur as nodes of networks in $F$.  Then, $\cyl 
{i_0} \ldots
\cyl {i_k}z = \bigcup_{N\in F} \cyl {i_0} \ldots
\cyl {i_k}N' = \bigcup_{N\in F} (N\restr{\set{0,1}})' \in \A$.  So $\Ra(\c C)
\subseteq \A$.
$\A$ is relation algebra reduct of $\c C\in\CA_\omega$ but has no
complete representation.
Let $n>2$. Let $\B=\Nr_n \c C$. Then
$\B\in \Nr_n\CA_{\omega}$, is atomic, but has no complete representation.

\item We have $\A\in \Nr_n\CA_{\omega}$ such that $|A|=\omega_1$, $\A$ is atomic, $A$ is simple and $\A$ has no complete representation.
Let $X=At\A$, the set of atoms of $\A$. Then $\sum X=1$. But $\A$ is not completely representable, 
this means that for any isomorphism $f:\A\to \B$, where $\B$ is a  set algebra with base $U$, we have 
$\bigcup_{x\in X} f(x)\neq {}^nU$. Let $T$ be an $L_n$ theory such that $\A\cong \Fm_T$. Then $T$ is an atomic theory, but it has not atomic model $\M$.
$T$ is an atomic theory because for evey $\psi$ consistent with $T$, there is an $\psi,$ such that $\psi/T$ is a an atom and
$\vdash \psi\to \phi$. On the other hand, if $\M$ were an atomic model of $T$, then 
the map $\Psi:\Fm_T\to Cs_{n}^{\M}$ defined by $\phi_T\mapsto \phi^{\M}$ would establish a complete
representation, which is a contradiction.
\end{enumarab} 
\end{example}

A theorem of Johnson says that if $\A$ is a cylindric algebra of dimension $n$  generated by a set $X$, such that $\Delta x\neq n$, for all $x\in X$,
has a representable diagonal free reduct, then the algebra itself is representable. Our next theorem shows that a similar situation holds 
when we replace representable by completely
representable. For $\A\in CA_n$, $\Rd_{df}\A$ denotes its diagonal free reduct.

For an $n$ dimensional cylindric algebra $c_{(n)}$ stands for the term $c_0\ldots c_{n-1}$.
The next theorem is due to Ian Hodkinson, based on ideas of Johnson. 

\begin{theorem} Let $\D$ be a cylindric algebra of dimension $n\geq 3$, that is generated by the set $\{x\in D: \Delta x\neq n\}.$
Then if $\Rd_{df}\D$ is completely representable, then so is $\D$.
\end{theorem}
\begin{proof} First suppose that $\D$ is simple, and let $h: \D\to \wp(V)$ be a complete representation, where $V=\prod_{i<n}U_i$ for sets $U_i$.
We can assume that $U_i=U_j$ for all $i,j<n$, and if $s\in V$, $i,j<n$ and $a_i=a_j$ then $a\in h(d_{ij}$. Indeed, let $\delta =\prod d_{ij}\in \D$. 
As $C$ is a cylndric algebra, we have $c_{(n)}\delta =1$, so for each $u\in U_i$ there is an $s\in h(\delta)$ with $a_i=u$. 
So there exists a function $s_i: U_i\to h(\delta)$ such that $(s_i(u))_i=u$ for each $u\in U_i$.

Let $U$ be the disjoint union of the $U_i$s. Let $t_i:U\to U_i$ be the surjection defined by $t_i(u)=(s_j(u))_i$.
Let $g: \D\to \wp(^nU)$ be defined via 
$$d\mapsto \{s\in {}^nU: (t_0(a_0),\ldots, t_{n-1}(a_{n-1}))\in h(d)\}.$$
Then $g$ is a complete representation of $\D.$ Now suppose $s\in {}^nU$, satisfies $s_i=s_j$ with $a_i\in U_k$, say, where $k<n$. 
Let $\bar{b}=s_k(a_i)=s_k(a_j)\in h(\delta).$ Then $t_i(a_i)=b_i$ and $t_j(a_j)=b_j$, so $(t_i(a_i): i<n)$ 
agrees with $\bar{b}$ on coordinates $i,j$. Since $\bar{b}\in h(\delta)$ and 
$\Delta d_{ij}=\{i,j\}$, then $(t_i(a_i): i<n)\in h(d_{ij}$ and so $s\in g(d_{ij}),$ as required.

Now define $\sim_{ij}=\{(a_i, a_j): \bar{a}\in h(d_{ij})$. Then it easy to 
check that $\sim_{01}=\sim_{i,j}$ is an equivalence relation on $U$. For 
$s,t\in {}^nU$, define $s\sim t$, if $s_i\sim t_i$ for each $i<n$, then $\sim$ is an equivalence relation
on $^nU$. Let
$$E=\{d\in D: h(d)\text { is a union of $\sim$ classes }\}.$$
Then $$\{d\in D: \Delta d\neq n\}\subseteq E.$$

Furthermore, $E$ is the domain of a complete subalgebra of $\C$. 
Let us check this. We have $\{0,1, d_{ij}: i,j<n\}\subseteq E$, since $\Delta 0=\Delta 1=\emptyset$ 
and $\Delta d_{ij}=\{i,j\}\neq n$ (as $n\geq 3$).
If $h(d)$ is a union of $\sim$ classes, then so
is $^nU\setminus h(d)= h(-d)$. If $S\subseteq E$ and $\sum S$ exists in $\D$, then because $h$ is complete representation 
we have $h(\sum^{\D}S)=\bigcup h[S]$, a union of $\sim$ classes
so $\sum S\in E$. Hence $E=C$. Now define $V=U/\sim_{01}$, and
define $g:\C\to \wp(^nV)$ via
$$c\mapsto \{(\bar{a}/\sim_{01}): \bar{a}\in h(c).$$
Then $g$ is a complete representation.

Now we drop the assumption that $\D$ is simple. Suppose that $h:\D\to \prod_{k\in K} Q_k$ is a complete representation. 
Fix $k\in K$, let $\pi_k: Q\to Q_k$ be the canonical projection, and let 
$\D_k=rng(\pi_k\circ h)$. We define diagonal elements in $\D_k$ by $d_{ij}=\pi_k(h^{\C}(d_{ij}))$. This expands $\D_k$ 
to a cylindric-type algebra $\C_k$ that is a homomorphic image of $\C$, and hence
is a cylindric algebra with diagonal free reduct $\D_k$. Then the inclusion map $i_k:\D_k\to Q_k$ is a complete
representation of $D_k$. 
Since obviously
$$\pi_k[h[\{c\in C: \Delta c\neq n\}]\subseteq \{c\in C_k: \Delta c\neq n\}$$ 
and $\pi_k, h$ preserve arbitrary sums, then $C_k$ is completely generated by $\{c\in C_k: \Delta c\neq n\}$. 
Now $c_{(n)}x$ is a discriminator term in $Q_k$, so $D_k$ is simple.
So by the above $\C_k$ has complete represenation $g_k:\C_k\to Q_k'$. Define
$g: \C\to \prod_{k\in K}Q_k'$ via
$$g(c)_k=g_k(\pi_k(h(c))).$$
Then $g$ defines a complete representation.
\end{proof}

Let $\B$ be the algebra constructed in example \ref{countable}; that is $\B\in \Nr_n\CA_{\omega}$ and $\B$ has no complete representation.
Since $\B$ is generated by its $\Ra$ reduct, that is its $2$ dimensional elements, then it follows that the diagonal free reduct of $\A$ is not 
completely representable, see theorem below.

A classical theorem of Vaught for first order logic says that countable atomic theories have countable atomic models,
such models are necessarily prime, and a prime model omits all non principal types.
We have a similar situation here:

\begin{theorem} Let $f:\A\to \wp(V)$ be an atomic representation of $\A\in CA_n$.
Then for any given family $(Y_i:i\in I)$ of subsets of $\A$,  if $\prod Y_i=0$ for all $i\in I$, then we have
$\bigcap_{y\in Y_i} f(y)=\emptyset$ for all $i\in I$.
\end{theorem}
\begin{proof} Let $i\in I$. Let $Z_i=\{-y: y\in Y_i\}$. Then $\sum Z_i=1$. Let $x$ be an atom. Then $x.\sum Z_i=x\neq 0$.
Hence there exists $z\in Z_i$, such that
$x.z\neq 0$. But $x$ is an atom, hence $x.z=x$ and so $x\leq z$. We have shown that for every atom $x$, there exists $z\in Z_i$ such that $x\leq z$.
It follows immediately that
$V=\bigcup_{x\in \At\A}f(x)\leq \bigcup_{z\in Z_i} f(z)$, and so $\bigcap_{y\in Y_i} f(y)=\emptyset,$
and we are done.
\end{proof}

Our last theorem shows that the omitting types fails in countable complete $L_n$ theories in a very strong sense.
theories chosen can be complete, atomic, having no atomic models.

\begin{corollary} For each finite $n\geq 3$, there exists a simple countable atomic representable polyadic equality 
algebra of dimension $n$ whose diagonal free reduct  not completely representable
\end{corollary}
\begin{demo}{Proof} Let $\A$ be a countable atomic representable polyadic algebra, that is not necessarily simple, satisfying that its $Df$ reduct is not 
completely representable. Such algebras exist, see the next section. 
Consider the elements $\{{\sf c}_na: a\in At\A'\}$. Then every simple component $S_a$ of $\A$ can be
obtained by relativizing to ${\sf c}_na$ for an atom $a$. 
Then one of the $S_a$'s should have no complete representation . Else for each atom $a$ $S_a$ has a complete representation $h_a$.
From those one constructs a complete representation for $\A$. 
The domain of the representation will be the disjoint union of the domains of $h_a$, and now represent $\A$ by
$$h(\alpha)=\bigcup_{a\in AtA}\{h(\alpha\cdot {\sf c}_na)\}.$$

\end{demo}

\section{Complete representability of polyadic algebras of infinite dimension}

\begin{theorem} Let $\alpha$ be an infinite ordinal. Let $\A\in \PA_{\alpha}$ be atomic. 
Then $\A$ has a complete representation.
\end{theorem}

\begin{demo}{Proof} Let $c\in A$ be non-zero. We will find a set $U$ and a homomorphism 
from $\A$ into the set algebra with universe $\wp(^{\alpha}U)$ that preserves arbitrary suprema
whenever they exist and also satisfies that  $f(c)\neq 0$. $U$ is called the base of the set algebra. 
Let $\mathfrak{m}$ be the local degree of $\A$, $\mathfrak{c}$ its effective cardinality and 
$\mathfrak{n}$ be any cardinal such that $\mathfrak{n}\geq \mathfrak{c}$ 
and $\sum_{s<\mathfrak{m}}\mathfrak{n}^s=\mathfrak{n}$. The cardinal 
$\mathfrak{n}$ will be the base of our desired representation.

Now there exists $\B\in \PA_{\mathfrak{n}}$ 
such that $\A\subseteq \Nr_{\alpha}\B$ and $A$ generates $\B$. 
The local degree of $\B$ is the same as that of $\A$, 
in particular each $x\in \B$ admits a support of cardinality $<\mathfrak{m}$. Furthermore, $|\mathfrak{n}\sim \alpha|=|\mathfrak{n}|$ and
for all $Y\subseteq A$, we have $\Sg^{\A}Y=\Nr_{\alpha}\Sg^{\B}Y.$ 
All this can be found in \cite{DM}, see the proof of theorem 1.6.1 therein; in such a proof, 
$\B$ is called a minimal dilation of $\A$. 
Without loss of generality, we assume that $\A=\Nr_{\alpha}\B$,  since $\Sg^{\A}A=\Nr_{\alpha}\Sg^{\B}A=\Nr_{\alpha}\B.$ (In the last equality 
we are using that $A$ generates $\B$).
Hence $\A$ is first order interpretable in $\B$. In particular, any first order sentence (e.g. the one expressing that $\A$ is atomic)
of the language of $\PA_{\alpha}$ translates effectively to a sentence $\hat{\sigma}$ of the language of $\PA_{\beta}$ such that for all 
$\C\in \PA_{\beta}$, we have
$\Nr_{\alpha}\C\models \sigma\longleftrightarrow \C\models \hat{\sigma}$. 
Here the languages are uncountable, even if $\A$ is countable and has countable dimension,  so we use effective in a loose sense, 
but it roughly means that there is an effective 
procedure or algorithm that does 
this translation. Since $\A=\Nr_{\alpha}\B$ and $\A$ is atomic, it 
follows that $\B$ is also atomic.
Let $\Gamma\subseteq \alpha$ and $p\in \A$. Then in $\B$ we have, see \cite{DM} the proof of theorem 1.6.1,  
\begin{equation}\label{tarek1}
\begin{split}
{\sf c}_{(\Gamma)}p=\sum\{{\sf s}_{\bar{\tau}}p: \tau\in {}^{\alpha}\mathfrak{n},\ \  \tau\upharpoonright \alpha\sim\Gamma=Id\}.
\end{split}
\end{equation}
Here, and elsewhere throughout the paper,  for a transformation $\tau$ wth domain $\alpha$ and range included in $\mathfrak{n}$,
$\bar{\tau}=\tau\cup Id_{\mathfrak{n}\sim \alpha}$. 
Let $X$ be the set of atoms of $\A$. Since $\A$ is atomic, then  $\sum^{\A} X=1$. By $\A=\Nr_{\alpha}\B$, we also have $\sum^{\B}X=1$.
We will further show that
for all $\tau\in {}^{\alpha}\mathfrak{n}$ we have,
\begin{equation}\label{tarek2}
\begin{split}
\sum {\sf s}_{\bar{\tau}}^{\B}X=1.
\end{split}
\end{equation}
It suffices to show that for every $\tau\in {}^{\alpha}\alpha$ and $a\neq 0\in \A$, there exists $x\in X$, such that ${\sf s}_{\tau}x\leq a.$
This will show that for any $\tau\in {}^{\alpha}\alpha$, the sum  $\sum{\sf s}_{\tau}^{\A}X$ 
is equal to the top element in $\A$, which is the same as that of $\B$. 
The required will then follow since for $Y\subseteq A$, we have
$\sum^{\A}Y=\sum ^{\B}Y$ by $\A=\Nr_{\alpha}\B$, and for all $x\in A$ and $\tau\in {}^{\alpha}\alpha$, we have 
${\sf s}_{\tau}^{\A}x={\sf s}_{\bar{\tau}}^{\B}x$, since $\B$ is a dilation of $\A$.

Our proof proceeds by certain non-trivial manipulations of substitutions. 
Assume that $\tau\in {}^{\alpha}\alpha$ and non-zero $a\in A$ are given. Suppose for the time being that $\tau$ is onto. 
We define a right inverse $\sigma$ of $\tau$ the usual way. That is for $i\in \alpha$ choose $j\in \tau^{-1}(i)$ and set $\sigma(i)=j$. 
Then $\sigma\in {}^{\alpha}\alpha$ and in fact $\sigma$ is one to one. 
Let $a'={\sf s}_{\sigma}a$. Then $a'\in A$ and  $a'\neq 0$, for if it did, then by polyadic axioms (7) and (8), we would get 
$0={\sf s}_{\tau}a'={\sf s}_{\tau\circ \sigma}a=a$ 
which is not the case, since $a\neq 0$.
Since $\A$ is atomic, then there exists an atom $x\in X$ such that $x\leq a'$.
Hence using that substitutions preserve the natural order on the boolean algebras in question, 
since they are boolean endomorphisms, and axioms (7) and (8)
in the polyadic axioms, we obtain
$${\sf s}_{\tau}x\leq {\sf s}_{\tau}a'={\sf s}_{\tau}{\sf s}_{\sigma}a={\sf s}_{\tau\circ \sigma}a={\sf s}_{Id}a=a.$$
Here all substitutions are evaluated in the algebra $\A$, and we are done in this case.
Now assume that $\tau$ is not onto. Here we use the spare dimensions of $\B$. 
By a simple cardinality argument, baring in mind that $|\mathfrak{n}\sim \alpha|=|\mathfrak{n}|$, 
we can find $\bar{\tau}\in {}^{\mathfrak{n}}\mathfrak{n}$, such that 
$\bar{\tau}\upharpoonright \alpha=\tau$, and 
$\bar{\tau}$ is onto. We can also assume that $\bar{\tau}\upharpoonright (\mathfrak{n}\sim \alpha)$ is one to one since 
$|\mathfrak{n}\sim \alpha|=|\mathfrak{n}\sim Range(\tau)|=\mathfrak{n}.$
Notice also that we have $\bar{\tau}^{-1}(\mathfrak{n}\sim \alpha)\cap \alpha=\emptyset$.  
Indeed if $x\in \bar{\tau}^{-1}(\mathfrak{n}\sim \alpha)\cap \alpha$, then $\bar{\tau}(x)\notin \alpha$, while $x\in \alpha$, 
which is impossible, since $\bar{\tau}\upharpoonright \alpha=\tau$, 
so that $\bar{\tau}(x)=\tau(x)\in \alpha$.
Now let $\sigma\in {}^{\mathfrak{n}}\mathfrak{n}$ be a right inverse of $\tau$ (on $\mathfrak{n}$), so that
for chosen $j\in \bar{\tau}^{-1}(i)$, we have $\sigma(i)=j.$
We distinguish between two cases:

(i) $\sigma(\alpha)\subseteq \alpha$.
Let $a'={\sf s}_{\sigma}^{\B}a={\sf s}_{\sigma\upharpoonright \alpha}^{\A}a.$ The last equation holds because $\B$ is a dilation of $\A$.
Then $a'\in A$, since $\sigma(\alpha)\subseteq \alpha$. Also $a'\neq 0$, by same reasoning as above.
Now there exists an atom $x\in X$, such that $x\leq a'$. A fairly straightforward computation, using that $\B$ is a dilation of $\A$,
together with the fact that substitutions preserve order, and axioms 
(7) and (8)
in the polyadic axioms applied to $\B$, gives
$${\sf s}_{{\tau}}^{\A}x\leq {\sf s}_{{\tau}}^{\A}a'={\sf s}_{\bar{\tau}}^{\B}a'= {\sf s}_{\bar{\tau}}^{\B}{\sf s}_{\sigma}^{\B}a=
{\sf s}_{\bar{\tau}\circ \sigma}^{\B}a={\sf s}_{Id}^{\B}a=a,$$
which finishes the proof in the first case.

(ii) $\sigma(\alpha)$ is not contained in $\alpha.$ 

We proceed as follows. Let $a'={\sf s}_{\sigma}^{\B}a\in \B$, then by the same reasoning as above, $a'\neq 0$. 
But $a'\in \B$, hence, there exists $\Gamma\subseteq \mathfrak{n}\sim \alpha$, such that 
$a''={\sf c}_{(\Gamma)}a'\in \A,$ since $\A=\Nr_{\alpha}\B$. But  $a'\leq a''$, so $a''$ is also a non-zero element in $\A$. 
Let $x\in X$ be an atom below $a''$. Then we have
$${\sf s}_{\tau}^{\A}x\leq {\sf s}_{\bar{\tau}}^{\B}a''={\sf s}_{\bar{\tau}}^{\B}{\sf c}_{(\Gamma)}^{\B}a'=
{\sf s}_{\bar{\tau}}^{\B}{\sf c}_{(\Gamma)}^{\B}{\sf s}_{\sigma}a=
{\sf c}_{(\Delta)}^{\B}{\sf s}_{\bar{\tau}}^{\B}{\sf s}_{\sigma}^{\B}a.$$
Here $\Delta=\bar{\tau}^{-1}\Gamma$, and the last equality follows from axiom (10) in the polyadic axioms noting that
$\bar{\tau}\upharpoonright \Delta$ is one to one since  
$\bar{\tau}\upharpoonright (\mathfrak{n}\sim \alpha)$ is one to one and  
$\Delta\subseteq \mathfrak{n}\sim \alpha$. The last inclusion holds by noting that $\bar{\tau}^{-1}(\mathfrak{n}\sim \alpha)\cap \alpha=\emptyset,$ 
$\Gamma\subseteq \mathfrak{n}\sim \alpha,$ and $\Delta=\bar{\tau}^{-1}\Gamma.$
But then going on, we have also by axioms (7) and (8) of the polyadic axioms
$${\sf c}_{(\Delta)}^{\B}{\sf s}_{\bar{\tau}}^{\B}{\sf s}_{\sigma}^{\B}a= 
{\sf c}_{(\Delta)}^{\B}{\sf s}_{\bar{\tau}\circ \sigma}^{\B}a={\sf c}_{(\Delta)}^{\B}{\sf s}_{Id}^{\B}a={\sf c}_{(\Delta)}^{\B}a=a.$$
The last equality follows from the fact that $\Delta\subseteq \mathfrak{n}\sim \alpha,$ 
and $a$ is $\alpha$-dimensional, that is $a\in \A=\Nr_{\alpha}\B.$  
We have proved that ${\sf s}_{\tau}^{\A}x\leq a$, and so we are done with the second slightly more difficult case, as well.

Let $S$ be the Stone space of $\B$, whose underlying set consists of all boolean ulltrafilters of 
$\B$. Let $X^*$ be the set of principal ultrafilters of $\B$ (those generated by the atoms).  
These are isolated points in the Stone topology, and they form a dense set in the Stone topology since $\B$ is atomic. 
So we have $X^*\cap T=\emptyset$ for every nowhere dense set $T$ (since principal ultrafilters, which are isolated points in the Stone topology,
lie outside nowhere dense sets). 
For $a\in \B$, let $N_a$ denote the set of all boolean ultrafilters containing $a$.
Now  for all $\Gamma\subseteq \alpha$, $p\in B$ and $\tau\in {}^{\alpha}\mathfrak{n}$, we have,
by the suprema, evaluated in (1) and (2):
\begin{equation}\label{tarek3}
\begin{split}
G_{\Gamma,p}=N_{{\sf c}_{(\Gamma)}p}\sim \bigcup_{{\tau}\in {}^{\alpha}\mathfrak{n}} N_{s_{\bar{\tau}}p}
\end{split}
\end{equation}
and
\begin{equation}\label{tarek4}
\begin{split}
G_{X, \tau}=S\sim \bigcup_{x\in X}N_{s_{\bar{\tau}}x}.
\end{split}
\end{equation}
are nowhere dense. 
Let $F$ be a principal ultrafilter of $S$ containing $c$. 
This is possible since $\B$ is atomic, so there is an atom $x$ below $c$; just take the 
ultrafilter generated by $x$.
Then $F\in X^*$, so $F\notin G_{\Gamma, p}$, $F\notin G_{X,\tau},$ 
for every $\Gamma\subseteq \alpha$, $p\in A$
and $\tau\in {}^{\alpha}\mathfrak{n}$.
Now define for $a\in A$
$$f(a)=\{\tau\in {}^{\alpha}\mathfrak{n}: {\sf s}_{\bar{\tau}}^{\B}a\in F\}.$$
Then $f$ is a homomorphism from $\A$ to the full
set algebra with unit $^{\alpha}\mathfrak{n}$, such that $f(c)\neq 0$. We have $f(c)\neq 0$ because $c\in F,$ so $Id\in f(c)$. 
The rest can be proved exactly as in \cite{super}; the preservation of the boolean operations and substitutions is fairly 
straightforward. Preservation of cylindrifications
is guaranteed by the condition that $F\notin G_{\Gamma,p}$ for all $\Gamma\subseteq \alpha$ and all $p\in A$. (Basically an elimination 
of cylindrifications, this condition is also used in \cite{DM}
to prove the main representation result for polyadic algebras.)
Moreover $f$ is an 
atomic representation since $F\notin G_{X,\tau}$ for every $\tau\in {}^{\alpha}\mathfrak{n}$, 
which means that for every $\tau\in {}^{\alpha}\mathfrak{n}$, 
there
exists $x\in X$, such that
${\sf s}_{\bar{\tau}}^{\B}x\in F$, and so $\bigcup_{x\in X}f(x)={}^{\alpha}\mathfrak{n}.$ 
We conclude that $f$ is a complete  representation by Lemma \ref{r}.
\end{demo}
Contrary to cylindric algebras, we have:
\begin{corollary} The class of completely representable polyadic algebras of infinite dimension is elementary
\end{corollary}
\begin{demo}{Proof} Atomicity can be expressed by a first order sentence.
\end{demo}

\section{ Non elementary classes related to complete representations and neat reducts}

\begin{theorem} Let $\K$ be any of cylindric algebra, polyadic algebra, with and without equality, or Pinter's substitution algebra.
We give a unified model theoretic construction, to show the following:
\begin{enumarab}
\item For $n\geq 3$ and $m\geq 3$, $\Nr_n\K_m$ is not elementary, and $S_c\Nr_n\K_{\omega}\nsubseteq \Nr_n\K_m.$
\item For any $k\geq 5$, $\Ra\CA_k$ is not elementary
and $S_c\Ra\CA_{\omega}\nsubseteq \Ra\CA_k$. 
\end{enumarab}
\end{theorem}


In the abstract $S_c$ stands for the operation of forming {\it complete} subalgebras. 
The relation algebra part formulated in the abstract reproves a result of Hirsch in \cite{r}, and answers a question of his 
posed in {\it op.cit}.
For $\CA$ and its relatives the idea is very much like that in \cite{MLQ}, the details implemented, in each separate case, 
though are significantly distinct, because we look for terms not in the clone of operations
of the algebras considered; and as much as possible, we want these to use very little spare dimensions.

The relation algebra part is more delicate. We shall construct a relation algebra $\A\in \Ra\CA_{\omega}$ with a complete subalgebra $\B$, 
such that $\B\notin \Ra\CA_k$, and $\B$ is elementary equivalent to $\A.$ (In fact, $\B$ will be an elementary subalgebra of $\A$.)

Roughly the idea is to use an uncountable cylindric algebra in $\Nr_3\CA_{\omega}$, hence representable, and a finite atom structure
of another cylindric algebra.
We construct a finite product of the the uncountable cylindric  algebra; the product will be indexed by the atoms of the atom structure; 
the $\Ra$ reduct of the former will be as desired; it will be a full $\Ra$ reduct of an $\omega$ dimensional algebra and it has a 
complete elementary equivalent subalgebra not in 
$\Ra\CA_k$. This is the same idea for $\CA$, but in this case, and the other cases of its relatives, one spare dimension suffices.
 
This subalgebra is obtained by replacing one of the components of the product with an elementary 
{\it countable} algebra. First order logic will not see this cardinality twist, but a suitably chosen term 
$\tau_k$ not term definable in the language of relation algebras will, witnessing that the twisted algebra is not in $\Ra\CA_k$. 
For $\CA$'s and its relatives, as mentioned in the previous paragraph, we are lucky enough to have $k$ just $n+1,$
proving the most powerful result.

We concentrate on relation algebras. Let $\tau_k$ be an $m$-ary term of $\CA_k$, with $k$ large enough, 
and let $m\geq 2$ be its rank. We assume that $\tau_k$ is not definable
of relation algebras (so that $k$ has to be $\geq 5$); such terms exist.
Let $\tau$ be a term expressible in the language of relation algebras, such that
$\CA_k\models \tau_k(x_1,\ldots x_m)\leq \tau(x_1,\ldots x_m).$ (This is an implication between two first order formulas using $k$-variables).
Assume further that  whenever $\A\in {\bf Cs}_k$ (a set algebra of dimension $k$) is uncountable, 
and $R_1,\ldots R_m\in A$  are such that at least one of them is uncountable, 
then $\tau_k^{\A}(R_1\ldots R_m)$ is uncountable as well. 
(For $\CA$'s and its relatives, $k=n+1$, for $\CA$'s and $\Sc$s, it is a unary term, 
for polyadic algebras it also uses one extra dimension, it is a generalized composition hence it is a binary term.)

\begin{lemma} Let $V=(\At, \equiv_i, {\sf d}_{ij})_{i,j<3}$ be a finite cylindric atom structure, 
such that $|\At|\geq {}^33.$  
Let $L$ be a signature consisting of the unary relation 
symbols $P_0,P_1,P_2$ and
uncountably many tenary predicate symbols. 
For $u\in V$, let $\chi_u$
be the formula $\bigwedge_{u\in V}  P_{u_i}(x_i)$.  
Then there exists an $L$-structure $\M$ with the following properties:
\begin{enumarab}

\item $\M$ has quantifier elimination, i.e. every $L$-formula is equivalent
in $\M$ to a boolean combination of atomic formulas.

\item The sets $P_i^{\M}$ for $i<n$ partition $M$,

\item For any permutation $\tau$ on $3,$
$\forall x_0x_1x_2[R(x_0,x_1,x_2)\longleftrightarrow R(x_{\tau(0)},x_{\tau(1)}, x_{\tau(2)}],$
\item $\M \models \forall x_0x_1(R(x_0, x_1, x_2)\longrightarrow 
\bigvee_{u\in V}\chi_u)$, 
for all $R\in L$,

\item $\M\models  \exists x_0x_1x_2 (\chi_u\land R(x_0,x_1,x_2)\land \neg S(x_0,x_1,x_2))$ 
for all distinct tenary $R,S\in L$, 
and $u\in V.$

\item For $u\in V$, $i<3,$ 
$\M\models \forall x_0x_1x_2
(\exists x_i\chi_u\longleftrightarrow \bigvee_{v\in V, v\equiv_iu}\chi_v),$

\item For $u\in V$ and any $L$-formula $\phi(x_0,x_1,x_2)$, if 
$\M\models \exists x_0x_1x_2(\chi_u\land \phi)$ then 
$\M\models 
\forall x_0x_1x_2(\exists x_i\chi_u\longleftrightarrow 
\exists x_i(\chi_u\land \phi))$ for all $i<3$

\end{enumarab}
\end{lemma} 
\begin{demo}{Proof}  
We cannot apply Frassie's theorem to our signature, because it is uncountably infinite. What we do 
instead is that we introduce a new $4$ -ary relation symbol, that will be used to code the uncountably many tenary
relation symbols. Let $\cal L$ be the relational signature containing unary relation symbols
$P_0,\ldots, P_{3}$ and a $4$-ary relation symbol $X$. 
Let $\bold K$ be the class of all finite 
$\cal L$-structures $\D$ satsfying
\begin{enumarab}
\item The $P_i$'s are disjoint : $\forall x \bigvee_{i<j<4}(P_i(x)\land \bigwedge_
{j\neq i}\neg P_j(x)).$ 
\item $\forall x_0 x_1x_2x_3(X(x_0, x_1, x_2, x_3)\longrightarrow P_{3}(x_{3})\land \bigvee_{u\in V}\chi_u).$ 
\end{enumarab}
Then $\bold K$ contains countably many 
isomorphism types. Also it is easy to check that $\bold K$ is closed under
substructures and that $\bold K$ has the the amalgamation Property
From the latter it follows that it has the Joint Embedding Property.
Then there is a  countably infinite homogeneous $\cal L $-structure
$\N$ with age $\bold K$.  $\N$ has quantifier elimination, and obviously, so does 
any elementary extension of $\N$. 
$\bold K$ contains structures 
with arbitrarily large $P_{3}$-part, so $P_{3}^{\N}$ is infinite.
Let $\N^*$ be an elementary extension of $\N$ such that $|P_3|^{\N^*}|=|L|$, and
fix a bijection $*$ from the set of binary relation symbols of $L$ to $P_{3}^{\cal \N^*}.$
Define an $L$-structure $\M$ with domain 
$P_0^{\N^*}\cup P_1^{\N^*}\cup P_2^{\N^*}\cup \ldots P_{3}^{\N^*}$, by:
$P_i^{\M}=P_i^{\N^*}$ for $i<3$ and for
binary $R\in L$, and $\tau\in S_3$ 
$$\M\models R(\tau\circ \bar{a})\text { iff }{\N^*}\models X(\bar{a},R^*).$$ 
If $\phi(\bar {x})$ is any $L$-formula, let $\phi^*(\bar {x},\bar {R})$ be the 
$\cal L$-formula with parameters
$\bar {R}$ from $\N^*$ obtained from $\phi$ by replacing each atomic subformula 
$R(\bar{x})$ by $X(\bar{x},R^*)$ and relativizing quantifiers to $\neg P_n$, 
that is replacing $(\exists x)\phi(x)$ and $(\forall x)\phi(x)$ 
by $(\exists x)(\neg P_3(x)\to \phi(x))$
and $(\forall x)(\neg P_3(x)\to \phi(x)),$ respectively.
A straightforward induction on complexity of formulas
gives that for $\bar {a}\in \M$  
$$\M\models \phi(\bar {a})\text { iff } {\N^*}\models \phi^*(\bar {a},\bar {R}).$$

We show that $\M$ is as required. For quantifier elimination, if $\phi(\bar {x})$ is an
$L$-formula , then $\phi^*(\bar {x} ,\bar {R}^*)$ is equivalent in $\N^*$ 
to a quantifier free $\cal L$-formula
$\psi(\bar {x},\bar {R^*})$. 
Then replacing $\psi$'s atomic subformulas 
$X(x,y,z,R^*)$ by $R(x,y,z)$,
replacing all $X(t_0,\cdots t_3)$ not of this form by $\bot$ , replacing subformulas $P_3(x)$
by $\bot$, and $P_i(R^*)$ by $\bot$ if $i<3$ and $\top$ if $i=3$, gives a quantifier free $L$
-formula $\psi$ equivalent in $\M$ to $\phi$. 
(2) follows from the definition of satisfiability.
\par Let $$\sigma=\forall x(\neg P_3(x)\longrightarrow \bigvee_{i<3}
(P_i(x)\land \bigwedge_{j\neq i}\neg P_j(x))).$$
Then $K\models \sigma$, so $\M\models \sigma$ and ${\N^*}\models \sigma.$
It follows from the definition that $\M$ satisfies (3); (4) is similar.
\par For (5), let $u\in V$ and let $r,s\in P_3^{\M}$ be distinct. 
Take a finite $\cal L$-structure
$D$ with points $a_i\in P_{u_i}^D(i<3)$ and distinct $r',s'\in P_3^D$ with 
$$D\models X(a_0,a_1,a_2,r')\land \neg X(a_0,a_1,a_2,s').$$ 
Then $D\in K$, so $D$ embeds into $\M$. 
By homogeneity, we can assume that the embedding
takes $r'$ to $r$ and $s'$ to $s$. Therefore 
$${\M}\models \exists \bar {x}(\chi_u\land X(\bar {x},r)\land \neg X(\bar {x},s)),$$ 
where $\bar {x}=\langle x_0,x_1,x_2\rangle.$
Since $r,s$ were arbitrary and $\N^*$ is an elementary extension
of $\M$, we get that  
$${\N*}\models \forall yz
(P_3(y)\land P_3(z)\land y\neq z\longrightarrow \exists \bar {x}(\chi_u\land 
X(\bar {x},y)\land \neg (X(\bar {x},z))).$$
The result for $\M$ now follows.
\par Note that it follows from (4,5) that $P_i^{\M}\neq \emptyset$ for each $i<3$.
So it is clear that 
$$\M\models \forall x_0x_1x_2(\exists x_i\chi_u\longleftrightarrow 
\bigvee_{v\in {}V, v\equiv_i u}\chi_v);$$ giving
(6).

\par Finally consider $(7)$. Clearly, it is enough to show that for any
$\cal L$-formula $\phi(\bar {x})$ with parameters
$\bar {r}\in P_3^{\M}, u\in S_3, i<3$, we have 
$${\M}\models \exists \bar {x}(\chi_u\land \phi)\longrightarrow \forall \bar{x}
(\exists x_i(\chi_u\longrightarrow \exists x_i(\chi_u\land \phi)).$$
For simplicity of notation assume $i=2.$
Let $\bar {a} ,\bar {b}\in \M$ with 
$${\M}\models (\chi_u\land \phi)(\bar {a})\text { and } {\M}\models \exists x_2(\chi_u(\bar {b})).$$
We require 
$${\M}\models \exists x_2(\chi_u\land \phi)(\bar {b}).$$
\par It follows from the assumptions that
$${\M}\models P_{u_0}(a_0)\land P_{u_1}(a_1)\land a_0\neq a_1,\text { and }
{\M}\models P_{u_0}(b_0)\land P_{u_1}(b_1)\land b_0\neq b_1.$$
These are the only relations on $a_0a_r\bar {r} $ and on $b_0b_1\bar {r}$
(cf. property (4) of Lemma), so 
$$\theta^{-}=\{(a_0,b_0)(a_1,b_1) (r_l,r_l): l<|\bar {r}|\}$$ 
is a partial isomorphism
of $\M$. By homogeneity, it is induced by an automorphism $\theta$ of $\M$.
Let $c=\theta(\bar {a})=(b_0,b_1,\theta(a_2))$.
Then ${\M}\models (\chi_u\land \phi)(\bar {c}).$
Since $\bar {c}\equiv_2 \bar {b}$, we have 
${\M}\models \exists x_2(\chi_u\land \phi)(\bar {b})$
as required. 
\end{demo}

\begin{lemma}\label{term}
\begin{enumarab} 

\item For $\A\in \CA_3$ or $\A\in \SC_3$, there exist
a unary term $\tau_4(x)$ in the language of $\SC_4$ and a unary term $\tau(x)$ in the language of $\CA_3$
such that $\CA_4\models \tau_4(x)\leq \tau(x),$
and for $\A$ as above, and $u\in \At={}^33$, 
$\tau^{\A}(\chi_{u})=\chi_{\tau^{\wp(^nn)}(u).}$ 

\item For $\A\in \PEA_3$ or $\A\in \PA_3$, there exist a binary
term $\tau_4(x,y)$ in the language of $\SC_4$ and another  binary term $\tau(x,y)$ in the language of $\SC_3$
such that $PEA_4\models \tau_4(x,y)\leq \tau(x,y),$
and for $\A$ as above, and $u,v\in \At={}^33$, 
$\tau^{\A}(\chi_{u}, \chi_{v})=\chi_{\tau^{\wp(^nn)}(u,v)}.$


\item Let $k\geq 5$.
Then there exist a term $\tau_k(x_1,\ldots x_m)$ in the language of $\CA_k$ and a term $\tau(x_1,\ldots, x_m)$ in the language of $\RA$,
expressible in $\CA_3$, such that
$\CA_k\models \tau_k(x_1,\ldots x_m)\leq \tau(x_1,\ldots x_m),$
and for  $\A$ as above, and $u_1,\ldots u_m\in \At$, 
$\tau^{\A}(\chi_{u_1},\ldots \chi_{u_m})=\chi_{\tau^{\Cm\At}(u_1,\ldots u_m)}.$ 
\end{enumarab}
\end{lemma}

\begin{proof} 

\begin{enumarab}

\item For all reducts of polyadic algebras, these terms are given in \cite{FM}, and \cite{MLQ}.
For cylindric algebras $\tau_4(x)={}_3 s(0,1)x$ and $\tau(x)=s_1^0c_1x.s_0^1c_0x$.
For polyadic algebras, it is a little bit more complicated because the former term above is definable.
In this case we have $\tau(x,y)=c_1(c_0x.s_1^0c_1y).c_1x.c_0y$, and $\tau_4(x,y)=c_3(s_3^1c_3x.s_3^0c_3y)$.

\item For relation algebras, we take the term corresponding to the following generalization of Johnson's $Q$'s. 
Given $1\leq n<\omega$ and $n^2$ tenary relations
we define
$$Q(R_{ij}: i,j<n)(x_0, x_2, x_3)\longleftrightarrow$$ 
$$\exists z_0\ldots z_{n-1}(z_0=x\land z_1=y\land 
z_2=z\land \bigwedge_{i,j,l<n} 
R_{ij}(z_i, z_j, z_l).$$
{\bf Still working out the required term $\tau$}
\end{enumarab}
\end{proof}

\begin{theorem}
\begin{enumarab} 
\item There exists $\A\in \Nr_3\QEA_{\omega}$
with an elementary equivalent cylindric  algebra, whose $\SC$ reduct is not in $\Nr_3\SC_4$.
Furthermore, the latter is a complete subalgebra of the former.
 
\item There exists a relation algebra $\A\in \Ra\CA_{\omega}$, with an elementary equivalent relation algebra not in $\Ra\CA_k$.
Furthermore, the latter is a complete subalgebra of the former.
\end{enumarab}
\end{theorem}
\begin{proof} Let $\L$ and $\M$ as above. Let 
$\A_{\omega}=\{\phi^M: \phi\in \L\}.$ 
Clearly $\A_{\omega}$ is a locally finite $\omega$-dimensional cylindric set algebra.
For the first part, we prove the theorem for $\CA$; and its relatives.

Then $\A\cong \Nr_3\A_{\omega}$, the isomorphism is given by 
$$\phi^{\M}\mapsto \phi^{\M}.$$
Quantifier elimination in $\M$ guarantees that this map is onto, so that $\A$ is the full $\Ra$ reduct.

For $u\in {}V$, let $\A_u$ denote the relativisation of $\A$ to $\chi_u^{\M}$
i.e $$\A_u=\{x\in A: x\leq \chi_u^{\M}\}.$$ $\A_u$ is a boolean algebra.
Also  $\A_u$ is uncountable for every $u\in V$
because by property (iv) of the above lemma, 
the sets $(\chi_u\land R(x_0,x_1,x_2)^{\M})$, for $R\in L$
are distinct elements of $\A_u$.  

Define a map $f: \Bl\A\to \prod_{u\in {}V}\A_u$, by
$$f(a)=\langle a\cdot \chi_u\rangle_{u\in{}V}.$$

Here, and elsewhere, for a relation algebra $\C$, $\Bl\C$ denotes its boolean reduct.
We will expand the language of the boolean algebra $\prod_{u\in V}\A_u$ by constants in 
such a way that
the relation algebra reduct of $\A$ becomes interpretable in the expanded structure.
For this we need.

Let $\P$ denote the 
following structure for the signature of boolean algebras expanded
by constant symbols $1_{u}$ for $u\in V$ 
and ${\sf d}_{ij}$ for $i,j\in 3$: 
We now show that the relation algebra reduct of $\A$ is interpretable in $\P.$  
For this it is enough to show that 
$f$ is one to one and that $Rng(f)$ 
(Range of $f$) and the $f$-images of the graphs of the cylindric algebra functions in $\A$ 
are definable in $\P$. Since the $\chi_u^{\M}$ partition 
the unit of $\A$,  each $a\in A$ has a unique expression in the form
$\sum_{u\in {}V}(a\cdot \chi_u^{\M}),$ and it follows that 
$f$ is boolean isomorphism: $bool(\A)\to \prod_{u\in {}V}\A_u.$
So the $f$-images of the graphs of the boolean functions on
$\A$ are trivially definable. 
$f$ is bijective so $Rng(f)$ is 
definable, by $x=x$. For the diagonals, $f({\sf d}_{ij}^{\A})$ is definable by $x={\sf d}_{ij}$.

Finally we consider cylindrifications for $i<3$. Let $S\subseteq {}V$ and  $i<3$, 
let $t_S$ be the closed term
$$\sum\{1_v: v\in {}V, v\equiv_i u\text { for some } u\in S\}.$$
Let
$$\eta_i(x,y)=\bigwedge_{S\subseteq {}V}(\bigwedge_{u\in S} x.1_u\neq 0\land 
\bigwedge_{u\in {}V\smallsetminus S}x.1_u=0\longrightarrow y=t_S).$$
We claim that for all $a\in A$, $b\in P$, we have 
$$\P\models \eta_i(f(a),b)\text { iff } b=f({\sf c}_i^{\A}a).$$
To see this, let $f(a)=\langle a_u\rangle_{u\in {}V}$, say. 
So in $\A$ we have $a=\sum_ua_u.$
Let $u$ be given; $a_u$ has the form $(\chi_i\land \phi)^{\M}$ for some $\phi\in L^3$, so
${\sf c}_i^A(a_u)=(\exists x_i(\chi_u\land \phi))^{\M}.$ 
By property (vi), if  $a_u\neq 0$, this is 
$(\exists x_i\chi_u)^M$; by property $5$,
this is $(\bigvee_{v\in {}V, v\equiv_iu}\chi_v)^{\M}.$
Let $S=\{u\in {}V: a_u\neq 0\}.$
By normality and additivity of cylindrifications we have,
$${\sf c}_i^A(a)=\sum_{u\in {}V} {\sf c}_i^Aa_u=
\sum_{u\in S}{\sf c}_i^Aa_u=\sum_{u\in S}(\sum_{v\in {}V, v\equiv_i u}\chi_v^{\M})$$
$$=\sum\{\chi_v^{\M}: v\in {}V, v\equiv_i u\text { for some } u\in S\}.$$
So $\P\models f({\sf c}_i^{\A}a)=t_S$. Hence $\P\models \eta_i(f(a),f({\sf c}_i^{\A}a)).$
Conversely, if $\P\models \eta_i(f(a),b)$, we require $b=f({\sf c}_ia)$. 
Now $S$ is the unique subset of $V$ such that 
$$\P\models \bigwedge_{u\in S}f(a)\cdot 1_u\neq 0\land \bigwedge_{u\in {}V\smallsetminus S}
f(a)\cdot 1_u=0.$$  So we obtain 
$$b=t_S=f({\sf c}_i^Aa).$$
The rest is the same as in \cite{MLQ} while for other relatives, the idea implemented is also the same;
one just uses the corresponding terms as in lemma \ref{term}.

For relation algebras we proceed as follows:
Now the $\Ra$ reduct of $\A$ is a generalized reduct of $\A$, hence $\P$ is first order interpretable in $\Ra\A$, as well.
 It follows that there are closed terms $1_{u, v},d_{i,j}$ and a formula $\eta$ built out of these closed terms such that 
$$\P\models \eta(f(a), b, c)\text { iff }b= f(a\circ c),$$
where the composition is taken in $\Ra\A$.

We have proved that $\Ra\A$ is interpretable in $\P$.
Furthermore it is easy to see that the interpretation 
is two dimensional and quantifier free.

For each $u\in V$, choose any countable boolean elementary 
complete subalgebra of $\A_{u}$, $\B_{u}$ say.
Le $u_i: i<m$ be elements in $V$ and let
$$Q=(\prod_{u_i: i<m}\A_{u_i}\times \B_{\tau^{\Cm\At}(u_1,\ldots u_m)}\times \prod_{u\in {}V\smallsetminus \{u_1,\ldots u_m, \tau^{\Cm\At\A}(u_1,\ldots u_m)\}} 
\A_u), 1_{u,v},d_{ij})_{u,v\in {}V,i,j<3}\equiv$$  
$$(\prod_{u\in V} \A_u, 1_{u,v}, {\sf d}_{ij})_{u\in V, i,j<3}=P.$$

Let $\B$ be the result of applying the interpretation given above to $Q$.
Then $\B\equiv \Ra\A$ as relation  algebras, furthermore $\Bl\B$ is a complete subalgebra of $\Bl\A$. 
Assume for contradiction that $\B=\Ra\D$ with $\D\in \CA_k$. Let $u_1,\ldots u_m\in V$ be such that 
$\tau_k^{\D}(\chi_{u_1},\ldots \chi_{u_n})$, is uncountable in $\D$. 

Because $\B$ is a full $\RA$ reduct, 
this set is contained in $\B.$ 

For simplicity assume that $\tau^{\Cm\At}(u_1\ldots u_m)=Id.$
On the other hand for $x_i\in B$, with $x_i\leq \chi_{u_i}$, we have 
$$\tau_k^{\D}(x_1,\ldots x_m)\leq \tau(x_1\ldots x_m)\in \tau(\chi_{u_1},\ldots {\chi_{u_m}})=\chi_{\tau(u_1\ldots u_m)}=\chi_{Id}.$$

But this is a contradiction, since  $B_{Id}=\{x\in B: x\leq \chi_{Id}\}$ is  countable.

\end{proof}

\section{Weakly representable atom structures that are not strongly representable}

There are two types of constructions of atom structures as in the title, the Andreka-N\'emeti construction,  
and the Hodkinson's style (model theoretic) style construction.
In the first, for any $n\geq 3$, for any pre assigned $k\in \omega$,  
a relation algeba atom structure which has an $n$ dimensional cylindric basis, 
such that the term algebra over the
cylindric atom structure of  the basic matrices is representable, and is in $\Nr_nCA_{n+k},$ but the complex algebra is not. 
Furthermore, the relation algebra atom structure itself is only weakly representable. 

In the second construction, one move backwards.
Hodkinson constructs, by model theoretic methods that are not too difficult,
an atomic representable cylindric algebra, whose atom structure is isomorphic to the cylindric algebra consisting 
of basic matrices over a relation atom structure. 
Both structures are again only weakly representable.

In the first case, we will simply the construction. In this second case, we will also simplify the (rainbow) construction, which we did before in \cite{w},
but our real gain here , is that the complex algebra over the cylindric atom structure, consisting of basic matrices,
which is the completion of $\A$, is not only not representable, but in fact it is not in $S\Nr_n\CA_{n+k}$ for any $k\in \omega$.
From this we show that each of the varieties$S\Nr_n\CA_{n+k}$ ($k\geq 2)$ is not closed under completions.

\subsection{First construction}

Here we give a simplified version of the construction in \cite{1}. Also several new consequences are implemented.
Basically this is the same argument used in opcit, except that is considerably simplified and hence is more streamlined, and we believe easier to 
grasp. Indeed the idea underlying some construction maybe simple, 
when they first occur to the reserach, but then trying to get the most general result, 
things become apprently much more complicated, though in esence the idea is the same.
Another thing is the first solution to a problem is usualy not the simplest.

Throughout, $H$ is a fixed  
infinite set on which we will impose necessary conditions as we proceed.  In fact,  
$H$ denotes the set of non-identity atoms in our future 
algebras  (i.e. all the algebras we construct in a while.) 
$$F=H\cup \{Id\}$$ is the universe
of our intended future atom structure $\cal F.$ 


Let $I$ be a finite set with $|I|\geq 6$. Let $J$ be the set of all $2$ element
subsets of $I$, and let
$$H=\{a_i^{P,W}: i\in \omega, P\in I, W\in J, P\in W\}.$$
For $P\in I$, let 
$$H^P=\{a_i^{P,W}: i\in \omega, W\in J, P\in W\}.$$
For $W\in J,$ let
$$E^W=\{a_i^{P,W}: i\in \omega, P\in W\}.$$
This way we have defined our {\it two} partitions of $H$.
For $i,j,k\in \omega$ $e(i,j,k)$ abbreviates that $i,j,k$ are {\it evenly distributed}, i.e.
$$e(i,j,k)\text { iff } (\exists p,q,r)\{p,q,r\}=\{i,j,k\}, r-q=q-p$$
For example $3,5,7$ are evenly distributed, but $3,5,8$ are not. 

All atoms are self-converse. We define the consistent triples as follows
Let $i,j,k\in \omega$, $P,Q,R\in I$ and $S,Z,W\in J$ such that
$P\in S$, $Q\in Z$ and $R\in W$. Then the triple
$(a_i^{P,S},a_j^{Q,Z}, a_k^{R,W})$ is consistent iff 
either 
\begin{enumroman}
\item $S\cap Z\cap W=\emptyset,$
or 
\item $e(i,j,k)\&\{P,Q,R\}|\neq 1.$
\end{enumroman}
It is easy to check that $e$ hence $\bold T$ is symmetric. 
Therefore $;$ satisfies the triangle
rule. Now choose a (finite) relation algebra $\bold M$ with atoms $I\cup \{1d\}$ such that
for all $P,Q\in I$, $P\neq Q$ we have
$$P;P=\{Q\in I: Q\neq P\}\cup \{Id\}\text { and } P;Q=H$$
Such an $\bold M$ exists. Such algebras are constructed by Maddux. It is known that
$\bold M$ cannot be represented on finite sets.
\begin{enumroman}
\item $\Cm{\cal F}$ is a relation algebra that is not representable.
\item ${\R}$ the term algebra over $\cal F$ is representable.
In other words, $\cal F$ is an example of a weakly representable atom structure
that is not strongly representable.
\end{enumroman}
\begin{proof} 
\begin{enumarab}
\item Non representabiliy uses the first partition of $H$. Note that $;$ is defined on $\Cm({\cal F})$ so  that 
$$H^P; H^Q=\bigcup \{H^Z: Z\leq P; Q\in {\bold M}\}.$$
So $\bold M$ is isomorphic to a subalgebra of $Cm F$. But $\Cm F$ 
can only be represented o infinite sets, while $\M$ only on finite ones, hence we are done.

\item The representability of the term algebra uses the second partition. 
It can be checked that $R=\{X\subseteq F: X\cap E^W\in Cof(E^W), \forall W\in J\}$. 
For any $a\in F$ and $W\in J$, let
$$U^a=\{X\in R: a\in X\}$$
and
$$U^W=\{X\in R: |Z\cap E^W|\geq \omega\}$$
Let
$$Uf=\{U^a: a\in F\}\cup \{U^W: W\in J: |E^W|\geq \omega \}$$.
$Uf$ denotes the set of ultrafilters of $\cal R$, that include at least one non-principal 
ultrafilter, that is an element of the form $U^W$.  
We call $(G,l)$ a  {\it consistent coloured graph} if $G$ is a set , 
$l:G\times G\to Uf$
such that for all $x,y,z\in G$, the following hold:

\begin{enumroman}

\item $l(x,y)=U^{Id}$ iff $x=y,$

\item $l(x,y)=l(y,x)$

\item The triple $(l(x,y),l(x,z),l(y,z))$ is consistent.

\end{enumroman}

We say that a consistent coloured graph $(G,l)$ is complete if for all 
$x,y\in G$, and $F,K\in Uf$, whenever 
$(l(x,y),F,K)$ is consistent, then there is a node $z$ such that
$l(z,x)=F$ and $l(z,y)=K$.
We will build a complete consistent graph step-by-step. So assume (inductively) 
that$(G,l)$ is a consistent coloured graph and $(l(x,y), F, K)$ is a consistent triple.
We shall extend $(G,l)$ with a new point $z$ such that $(l(x,y), l(z,x), l(z,y)) =(l(x,y),G,K).$
Let $z\notin G.$ We define $l(z,p)$ for $p\in G$ as follows:
$$l(z,x)=F$$
$$l(z,y)=K, \text { and if } p\in G\smallsetminus \{x,y\}, \text { then }$$
$$l(z,p)=U^W\text { for some }W\in J' \text { such that both }$$
$$(U^W, F, l(x,p))\text { and } (U^W, K, l(y,p))\text { are consistent }.$$
Such a $W$ exists by our assumptions (i)-(iii).
Conditions (i)-(ii) guarantee that this extension is again a consistent coloured graph.

We now show that any non-empty complete coloured graph $(G,l)$ gives a representation 
for $\cal R.$ For any $X\in R$ define
$$rep(X)=\{(u,v)\in G\times G: X\in l(u,v)\}$$
We show that 
$$rep:{\cal R}\to R(G)$$ 
is an embedding. $rep$ is a boolean homomorphism
because all the labels are ultrafilters. 
$$rep(Id)=\{(u,u): u\in G\},$$  
and for all $X\in R$,
$$rep(X)^{-1}=rep(X).$$
The latter follows from the first condition in the definition of a consistent coloured graph.
From the second condition in the definition of a consistent coloured graph, we have:
$$rep(X);rep(Y)\subseteq rep(X;Y).$$
Indeed, let $(u,v)\in rep(X), (v,w)\in rep(Y)$ I.e. $X\in l(u,v), Y\in l(v,w).$
Since $(l(u,v), l(v,w), l(u,w))$ is consistent, then $X;Y\in l(u,w)$, i.e. $(u,w)\in 
rep(X;Y).$
On the other hand, since $(G,l)$ is complete and because (i)-(ii) hold, we have:
$$rep(X;Y)\subseteq rep(X);rep(Y),$$
because $(G,l)$ is complete and because (i) and (ii) hold.
Indeed, let $(u,v)\in rep(X;Y)$. Then $X;Y\in l(u,v)$. 
We show that there are $F,K\in Uf$ such that
$$X\in F, Y\in K \text { and } (l(u,v), F,K)\text { is consistent }.$$
We distinguish between two cases:

{\bf Case 1}. $l(u,v)=U^a$ for some $a\in F$. By $X;Y\in U^a$ we have $a\in X;Y.$
Then there are $b\in X$, $c\in Y$ with $a\leq b;c.$ Then $(U^a, U^b, U^c)$ is consistent.

{\bf Case 2.} 
$l(u,v)=U^W$ for some $W\in J'$. Then $|X;Y\cap E^W|\geq \omega$ 
by $X;Y\in U^W$.
Now if both $X$ and $Y$ are finite, then there are $a\in X$, $b\in Y$ with 
$|a;b\cap E^W|\geq \omega$.
Then $(U^W, U^a, U^b)$ is consistent by (i). Assume that one of $X,Y$, say $X$ is infinite.
Let $S\in J'$ such that $|X\cap E^S|\geq \omega$ and let $a\in Y$ be arbitrary. 
Then $(U^W, U^S, U^a)$ is consistent by (ii)
and $X\in U^S, Y\in U^a.$

Finally, $rep$ is one to one because $rep(a)\neq \emptyset$ for all $a\in A$. 
Indeed $(u,v)\in rep(Id)$ for any
$u\in G$. Let $a\in H$. Then $(U^{Id}, U^a, U^a)$ is consistent, so there is a $v\in G$ with 
$l(u,v)=U^a$. Then $(u,v)\in rep(a).$

\end{enumarab}

\end{proof}


We define the atom structure like we did before. The basic matrices of the atom structure above form a 
$3$ dimensional cylindric algebra. We want an 
$n$ dimensional one.  Our previous construction of the atom structure
satisfies $(\forall a_1\ldots a_n3_1\ldots b_3\in H)(\exists W\in J')(a_1;b_1)\cap (a_3;b_3)\in U^W.$
We strengthen this condition to 
$$(\forall a_1\ldots a_nb_1\ldots b_n\in H)(\exists W\in J')(a_1;b_1)\cap (a_n;b_n)\in U^W.$$
Accordingly instead of $H$ we write $H_n$. 
This condition will entail that the set of {\it all} $n$ by $n$ 
basic matrices of ${\R}_n$ forms a cylindric bases as we proceed
to show:
Let $n>2$ be given. Let $I$ be a finite set with $|I|\geq 2n+2$. 
Let $J$ be the set of all $2$ element
subsets of $I$, and let
$$H_n=\{a_j^{P,W}: j\in \omega, P\in I, W\in J, P\in W\}.$$
For $P\in I$, let 
$$H^P=\{a_j^{P,W}: j\in \omega, W\in J, P\in W\}.$$
For $W\in J,$ let
$$E^W=\{a_j^{P,W}: j\in \omega, P\in W\}.$$
This way we have defined our two partitions of $H_n$.
For $i,j,k\in \omega$ $e(i,j,k)$ abbreviates that $i,j,k$ are {\it evenly distributed}, i.e.
$$e(i,j,k)\text { iff } (\exists p,q,r)\{p,q,r\}=\{i,j,k\}, r-q=q-p$$

\begin{definition}
We define the consistent triples as before
Let $i,j,k\in \omega$, $P,Q,R\in I$ and $S,Z,W\in J$ such that
$P\in S$, $Q\in Z$ and $R\in W$. Then we set
$(a_i^{P,S},a_j^{Q,Z}, a_k^{R,W})$ is consistent iff 
\begin{enumarab}
\item $S\cap Z\cap W=\emptyset$,
or
\item $e(i,j,k)\text {and }|\{P,Q,R\}|\neq 1.$
\end{enumarab}
\end{definition}

Now we have:

Let $n>2$ be finite. Let 
$H_n$ be as specified above and let ${\R}_n$ be the relation algebra 
based
on $F=H_n\cup \{Id\}$. Then the following hold:
\begin{athm}{Theorem 2}
Suppose that $n$ is a finite ordinal with $n>2$ and $k\geq 0$.
There is a countable 
symmetric integral representable 
relation algebra ${\R}$
such that
\begin{enumroman}
\item Its completion, i.e. the complex algebra of its atom structure is 
not representable, so $\R$ is representable but not completely representable 
\item $\R$ is generated by a single element.
\item The (countable) set $\B_n{\R}$ of all $n$ by $n$ basic matrices over $\R$ 
constitutes an $n$-dimensional cylindric basis. 
Thus $\B_n{\R}$ is a cylindric atom structure 
and the full complex algebra $\Cm(\B_n{\R})$ 
with universe the power set of $\B_n{\R}$
is an $n$-dimensional cylindric algebra 
\item The {\it term algebra} over the atom structure 
$\B_n{\R}$, which is the countable subalgebra of $\Cm(\B_n{\R})$ 
generated by the countable set of 
$n$ by $n$ basic matrices, $\C=\Tm(B_n \R)$ for short,
is a countable representable $\CA_n$, but $\Cm(\B_n)$ is not representable. 
\item Hence $\C$ is a simple, atomic representable but not completely representable $\CA_n$
\item $\C$ is generated by a single $2$ dimensional element $g$, the relation algebraic reduct
of $\C$ does not have a complete representation and is also generated by $g$ as a relation algebra, and 
$\C$ is a sub-neat reduct of some simple representable $\D\in \CA_{n+k}$ 
such that the relation algebraic reducts of $\C$ and $\D$
coincide.  
\end{enumroman}
\end{athm}


Here we include more examples. 

\begin{example}

Let $l\in \omega$, $l\geq 2$, and let $\mu$ be a non-zero cardinal. Let $I$ be a finite set,
$|I|\geq 3l.$ Let 
$$J=\{(X,n): X\subseteq I, |X|=l,n<\mu\}.$$
Let $H$ be as before, i.e.
$$H=\{a_i^{P,W}: i\in \omega, P\in I, W\in J\}.$$
Define
$(a_i^{P,S,p}, a_j^{Q,Z,q}, a_k^{R,W,r})$ is consistent ff 

$S\cap Z\cap W=\emptyset$ or 
$ e(i,j,k) \text { and } |\{P,Q,R\}|\neq 1.$

Pending on $l$ and $\mu$, let us call these atom structures ${\cal F}(l,\mu).$
Then the example in section 2 is just 
${\cal F}(2,1).$

If $\mu\geq \omega$, then $J$ as defined in section 2 would be infinite, 
and $Uf$ will be a proper subset of the ultrafilters.
It is not difficult to show that if $l\geq \omega$ 
(and we relax the condition that $I$ be finite), then
$\Cm{\cal F}(l,\mu)$ is completely representable, 
and if $l<\omega$ then $Cm{\cal F}(l,\mu)$ is non-representable.


\end{example}

\begin{corollary}
\begin{enumarab}
\item  The classes $\RRA$ is not finitely axiomatizable.
\item  The elementary 
closure of the class ${\bf CRA}$ is not finitely axiomatizable.
\end{enumarab}
\end{corollary}
\begin{demo}{Proof} 
Let ${\cal D}$ be a non-
trivial ultraproduct of ${\cal F}(i,1)$, $i\in \omega$. Then $Cm{\cal D}$
is completely representable. 
Thus ${\R}(i,1)$ are $RRA$'s 
without a complete representation while their ultraproduct has a complete representation.
Also $Cm{\cal D}(i,1)$ $i\in \omega$ 
are non representable with a completely representable ultraproduct.
This yields the desired result.
\end{demo}
We note that a variation on the theme, usiraltion atom structures with $n$ dimensional cylindric bases,
yields the analogous result for cylindric algebras.
In more detail

\begin{corollary} Let $2<n<\omega$
\begin{enumarab}
\item  The classes $\RCA_n$ is not finitely axiomatizable.
\item  The elementary 
closure of the class  of completely representable $\CA_n$ 
is not finitely axiomatizable.
\end{enumarab}
\begin{demo} Let $n\geq 3$.  For $k\geq 0$, let $\B_k\in RCA_n\cap \Nr_n\CA_{n+k}$ 
such that $\B$ has no complete representation. Then $\B_k=\Nr_n\C_k$, with  $C_k\in \CA_{n+k}$. 
Let $C_n^+\in \CA_{\omega}$ 
be such that $\Rd_{n+k}\C_k^+=C_k$. Then $\B=\prod \B_k/F=\prod Nr_nC_k/F=\Nr_n\prod C_k^+/F$. 
Then $\B$ is completely representable. 
\end{demo}

\end{corollary}


\subsection{Completions of sub neat reducts}

We use a construction of Hirsch and Hodkinson \cite{HHbook}. We follow the notation in op.cit. 
$\Ra\CA_n$ stands for the class of relation algebra reducts of $\CA_n$.
It is known that $\RCA_n={\bf S}\Nr_n\CA_{\omega}$ for any finite $n$ and that for $k\in \omega$ and $n>2$, 
$\RCA_n\subset {\bf S}\Nr_n\CA_{n+k+1}\subset 
{\bf S}\Nr_n\CA_{n+k}$ \cite{HHbook}. 

Let $\RA_n$ be the class of subalgebras of atomic relation algebras having $n$ dimensional relational basis.
Then ${\bf S}\Ra\CA_n\subseteq \RA_n$ \cite{HHbook}. 
The full complex algebra of an atom structure $S$
will be denoted by $\Cm S$, and the term algebra by $\Tm S.$
$S$ could be a relation atom structure or a cylindric atom structure.

\begin{theorem} Let $n\geq 3$. Assume that for any simple atomic relation algebra $\cal A$ with atom structure $S$, 
there is a cylindric atom structure $H$ such that:
\begin{enumarab}
\item If $\Tm S\in \RRA$, then $\Tm H\in \RCA_n$.
\item $\Cm S$ is embeddable in $\Ra$ reduct of $\Cm H$.
\end{enumarab}
Then for all $k\geq 3$, $S\Nr_n\CA_{n+k}$ is not closed under completions.

\end{theorem}
\begin{demo}{Proof} Let $S$ be a relation atom structure such that $\Tm S$ is representable while $\Cm S\notin \RA_6$.
Such an atom structure exists \cite{HHbook} Lemmas 17.34-17.36 and are finite. 
It follows that $\Cm S\notin {\bf S}\Ra\CA_n$.
Let $H$ be the $\CA_n$ atom structure provided by the hypothesis of the previous theorem.  
Then $\Tm H\in \RCA_n$. We claim that $\Cm H\notin {\bf S}\Nr_n\CA_{n+k}$, $k\geq 3$. 
For assume not, i.e. assume that $\Cm H\in {\bf S}\Nr_n\CA_{n+k}$, $k\geq 3$.
We have $\Cm S$ is embeddable in $\Ra\Cm H.$  But then the latter is in ${\bf S}\Ra\CA_6$
and so is $\Cm S$, which is not the case.
 \end{demo}
 
Monk and Maddux constructs such an $H$ for $n=3$ and Hodkinson constructs  an $H,$ but $H$ does not satisfy (2).

Now we prove, in a different way that classes of subneat reducts are not closed under completions. the condition theorem above. 
We use the model theoretic methods in \cite{Hodkinson} and \cite{w}, 
to construct strongly representable atom structures that are not weakly representable.
Our chosen graph is yet simpler than the one used in \cite{w}, and we prove an even  stronger result namely that for each $k\geq 2$
the variety $S\Nr_n\CA_{n+k}$ is not closed under completions, hence is it also not atom-canonical.
This answers a question of Hirsch and Hodkinson. When the proofs are completely analogous to profs in \cite{w} 
(the esence of the construction is the same) we refer to \cite{w}.

In this section by a graph we shall mean a pair $\G=(G,E)$ where $G$ is a non-empty set of vertices 
and $E\subseteq G\times G$ is an irreflexive and
symmetric binary relation on $G$. A pair $(x,y)\in E$ is called an edge of $\G$. Fix finite $N\geq n(n-1)/2$. 
 $\G$ will denote the graph $\G=(\N,E)$ with nodes $\N$ and $i,l$ is an edge i.e $(i,l)\in E$ if  
$0<|i-l|<N$. The graph used in \cite{w} is an countable union of cliques, denote by $N\times \omega$. 

\begin{definition}
A \textit{labelled graph} is an undirected graph $\Gamma$ such that
every edge ( \textit{unordered} pair of distinct nodes ) of $\Gamma$
is labelled by a unique label from $(\G \cup \{\rho\}) \times n$, where
$\rho \notin \G$ is a new element. The colour of $(\rho, i)$ is
defined to be $i$. The \textit{colour} of $(a, i)$ for $a \in \G$ is
$i$.
\end{definition}

We will write $ \Gamma (x, y)$ for the label of an edge $ (x, y)$ in
the labelled graph $\Gamma$. Note that these may not always be
defined: for example, $ \Gamma (x, x)$ is not.\\
If $\Gamma$ is a labelled graph, and $ D \subseteq \Gamma$, we write
$ \Gamma \upharpoonright D$ for the induced subgraph of $\Gamma$ on
the set $D$ (it inherits the edges and colours of $\Gamma$, on its
domain $D$). We write $\triangle \subseteq \Gamma$ if $\triangle$ is
an induced subgraph of $\Gamma$ in this sense.

\begin{definition}
Let $ \Gamma, \triangle$ be labelled graphs, and $\theta : \Gamma
\rightarrow \triangle$ be a map. $\theta$ is said to be a
\textit{labelled graph embedding}, or simple an \textit{embedding},
if it is injective and preserves all edges, and all colours, where
defined, in both directions. An \textit{isomorphism} is a bijective
embedding.
\end{definition}

Now we define a class $\GG$ of certain labelled graphs.

\begin{definition}
The class $\GG$ consists of all complete labelled graphs $\Gamma$ (possibly
the empty graph) such that for all distinct $ x, y, z \in \Gamma$,
writing $ (a, i) = \Gamma (y, x)$, $ (b, j) = \Gamma (y, z)$, $ (c,
l) = \Gamma (x, z)$, we have:\\
\begin{enumarab}
\item $| \{ i, j, l \} > 1 $, or
\item $ a, b, c \in \G$ and $ \{ a, b, c \} $ has at least one edge
of $\G$, or
\item exactly one of $a, b, c$ -- say, $a$ -- is $\rho$, and $bc$ is
an edge of $\G$, or
\item two or more of $a, b, c$ are $\rho$.
\end{enumarab}
Clearly, $\GG$ is closed under isomorphism and under induced
subgraphs.
\end{definition}

\begin{theorem}
There is a countable labelled graph $M\in \GG$ with the following
property:\\
$\bullet$ If $\triangle \subseteq \triangle' \in \GG$, $|\triangle'|
\leq n$, and $\theta : \triangle \rightarrow M$ is an embedding,
then $\theta$ extends to an embedding $\theta' : \triangle'
\rightarrow M$.
\end{theorem}
\begin{proof} \cite{Hodkinson}, \cite{w} prop. 3.5
\end{proof}
Let $L^n$ denote the first-order fragment of $L^n_{\infty \omega}$

\begin{definition}
An $n$-\textit{back-and-forth system} on $A$ is a set $\Theta$ of
one-to-one partial maps : $ A \rightarrow A$ such that:\\
\begin{enumerate}
\item if $\theta \in \Theta$ then $|\theta| \leq n$
\item if $ \theta' \subseteq \theta \in \Theta$ then $\theta' \in \Theta$
\item if $\theta \in \Theta$, $|\theta| \leq n$, and $ a \in A$,
then there is $ \theta' \supseteq \theta$ in $\Theta$ with $ a \in
dom(\theta')$ (forth)
\item if $\theta \in \Theta$, $|\theta| \leq n$, and $ a \in A$,
then there is $ \theta' \supseteq \theta$ in $\Theta$ with $ a \in
rng(\theta')$ (back).

\end{enumerate}
\end{definition}

 Recall that a \textit{partial isomorphism} of $A$ is a
partial map $ \theta : A \rightarrow A$ that preserves all
quantifier-free $L$-formulas.
The next theorem is a well known result in model theory.
\begin{theorem} 
\begin{enumarab}
\item Let $\Theta$ be an $n$-back-and-forth system
of partial isomorphism on $A$, let $\bar{a}, \bar{b} \in {}^{n}A$,
and suppose that $ \theta = ( \bar{a} \mapsto \bar{b})$ is a map in
$\Theta$. Then $ A \models \phi(\bar{a})$ iff $ A \models
\phi(\bar{b})$, for any formula $\phi$ of $L^n_{\infty \omega}$.

\item If $W$ is $L^n_{\infty \omega}$ definable, $\Theta$ is an
 $n$-\textit{back-and-forth} system
of partial isomorphisms on $A$, $\bar{a}, \bar{b} \in W$, and $
\bar{a} \mapsto \bar{b} \in \Theta$, then $ A \models \phi(\bar{a})$
iff $ A \models \phi(\bar{b})$ for any formula $\phi$ of
$L^n_{\infty \omega}$.
\end{enumarab}
\end{theorem}
\begin{proof} \cite{w} fact 4.5, corollary 4.6
\end{proof}
\begin{definition} Let $L^+$ be the signature consisting of the binary
relation symbols $(a, i)$, for each $a \in \G \cup \{ \rho \}$ and
$ i < n$. Let $L = L^+ \setminus \{ (\rho, i) : i < n \}$. From now
on, the logics $L^n, L^n_{\infty \omega}$ are taken in this
signature.
\end{definition}

We may regard any non-empty labelled graph equally as an
$L^+$-structure, in the obvious way. The $n$-homogeneity built into
$M$ by its construction would suggest that the set of all partial
isomorphisms of $M$ of cardinality at most $n$ forms an
$n$-back-and-forth system. This is indeed true, but we can go
further.

\begin{definition}
Let $\chi$ be a permutation of the set $\omega \cup \{ \rho\}$. Let
$ \Gamma, \triangle \in \GG$ have the same size, and let $ \theta :
\Gamma \rightarrow \triangle$ be a bijection. We say that $\theta$
is a $\chi$-\textit{isomorphism} from $\Gamma$ to $\triangle$ if for
each distinct $ x, y \in \Gamma$,
\begin{itemize}
\item If $\Gamma ( x, y) = (a, j)$ with $a\in \N$, then there exist unique $l\in \N$ and $r$ with $0\leq r<N$ such that
$a=Nl+r$. 
\begin{equation*}
\triangle( \theta(x),\theta(y)) =
\begin{cases}
( N\chi(i)+r, j), & \hbox{if $\chi(i) \neq \rho$} \\
(\rho, j),  & \hbox{otherwise.} \end{cases}
\end{equation*}
\end{itemize}

\begin{itemize}
\item If $\Gamma ( x, y) = ( \rho, j)$, then
\begin{equation*}
\triangle( \theta(x),\theta(y)) \in
\begin{cases}
\{( N\chi(\rho)+s, j): 0\leq s < N \}, & \hbox{if $\chi(\rho) \neq \rho$} \\
\{(\rho, j)\},  & \hbox{otherwise.} \end{cases}
\end{equation*}
\end{itemize}
\end{definition}

\begin{definition}
For any permutation $\chi$ of $\omega \cup \{\rho\}$, $\Theta^\chi$
is the set of partial one-to-one maps from $M$ to $M$ of size at
most $n$ that are $\chi$-isomorphisms on their domains. We write
$\Theta$ for $\Theta^{Id_{\omega \cup \{\rho\}}}$.
\end{definition}

\begin{lemma}
For any permutation $\chi$ of $\omega \cup \{\rho\}$, $\Theta^\chi$
is an $n$-back-and-forth system on $M$.
\end{lemma}
\begin{demo}{Proof}
The proof is similar to theorem 4.11 in \cite{w}; however, the graph used is different, so we include the proof.
Clearly, $\Theta^\chi$ is closed under restrictions. We check the
``forth" property. Let $\theta \in \Theta^\chi$ have size $t < n$.
Enumerate $ dom(\theta)$, $rng(\theta)$ respectively as $ \{ a_0,
\ldots, a_{t-1} \}$, $ \{ b_0,\ldots b_{t-1} \}$, with $\theta(a_i)
= b_i$ for $i < t$. Let $a_t \in M$ be arbitrary, let $b_t \notin M$
be a new element, and define a complete labelled graph $\triangle
\supseteq M \upharpoonright \{ b_0,\ldots, b_{t-1} \}$ with nodes
$\{ b_0,\ldots, b_{t} \}$ as follows.\\

Choose distinct "nodes"$ e_s < N$ for each $s < t$, such that no
$(e_s, j)$ labels any edge in $M \upharpoonright \{ b_0,\dots,
b_{t-1} \}$. This is possible because $N \geq n(n-1)/2$, which
bounds the number of edges in $\triangle$. We can now define the
colour of edges $(b_s, b_t)$ of $\triangle$ for $s = 0,\ldots, t-1$.

\begin{itemize}
\item If $M ( a_s, a_t) = ( Ni+r, j)$, for some $i\in \N$ and $0\leq r<N$, then
\begin{equation*}
\triangle( b_s, b_t) =
\begin{cases}
( N\chi(i)+r, j), & \hbox{if $\chi(i) \neq \rho$} \\
\{(\rho, j)\},  & \hbox{otherwise.} \end{cases}
\end{equation*}
\end{itemize}

\begin{itemize}
\item If $M ( a_s, a_t) = ( \rho, j)$, then assuming that $e_s=Ni+r,$ $i\in \N$ and $0\leq r<N$,
\begin{equation*}
\triangle( b_s, b_t) =
\begin{cases}
( N\chi(\rho)+r, j), & \hbox{if $\chi(\rho) \neq \rho$} \\
\{(\rho, j)\},  & \hbox{otherwise.} \end{cases}
\end{equation*}
\end{itemize}

This completes the definition of $\triangle$. It is easy to check 
that $\triangle \in \GG$. Hence, there is a graph embedding $ \phi : \triangle \rightarrow M$
extending the map $ Id_{\{ b_0,\ldots b_{t-1} \}}$. Note that
$\phi(b_t) \notin rng(\theta)$. So the map $\theta^+ = \theta \cup
\{(a_t, \phi(b_t))\}$ is injective, and it is easily seen to be a
$\chi$-isomorphism in $\Theta^\chi$ and defined on $a_t$.
The converse,``back" property is similarly proved ( or by symmetry,
using the fact that the inverse of maps in $\Theta$ are
$\chi^{-1}$-isomorphisms).
\end{demo}

But we can also derive a connection between classical and
relativised semantics in $M$, over the following set $W$:\\

\begin{definition}
Let $W = \{ \bar{a} \in {}^n M : M \models ( \bigwedge_{i < j < n,
l < n} \neg (\rho, l)(x_i, x_j))(\bar{a}) \}.$
\end{definition}
\begin{theorem}
$M \models_W \varphi(\bar{a})$ iff $M \models \varphi(\bar{a})$, for
all $\bar{a} \in W$ and all $L^n$-formulas $\varphi$.
\end{theorem}
\begin{demo} {Proof} \cite{w}
\end{demo}

We can now extract form the labelled graph $M$ of Theorem  a
relativised set algebra $\A$, which will turn out to be
representable atomic polyadic algebra.
\begin{definition}
\begin{enumerate}
\item For an $L^n_{\infty \omega}$-formula $\varphi $, we define
$\varphi^W$ to be the set $\{ \bar{a} \in W : M \models_W \varphi
(\bar{a}) \}$. 

\item We define $\A$ to be the relativised set algebra with domain
$$\{\varphi^W : \varphi \,\ \textrm {a first-order} \;\ L^n-
\textrm{formula} \}$$  and unit $W$, endowed with the algebraic
operations ${\sf d}_{ij}, {\sf c}_i, $ ect., in the standard way .

\end{enumerate}

Note that $\A$ is indeed closed under the operations and so is a
bona fide relativised set algebra. For, reading off from the
definitions of the standard operations and the relativised
semantics, we see that for all $L^n$-formulas $\varphi,\psi,$
\begin{itemize}
\item $-^{\A}(\varphi^W)=(\neg\varphi)^W$
\item $\varphi^W \cdot ^{\A} \psi^W = (\varphi \wedge \psi)^W$
\item ${\sf d}^{\A}_{ij}=(x_i = x_j)^W \,\ \textrm{for all} \;\ i,j<n.$
\item ${\sf c}_i^{\A}(\varphi^W)=(\exists x_i \varphi)^W \,\ \textrm{for all} \,\ i <
n.$\\ 
For a formula $\phi$ and $i,j<n$ , $\phi[x_i,x_j]$ stands for the formula obtained from $\phi$ by interchanging the free occurences of $x_i$ 
and $x_j$. Then we have:
\item ${\sf p}_{ij}^{\A}(\varphi^W)=\varphi[x_i,x_j]^W$
\end{itemize}

\end{definition}

\begin{theorem}
$\A$ is a representable (countable) atomic polyadic algebra
\end{theorem}
\begin{demo}{Proof}
\cite{w} prop. 5.2
\end{demo}

Define $\cal C$ to be the complex algebra over $At\A$, the atom structure of $\A$.
Then $\cal C$ is the completion of $\A$. The domain of $\cal C$ is $\wp(At\A)$. The diagonal ${\sf d}_{ij}$ is interpreted as the set of all $S\in At\A$ with $a_i=a_j$ for some $\bar{a}\in S$.
The cylindrification ${\sf c}_i$ is interpreted by ${\sf c}_iX=\{S\in At\A: S\subseteq c_i^{\A}(S')\text { for some } S'\in X\}$, for $X\subseteq At\A$.
Finally ${\sf p}_{ij}X=\{S\in At\A: S\subseteq {\sf p}_{ij}^{\A}(S')\text { for some } S'\in X\}.$
Let $\cal D$ be the relativized set algebra with domain $\{\phi^W: \phi\text { an $L_{\infty\omega}^n$ formula }\}$,  unit $W$
and operations defined like those of $\cal A$.
Then it is not hard to show that ${\cal C}\cong \cal D$, via the map $X\mapsto \bigcup X$.

We shall use the following form of Ramsey's theorem:
If $n<\omega$, $S$ is a finite set, $f:[\omega]^n\to S$ is a map, then there exists an infinite $H\subseteq \omega$ such that 
$f\upharpoonright [H]^n$ is constant. Here $[X]^n$ denotes the set of all subsets of $X$ of size $n$.

\begin{theorem} 
$\Rd_{ca}{\cal C}\notin S\Nr_n\CA_{n+2}.$ In particular, $\Rd_{ca}\cal C$ is not representable.
\end{theorem} 

\begin{demo}{Proof} We define a relation algebra atom structure $\alpha(\G)$ of the form
$(\{1'\}\cup (\G\times n), R_{1'}, \breve{R}, R_;)$.
The only identity atom is $1'$. All atoms are self converse, 
so $\breve{R}=\{(a, a): a \text { an atom }\}.$
The colour of an atom $(a,i)\in \G\times n$ is $i$. The identity $1'$ has no colour. A triple $(a,b,c)$ 
of atoms in $\alpha(\G)$ is consistent if
$R;(a,b,c)$ holds. Then the consistent triples are $(a,b,c)$ where

\begin{itemize}

\item one of $a,b,c$ is $1'$ and the other two are equal, or

\item none of $a,b,c$ is $1'$ and they do not all have the same colour, or

\item $a=(a', i), b=(b', i)$ and $c=(c', i)$ for some $i<n$ and 
$a',b',c'\in \G$, and there exists at least one graph edge
of $G$ in $\{a', b', c'\}$.

\end{itemize}
$\alpha(\G)$ can be checked to be a relation atom structure. 
The atom structure of $\Rd_{ca}\A$ is isomorphic (as a cylindric algebra
atom structure) to the atom structure ${\cal M}_n$ of all $n$-dimensional basic
matrices over the relation algebra atom structure $\alpha(\G)$.
Indeed, for each  $m  \in {\cal M}_n, \,\ \textrm{let} \,\ \alpha_m
= \bigwedge_{i,j<n}  \alpha_{ij}. $ Here $ \alpha_{ij}$ is $x_i =
x_j$ if $ m_{ij} = 1$' and $R(x_i, x_j)$ otherwise, where $R =
m_{ij} \in L$. Then the map $(m \mapsto
\alpha^W_m)_{m \in {\cal M}_n}$ is a well - defined isomorphism of
$n$-dimensional cylindric algebra atom structures.
We can show that thae $\Cm\alpha(\G)$ is not representable like exactly in \cite{w} and Hodkinson using Ramseys theore.
Here we show something stronger.

We shall first show that $\Cm\alpha(\G)$ is not in $S\Ra \CA_{n+2}$. The idea is to use relativized representations. Such algebras are 
localy representable, but the epresentation is global enough so that Ramseys theorem applies. 
Hence the full complex cylindric algebra over the set of $n$ by $n$ basic matrices
- which is isomorphic to $\cal C$ is not in $S\Nr_n\CA_{n+2}$ for we have a relation algebra
embedding of $\Cm\alpha(\G)$ onto $\Ra\Cm{\cal M}_n$.
Assume for contradiction that $\Cm\alpha(\G)\in S\Ra \CA_{n+2}$. Then $\Cm\alpha(\G)$ has an $n$-flat representation $M$ \cite{HH}  13.46, 
which is $n$ square \cite{HH} 13.10. 
In particular, there is a set $M$, $V\subseteq M\times M$ and $h: \Cm\alpha(\G)\to \wp(V)$ 
such that $h(a)$ ($a\in \Cm\alpha(\G)$) is a binary relation on $M$, and
$h$ respects the relation algebra operations. Here $V=\{(x,y)\in M\times M: M\models 1(x,y)\}$, where $1$ is the greatest element of 
$\Cm\alpha(\G)$.
A clique $C$ of $M$ is a subset of the domain $M$ such that for $x,y\in C$ we have $M\models 1(x,y)$, equivalently $(x,y)\in V$.
Since $M$ is $n+2$ square, then for all cliques $C$ of $M$ with $|C|<n+2$, all $x,y\in C$ and $a,b\in \Cm\alpha(\G)$, $M\models (a;b)(x,y)$ 
there exists $z\in M$ such that $C\cup \{z\}$ is a clique and $M\models a(x,z)\land b(z,y)$.
For $Y\subseteq \N$ and $s<n$, set 
$$[Y,s]=\{(l,s): l\in Y\}.$$
For $r\in \{0, \ldots N-1\},$ $N\N+r$ denotes the set $\{Nq+r: q\in \N\}.$
Let $$J=\{1', [N\N+r, s]: r<N,  s<n\}.$$
Then $\sum J=1$ in $\Cm\alpha(\G).$
As $J$ is finite, we have for any $x,y\in M$ there is a $P\in J$ with
$(x,y)\in h(P)$.
Since $\Cm\alpha(\G)$ is infinite then $M$ is infinite. 
By Ramsey's Theorem, there are distinct
$x_i\in X$ $(i<\omega)$, $J\subseteq \omega\times \omega$ infinite
and $P\in J$ such that $(x_i, x_j)\in h(P)$ for $(i, j)\in J$, $i\neq j$. Then $P\neq 1'$. 
Also $(P;P)\cdot P\neq 0$. 
This follows from $n+2$ squareness and that if $x,y, z\in M$, 
$a,b,c\in \Cm\alpha(\G)$, $(x,y)\in h(a)$, $(y, z)\in h(b)$, and 
$(x, z)\in h(c)$, then $(a;b)\cdot c\neq 0$. 
A non -zero element $a$ of $\Cm\alpha(\G)$ is monochromatic, if $a\leq 1'$,
or $a\leq [\N,s]$ for some $s<n$. 
Now  $P$ is monochromatic, it follows from the definition of $\alpha$ that
$(P;P)\cdot P=0$. This contradiction shows that 
$\Cm\alpha(\G)$ is not in $S\Ra\CA_{n+2}$. Hence $\Cm{\cal M}_n\notin S\Nr_n\CA_{n+2}$.
\end{demo}

\section{Strongly representable atom structures}

From now on we follow closely \cite{strong}, generalizing this result to other algebras, starting from diagonal free to quasipolyadic equality algebras. 
In \cite{strong} definition 3.5, 
the authors define a cylindric atom structure based on a graph $\Gamma$.
We enrich this atom structure by the relations corresponding to the polyadic operations:

\begin{definition} We define an atom structure $\eta(\Gamma)=(H, D_{ij}, \equiv_i, P_{ij})$
as follows.
\begin{enumarab}
\item $H$ is the set of all pairs $(K, \sim)$ where $K:n\to \Gamma\times n$ is a partial map and $\sim$ is an equivalence relation
on $n$ satisfying the followng conditions

(a) If $|n/\sim|=n$ then $dom(K)=n$ and $rng(K)$ is not  independent subset of $n$. 

(b) If $|n/\sim|=n-1$, then $K$ is defined only on the unique $\sim$ class $\{i,j\}$ say of size $2$
and $K(i)=K(j)$

(c) If $|n/\sim|\leq n-2$, then $K$ is nowhere defined.

\item $D_{ij}=\{(K,\sim)\in H: i\sim j\}$

\item $(K,\sim)\equiv_i (K', \sim')$ iff $K(i)=K'(i)$ and $\sim\upharpoonright (n\setminus \{i\})=\sim'\upharpoonright (n\setminus \{i\})$

\item $(K.\sim)\equiv_{ij}(K',\sim')$ iff $K(i)=K'(j)$ and $K(j)=K'(i),$ $K\upharpoonright n\sim \{i,j\}=K'\upharpoonright n\sim \{i,j\}$ 
and if $i\sim j$ then $\sim=\sim'$, if not, then $\sim'$ is related to $\sim$ as follows
For all $k\notin [i]_{\sim}\cup [j]_{\sim}$ $[k]_{\sim'}=[k]_{\sim}$ $[i]_{\sim'}=[j]_{\sim}\setminus \{j\}\cup \{i\}$ and 
$[j]_{\sim'}=[i]_{\sim}\setminus \{i\}\cup \{j\}.$ 
\end{enumarab}
\end{definition}

\begin{definition} Let $\C(\Gamma)$ be the complex algebra of polyadic type of the above atom structure.
That is $\C(\Gamma)=(\B(\eta(\Gamma)), {\sf c}_i, {\sf s}_i^j, {\sf p}_{ij}, {\sf d}_{ij})_{i,j<n}$ with extra non-Boolean operations
defined by: 
$${\sf d}_{ij}=D_{ij}$$
$${\sf c}_iX=\{c: \exists a\in X, a\equiv_ic\}.$$
$${\sf p}_{ij}X=\{c:\exists a\in X, a\equiv_{ij}c\}$$
and $${\sf s}_i^jx={\sf c}_j(x\cap d_{ij}).$$
\end{definition}
For $\A\in \CA_n$ and $x\in A$, recall that $\Delta x,$ the dimension set of $x$, is the set $\{i\in n: {\sf c}_ix\neq x\}$.

\begin{theorem} For any graph $\Gamma$, $\C(\Gamma)$ is a simple $\PEA_n$, that is generated by the set
$\{x\in C: \Delta x\neq n\}$.
\end{theorem}
\begin{demo}{Proof} $\Rd_{ca}\C(\Gamma)$ is a simple $\CA_n$ by \cite{h} lemma 5.1, 
hence if we prove that $\C(\Gamma)$ is a polyadic equality algebra, then as a polyadic equality algebra it will be simple. 
This follows from the simple observation that any polyadic ideal in $\C(\Gamma)$ is a cylindric ideal.
Furthermore for any atom $x=\{(K, \sim)\}$ of $\C$ we have $x={\sf c}_0x\cap {\sf c}_1x\ldots\cap {\sf c}_nx$.
We need to check the polyadic axioms. 
Since $\Rd_{ca}\A$ is a $\CA_n$ we need to show that 
the following hold for all $i,j,k\in n$:
\begin{enumarab}
\item ${\sf p}_{ij}$'s  are boolean endomorphisms
\item ${\sf p}_{ij}{\sf p}_{ij}x=x$
\item ${\sf p}_{ij}{\sf p}_{ik}={\sf p}_{jk}{\sf p}_{ij}x   \text { if } |\{i,j,k\}|=3$
\item ${\sf p}_{ij}{\sf s}_i^jx={\sf s}_j^ix$
\end{enumarab}
These properties, follow from the definitions, and are therefore left to the reader.
\end{demo}
Let $\Gamma=(G,E)$ be a graph. Then A set $X\subseteq G$ is independent if $E\cap (X\times X)=\emptyset$.
The chromatic number $\chi(\Gamma)$ of $\Gamma$ is the least  $k<\omega$ such that $G$ can be partitioned into $k$ independent sets, and
$\infty$ if there is no such set.

\begin{theorem} 
\begin{enumroman}
\item Suppose that $\chi(\Gamma)=\infty$. Then $\C(\Gamma)$ is representable  as a polyadic equality algebra.
\item If $\Gamma$ is infinite and $\chi(\Gamma)<\infty$ then $\Rd_{df}\C(\Gamma)$ is not representable.
\end{enumroman}
\end{theorem}
\begin{demo}{Proof} (i) We have $\Rd_{ca}\C(\Gamma)$ is representable. 
Let $J=\{x\in C: \Delta x\neq n\}$. Then $\C(\Gamma)$ is generated from $J$ using infinite intersections and complementation.
Let $f$ be an isomorphism of of $\Rd_{ca}\C(\Gamma)$ onto a cylindric set algebra with base $U$. Since the ${\sf p}_{ij}$'s  
distribute over arbitrary (unions and) intersections, 
it suffices to  show that $f{\sf p}_{kl}x={\sf p}_{kl}fx$ for all $x\in J$. Let $\mu\in n\setminus \Delta x$.
If $k=\mu$ or $l=\mu$, say $k=\mu$, then using the polyadic axioms we have
$$f{\sf p}_{kl}x=f{\sf }{\sf p}_{kl}{\sf c}_kx=f{\sf s}_l^kx={\sf s}_l^kfx={\sf s}_l^k{\sf c}_kfx= {\sf p}_{kl}fx.$$ 
If $\mu\neq k,l$ then again using the polyadic axioms we get
$$f{\sf p}_{kl}x=f{\sf s}_{\mu}^l{\sf s}_l^k {\sf s}_k^{\mu}{\sf c}_{\mu}x=
{\sf s}_{\mu}^l{\sf s}_l^k {\sf s}_k^{\mu}{\sf c}_{\mu}fx={\sf p}_{kl}f(x)$$

(ii) Note that $\Rd_{ca}\C(\Gamma)$ is generated by $\{x\in C: \Delta x\neq n\}$ using infinite intersections and complementation.
\end{demo}
Recall that an atom structure is strongly representable if the complex algebra over this atom structure is representable \cite{h}.
We now have:
\begin{theorem}\label{st} Let $t$ be any signature between 
$\Df_n$ and $\PEA_n$. Then the class of strongly representable atom structures of type $t$
is not elementary.
\end{theorem}
\begin{demo}{Proof} \cite{strong} theorem 6.1. 
By a famous theorem of Erdos, for every $k<\omega$, there is a finite graph $G_k$ with $\chi(G_k)>k$ 
and with no cycles of length $<k$.
Let $\Gamma_k$ be the disjoint union of of the $G_l$ for $l>k$. Then $\chi(\Gamma_k)=\infty$. Thus, by the previous theorem 
$\C(\Gamma_k)\in {\bf RPEA_n}$. In fact, being simple, $\C(\Gamma_k)$ is actually a  polyadic set algebra.
Let $\Gamma$ be a non principal ultraproduct $\prod_D \Gamma_k$. So $\Gamma$ has no cycles, and so $\chi(\Gamma)\leq 2$.
It follows, again from the previous theorem,  that $\Rd_{df}\C(\Gamma)$ is not representable.
From $\prod_D\C(\Gamma_k)\cong \C(\prod_D\Gamma_k)$ we are done.
\end{demo}
\begin{theorem} For any such case, the classes of strongly representable and completely representable 
atom structures are not elementary
\end{theorem}

\section{Neat reducts, complete representations and games}

Next we characterize the class $\Nr_n\CA_{\omega}$ using games.
Our treatment in this part 
follows very closely \cite{R}. The essential difference is that we deal with $n$ 
dimensional networks and composition moves are replaced by cylindrifier moves in the games.

\begin{definition}\label{def:string} 
Let $n$ be an ordinal. An $s$ word is a finite string of substitutions $({\sf s}_i^j)$, 
a $c$ word is a finite string of cylindrifications $({\sf c}_k)$.
An $sc$ word is a finite string of substitutions and cylindrifications
Any $sc$ word $w$ induces a partial map $\hat{w}:n\to n$
by
\begin{itemize}

\item $\hat{\epsilon}=Id$

\item $\widehat{w_j^i}=\hat{w}\circ [i|j]$

\item $\widehat{w{\sf c}_i}= \hat{w}\upharpoonright(n\sim \{i\}$ 

\end{itemize}
\end{definition}
If $\bar a\in {}^{<n-1}n$, we write ${\sf s}_{\bar a}$, or more frequently 
${\sf s}_{a_0\ldots a_{k-1}}$, where $k=|\bar a|$,
for an an arbitary chosen $sc$ word $w$
such that $\hat{w}=\bar a.$ 
$w$  exists and does not 
depend on $w$ by \cite[definition~5.23 ~lemma 13.29]{HHbook}. 
We can, and will assume \cite[Lemma 13.29]{HHbook} 
that $w=s{\sf c}_{n-1}{\sf c}_n.$
[In the notation of \cite[definition~5.23,~lemma~13.29]{HHbook}, 
$\widehat{s_{ijk}}$ for example is the function $n\to n$ taking $0$ to $i,$
$1$ to $j$ and $2$ to $k$, and fixing all $l\in n\setminus\set{i, j,k}$.]
Let $\delta$ be a map. Then $\delta[i\to d]$ is defined as follows. $\delta[i\to d](x)=\delta(x)$
if $x\neq i$ and $\delta[i\to d](i)=d$. We write $\delta_i^j$ for $\delta[i\to \delta_j]$.

\begin{definition}
From now on let $2\leq n<\omega.$ Let $\C$ be an atomic $\CA_{n}$. 
An \emph{atomic  network} over $\C$ is a map
$$N: {}^{n}\Delta\to At\cal C$$ 
such that the following hold for each $i,j<n$, $\delta\in {}^{n}\Delta$
and $d\in \Delta$:
\begin{itemize}
\item $N(\delta^i_j)\leq {\sf d}_{ij}$
\item $N(\delta[i\to d])\leq {\sf c}_iN(\delta)$ 
\end{itemize}
\end{definition}
Note than $N$ can be viewed as a hypergraph with set of nodes $\Delta$ and 
each hyperedge in ${}^{\mu}\Delta$ is labelled with an atom from $\C$.
We call such hyperedges atomic hyperedges.
We write $\nodes(N)$ for $\Delta.$ But it can happen 
let $N$ stand for the set of nodes 
as well as for the function and the network itself. Context will help.

Define $x\sim y$ if there exists $\bar{z}$ such that $N(x,y,\bar{z})\leq {\sf d}_{01}$.
Define an equivalence relation
$\sim$ over the set of all finite sequences over $\nodes(N)$ by $\bar
x\sim\bar y$ iff $|\bar x|=|\bar y|$ and $x_i\sim y_i$ for all
$i<|\bar x|$.

(3) A \emph{ hypernetwork} $N=(N^a, N^h)$ over $\cal C$ 
consists of a network $N^a$
together with a labelling function for hyperlabels $N^h:\;\;^{<
\omega}\!\nodes(N)\to\Lambda$ (some arbitrary set of hyperlabels $\Lambda$)
such that for $\bar x, \bar y\in\; ^{< \omega}\!\nodes(N)$ 
\begin{enumerate}
\renewcommand{\theenumi}{\Roman{enumi}}
\setcounter{enumi}3
\item\label{net:hyper} $\bar x\sim\bar y \Rightarrow N^h(\bar x)=N^h(\bar y)$. 
\end{enumerate}
If $|\bar x|=k\in nats$ and $N^h(\bar x)=\lambda$ then we say that $\lambda$ is
a $k$-ary hyperlabel. $(\bar x)$ is referred to a a $k$-ary hyperedge, or simply a hyperedge.
(Note that we have atomic hyperedges and hyperedges) 
When there is no risk of ambiguity we may drop the superscripts $a,
h$. 

The following notation is defined for hypernetworks, but applies
equally to networks.  

(4) If $N$ is a hypernetwork and $S$ is any set then
$N\restr S$ is the $n$-dimensional hypernetwork defined by restricting
$N$ to the set of nodes $S\cap\nodes(N)$.  For hypernetworks $M, N$ if
there is a set $S$ such that $M=N\restr S$ then we write $M\subseteq
N$.  If $N_0\subseteq N_1\subseteq \ldots $ is a nested sequence of
hypernetworks then we let the \emph{limit} $N=\bigcup_{i<\omega}N_i$  be
the hypernetwork defined by
$\nodes(N)=\bigcup_{i<\omega}\nodes(N_i)$,\/ $N^a(x_0,\ldots x_{n-1})= 
N_i^a(x_0,\ldots x_{n-1})$ if
$x_0\ldots x_{\mu-1}\in\nodes(N_i)$, and $N^h(\bar x)=N_i^h(\bar x)$ if $\rng(\bar
x)\subseteq\nodes(N_i)$.  This is well-defined since the hypernetworks
are nested and since hyperedges $\bar x\in\;^{<\omega}\nodes(N)$ are
only finitely long.

For hypernetworks $M, N$ and any set $S$, we write $M\equiv^SN$
if $N\restr S=M\restr S$.  For hypernetworks $M, N$, 
and any set $S$, we write $M\equiv_SN$ 
if the symmetric difference $\Delta(\nodes(M), \nodes(N))\subseteq S$ and
$M\equiv^{(\nodes(M)\cup\nodes(N))\setminus S}N$. We write $M\equiv_kN$ for
$M\equiv_{\set k}N$.

Let $N$ be a network and let $\theta$ be any function.  The network
$N\theta$ is a complete labelled graph with nodes
$\theta^{-1}(\nodes(N))=\set{x\in\dom(\theta):\theta(x)\in\nodes(N)}$,
and labelling defined by 
$(N\theta)(i_0,\ldots i_{\mu-1}) = N(\theta(i_0), \theta(i_1), \theta(i_{\mu-1}))$,
for $i_0, \ldots i_{\mu-1}\in\theta^{-1}(\nodes(N))$.  Similarly, for a hypernetwork
$N=(N^a, N^h)$, we define $N\theta$ to be the hypernetwork
$(N^a\theta, N^h\theta)$ with hyperlabelling defined by
$N^h\theta(x_0, x_1, \ldots) = N^h(\theta(x_0), \theta(x_1), \ldots)$
for $(x_0, x_1,\ldots) \in \;^{<\omega}\!\theta^{-1}(\nodes(N))$.

Let $M, N$ be hypernetworks.  A \emph{partial isomorphism}
$\theta:M\to N$ is a partial map $\theta:\nodes(M)\to\nodes(N)$ such
that for any $
i_i\ldots i_{\mu-1}\in\dom(\theta)\subseteq\nodes(M)$ we have $M^a(i_1,\ldots i_{\mu-1})= 
N^a(\theta(i), \ldots\theta(i_{\mu-1}))$
and for any finite sequence $\bar x\in\;^{<\omega}\!\dom(\theta)$ we
have $M^h(\bar x) = 
N^h\theta(\bar x)$.  
If $M=N$ we may call $\theta$ a partial isomorphism of $N$.

\begin{definition}\label{def:games} Let $2\leq n<\omega$. For any $\CA_{n}$  
atom structure $\alpha$, and $n\leq m\leq
\omega$, we define two-player games $F_{n}^m(\alpha),$ \; and 
$H_{n}(\alpha)$,
each with $\omega$ rounds, 
and for $m<\omega$ we define $H_{m,n}(\alpha)$ with $n$ rounds.

\begin{itemize}
\item 
Let $m\leq \omega$.  
In a play of $F_{n}^m(\alpha)$ the two players construct a sequence of
networks $N_0, N_1,\ldots$ where $\nodes(N_i)$ is a finite subset of
$m=\set{j:j<m}$, for each $i$.  In the initial round of this game \pa\
picks any atom $a\in\alpha$ and \pe\ must play a finite network $N_0$ with
$\nodes(N_0)\subseteq  n$, 
such that $N_0(\bar{d}) = a$ 
for some $\bar{d}\in{}^{\mu}\nodes(N_0)$.
In a subsequent round of a play of $F_{n}^m(\alpha)$ \pa\ can pick a
previously played network $N$ an index $\l<n$, a ``face" 
$F=\langle f_0,\ldots f_{n-2} \rangle \in{}^{n-2}\nodes(N),\; k\in
m\setminus\set{f_0,\ldots f_{n-2}}$, and an atom $b\in\alpha$ such that 
$b\leq {\sf c}_lN(f_0,\ldots f_i, x,\ldots f_{n-2}).$  
(the choice of $x$ here is arbitrary, 
as the second part of the definition of an atomic network together with the fact
that $\cyl i(\cyl i x)=\cyl ix$ ensures that the right hand side does not depend on $x$).
This move is called a \emph{cylindrifier move} and is denoted
$(N, \langle f_0, \ldots f_{\mu-2}\rangle, k, b, l)$ or simply $(N, F,k, b, l)$.
In order to make a legal response, \pe\ must play a
network $M\supseteq N$ such that 
$M(f_0,\ldots f_{i-1}, k, f_i,\ldots f_{n-2}))=b$ 
and $\nodes(M)=\nodes(N)\cup\set k$.

\pe\ wins $F_{n}^m(\alpha)$ if she responds with a legal move in each of the
$\omega$ rounds.  If she fails to make a legal response in any
round then \pa\ wins.

\item
Fix some hyperlabel $\lambda_0$.  $H_{n}(\alpha)$ is  a 
game the play of which consists of a sequence of
$\lambda_0$-neat hypernetworks 
$N_0, N_1,\ldots$ where $\nodes(N_i)$
is a finite subset of $\omega$, for each $i<\omega$.  
In the initial round \pa\ picks $a\in\alpha$ and \pe\ must play
a $\lambda_0$-neat hypernetwork $N_0$ with nodes contained in
$\mu$ and $N_0(\bar d)=a$ for some nodes $\bar{d}\in {}^{\mu}N_0$.  
At a later stage
\pa\ can make any cylindrifier move $(N, F,k, b, l)$ by picking a
previously played hypernetwork $N$ and $F\in {}^{n-2}\nodes(N), \;l<n,  
k\in\omega\setminus\nodes(N)$ 
and $b\leq {\sf c}_lN(f_0, f_{l-1}, x, f_{n-2})$.  
[In $H_{n}$ we
require that \pa\ chooses $k$ as a `new node', i.e. not in
$\nodes(N)$, whereas in $F_{n}^m$ for finite $m$ it was necessary to allow
\pa\ to `reuse old nodes'. This makes the game easior as far as $\forall$ is concerned.) 
For a legal response, \pe\ must play a
$\lambda_0$-neat hypernetwork $M\equiv_k N$ where
$\nodes(M)=\nodes(N)\cup\set k$ and 
$M(f_0, f_{i-1}, k, f_{n-2})=b$.
Alternatively, \pa\ can play a \emph{transformation move} by picking a
previously played hypernetwork $N$ and a partial, finite surjection
$\theta:\omega\to\nodes(N)$, this move is denoted $(N, \theta)$.  \pe\
must respond with $N\theta$.  Finally, \pa\ can play an
\emph{amalgamation move} by picking previously played hypernetworks
$M, N$ such that $M\equiv^{\nodes(M)\cap\nodes(N)}N$ and
$\nodes(M)\cap\nodes(N)\neq \emptyset$.  
This move is denoted $(M,
N)$.  To make a legal response, \pe\ must play a $\lambda_0$-neat
hypernetwork $L$ extending $M$ and $N$, where
$\nodes(L)=\nodes(M)\cup\nodes(N)$.

Again, \pe\ wins $H_n(\alpha)$ if she responds legally in each of the
$\omega$ rounds, otherwise \pa\ wins. 

\item For $m< \omega$ the game $H_{m,n}(\alpha)$ is similar to $H_n(\alpha)$ but
play ends after $m$ rounds, so a play of $H_{m,n}(\alpha)$ could be
\[N_0, N_1, \ldots, N_m\]
If \pe\ responds legally in each of these
$m$ rounds she wins, otherwise \pa\ wins.
\end{itemize}

\end{definition}

\begin{definition}\label{def:hat}
For $m\geq 5$ and $\c C\in\CA_m$, if $\A\subseteq\Nr_n(\C)$ is an
atomic cylindric algebra and $N$ is an $\A$-network then we define
$\widehat N\in\C$ by
\[\widehat N =
 \prod_{i_0,\ldots i_{n-1}\in\nodes(N)}{\sf s}_{i_0, \ldots i_{n-1}}N(i_0\ldots i_{n-1})\]
$\widehat N\in\C$ depends
implicitly on $\C$.
\end{definition}
We write $\A\subseteq_c \B$ if $\A\in S_c\{\B\}$. 
\begin{lemma}\label{lem:atoms2}
Let $n<m$ and let $\A$ be an atomic $\CA_n$, 
$\A\subseteq_c\Nr_n\C$
for some $\C\in\CA_m$.  For all $x\in\C\setminus\set0$ and all $i_0, \ldots i_{n-1} < m$ there is $a\in\At(\A)$ such that
${\sf s}_{i_0\ldots i_{n-1}}a\;.\; x\neq 0$.
\end{lemma}
\begin{proof}
We can assume, see definition  \ref{def:string}, 
that ${\sf s}_{i_0,\ldots i_{n-1}}$ consists only of substitutions, since ${\sf c}_{m}\ldots {\sf c}_{m-1}\ldots 
{\sf c}_nx=x$ 
for every $x\in \A$.We have ${\sf s}^i_j$ is a
completely additive operator (any $i, j$), hence ${\sf s}_{i_0,\ldots i_{\mu-1}}$ 
is too  (see definition~\ref{def:string}).
So $\sum\set{{\sf s}_{i_0\ldots i_{n-1}}a:a\in\At(\A)}={\sf s}_{i_0\ldots i_{n-1}}
\sum\At(\A)={\sf s}_{i_0\ldots i_{n-1}}1=1$,
for any $i_0,\ldots i_{n-1}<n$.  Let $x\in\C\setminus\set0$.  It is impossible
that ${\sf s}_{i_0\ldots i_{n-1}}\;.\;x=0$ for all $a\in\At(\c A)$ because this would
imply that $1-x$ was an upper bound for $\set{{\sf s}_{i_0\ldots i_{n-1}}a:
a\in\At(\A)}$, contradicting $\sum\set{{\sf s}_{i_0\ldots i_{n-1}}a :a\in\At(\c A)}=1$.
\end{proof}

We now prove two Theorems relating neat embeddings
to the games we defined:

\begin{theorem}\label{thm:n}
Let $n<m$, and let $\A$ be a $\CA_m$.  
If $\A\in{\bf S_c}\Nr_{n}\CA_m, $
then \pe\ has a \ws\ in $F^m(\At\A)$. In particular if $\A$ is $CR$ then \pe has a \ws in $F^{\omega}(\At\A)$
\end{theorem}
\begin{proof}
If $\A\subseteq\Nr_n\C$ for some $\C\in\CA_m$ then \pe\ always
plays hypernetworks $N$ with $\nodes(N)\subseteq n$ such that
$\widehat N\neq 0$. In more detail, in the initial round , let $\forall$ play $a\in \At \cal A$.
$\exists$ play a network $N$ with $N(0, \ldots n-1)=a$. Then $\widehat N=a\neq 0$.
At a later stage suppose $\forall$ plays the cylindrifier move 
$(N, \langle f_0, \ldots f_{\mu-2}\rangle, k, b, l)$ 
by picking a
previously played hypernetwork $N$ and $f_i\in \nodes(N), \;l<\mu,  k\notin \{f_i: i<n-2\}$, 
and $b\leq {\sf c}_lN(f_0,\ldots  f_{i-1}, x, f_{n-2})$.
Let $\bar a=\langle f_0\ldots f_{l-1}, k\ldots f_{n-2}\rangle.$
Then ${\sf c}_k\widehat N\cdot {\sf s}_{\bar a}b\neq 0$.
Then there is a network  $M$ such that
$\widehat{M}.\widehat{{\sf c}_kN}\cdot {\sf s}_{\bar a}b\neq 0$. Hence 
$M(f_0,\dots  k, f_{n-2})=b.$
\end{proof}

\begin{theorem}\label{thm:RaC}
Let $\alpha$ be a countable 
$\CA_n$ atom structure.  If \pe\ has a \ws\ in $H_n(\alpha),$ then
there is a representable cylindric algebra $\C$ of
dimension $\omega$ such that $\Nr_n\c C$ is atomic 
and $\At \Nr_n\C\cong\alpha$.
\end{theorem}

\begin{proof} 
We shall construct a generalized atomic weak set algebra of dimension $\omega$ such that the atom 
structure of its full neat reduct is isomorphic to
the given atom structure. Suppose \pe\ has a \ws\ in $H_n(\alpha)$. Fix some $a\in\alpha$. We can define a
nested sequence $N_0\subseteq N_1\ldots$ of hypernetworks
where $N_0$ is \pe's response to the initial \pa-move $a$, requiring that
\begin{enumerate}
\item If $N_r$ is in the sequence and 
and $b\leq {\sf c}_lN_r(\langle f_0, f_{n-2}\rangle\ldots , x, f_{n-2})$.  
then there is $s\geq r$ and $d\in\nodes(N_s)$ such 
that $N_s(f_0, f_{i-1}, d, f_{n-2})=b$.
\item If $N_r$ is in the sequence and $\theta$ is any partial
isomorphism of $N_r$ then there is $s\geq r$ and a
partial isomorphism $\theta^+$ of $N_s$ extending $\theta$ such that
$\rng(\theta^+)\supseteq\nodes(N_r)$.
\end{enumerate}
We can schedule these requirements
to extend so that eventually, every requirement gets dealt with.
If we are required to find $k$ and $N_{r+1}\supset N_r$
such that 
$N_{r+1}(f_0, k, f_{n-2})=b$ then let $k\in \omega\setminus \nodes(N_r)$
where $k$ is the least possible for definiteness, 
and let $N_{r+1}$ be \pe's response using her \ws, 
to the \pa move $N_r, (f_0,\ldots f_{n-1}), k, b, l).$
For an extension of type 2, let $\tau$ be a partial isomorphism of $N_r$
and let $\theta$ be any finite surjection onto a partial isomorphism of $N_r$ such that 
$dom(\theta)\cap nodes(N_r)= dom\tau$. \pe's response to \pa's move $(N_r, \theta)$ is necessarily 
$N\theta.$ Let $N_{r+1}$ be her response, using her wining strategy, to the subsequent \pa 
move $(N_r, N_r\theta).$ 

Now let $N_a$ be the limit of this sequence.
This limit is well-defined since the hypernetworks are nested.  

Let $\theta$ be any finite partial isomorphism of $N_a$ and let $X$ be
any finite subset of $\nodes(N_a)$.  Since $\theta, X$ are finite, there is
$i<\omega$ such that $\nodes(N_i)\supseteq X\cup\dom(\theta)$. There
is a bijection $\theta^+\supseteq\theta$ onto $\nodes(N_i)$ and $j\geq
i$ such that $N_j\supseteq N_i, N_i\theta^+$.  Then $\theta^+$ is a
partial isomorphism of $N_j$ and $\rng(\theta^+)=\nodes(N_i)\supseteq
X$.  Hence, if $\theta$ is any finite partial isomorphism of $N_a$ and
$X$ is any finite subset of $\nodes(N_a)$ then
\begin{equation}\label{eq:theta}
\exists \mbox{ a partial isomorphism $\theta^+\supseteq \theta$ of $N_a$
 where $\rng(\theta^+)\supseteq X$}
\end{equation}
and by considering its inverse we can extend a partial isomorphism so
as to include an arbitrary finite subset of $\nodes(N_a)$ within its
domain.
Let $L$ be the signature with one $\mu$ -ary predicate symbol ($b$) for
each $b\in\alpha$, and one $k$-ary predicate symbol ($\lambda$) for
each $k$-ary hyperlabel $\lambda$.  

 
For fixed $f_a\in\;^\omega\!\nodes(N_a)$, let
$U_a=\set{f\in\;^\omega\!\nodes(N_a):\set{i<\omega:g(i)\neq
f_a(i)}\mbox{ is finite}}$.
Notice that $U_a$ is weak unit ( a set of sequences agrreing cofinitely with a fixed one)


We can make $U_a$ into the base of an $L$ relativized structure $\c N_a$.
Satisfiability for $L$ formluas at assignments $f\in U_a$ is defined the usual Tarskian way.

For $b\in\alpha,\;
l_0, \ldots l_{\mu-1}, i_0 \ldots, i_{k-1}<\omega$, \/ $k$-ary hyperlabels $\lambda$,
and all $L$-formulas $\phi, \psi$, let
\begin{eqnarray*}
\c N_a, f\models b(x_{l_0}\ldots  x_{n-1})&\iff&N_a(f(l_0),\ldots  f(l_{n-1}))=b\\
\c N_a, f\models\lambda(x_{i_0}, \ldots,x_{i_{k-1}})&\iff&  N_a(f(i_0), \ldots,f(i_{k-1}))=\lambda\\
\c N_a, f\models\neg\phi&\iff&\c N_a, f\not\models\phi\\
\c N_a, f\models (\phi\vee\psi)&\iff&\c N_a,  f\models\phi\mbox{ or }\c N_a, f\models\psi\\
\c N_a, f\models\exists x_i\phi&\iff& \c N_a, f[i/m]\models\phi, \mbox{ some }m\in\nodes(N_a)
\end{eqnarray*}
For any $L$-formula $\phi$, write $\phi^{\c N_a}$ for the set of assighnments satisfying it; that is
$\set{f\in\;^\omega\!\nodes(N_a): \c N_a, f\models\phi}$.  Let
$D_a = \set{\phi^{\c N_a}:\phi\mbox{ is an $L$-formula}}.$ 
Then this is the universe of the following weak set algebra 
\[\c D_a=(D_a,  \cup, \sim, {\sf D}_{ij}, {\sf C}_i)_{ i, j<\omega}\] 
then 
$\c D_a\in\RCA_\omega$. (Weak set algebras are representable).

Let $\phi(x_{i_0}, x_{i_1}, \ldots, x_{i_k})$ be an arbitrary
$L$-formula using only variables belonging to $\set{x_{i_0}, \ldots,
x_{i_k}}$.  Let $f, g\in U_a$ (some $a\in \alpha$) and suppose
is a partial isomorphism of $N_a$.  We can prove by induction over the
quantifier depth of $\phi$ and using (\ref{eq:theta}), that
\begin{equation}
\c N_a, f\models\phi\iff \c N_a,
g\models\phi\label{eq:bf}\end{equation} 

Let $\c C=\prod_{a\in
\alpha}\c \c D_a$.  Then  $\c C\in\RCA_\omega$, and $\c C$ is the desired generalized weak set algebra.
Note that unit of $\c C$ is the disjoint union of the weak spaces.
We set out to prove our claim. 
We shall show that $\alpha\cong \At\Nr_n\c C.$

An element $x$ of $\c C$ has the form
$(x_a:a\in\alpha)$, where $x_a\in\c D_a$.  For $b\in\alpha$ let
$\pi_b:\c C\to\c \c D_b$ be the projection defined by
$\pi_b(x_a:a\in\alpha) = x_b$.  Conversely, let $\iota_a:\c D_a\to \c
C$ be the embedding defined by $\iota_a(y)=(x_b:b\in\alpha)$, where
$x_a=y$ and $x_b=0$ for $b\neq a$.  Evidently $\pi_b(\iota_b(y))=y$
for $y\in\c D_b$ and $\pi_b(\iota_a(y))=0$ if $a\neq b$.

Suppose $x\in\Nr_{\mu}\c C\setminus\set0$.  Since $x\neq 0$, 
it must have a non-zero component  $\pi_a(x)\in\c D_a$, for some $a\in \alpha$.  
Say $\emptyset\neq\phi(x_{i_0}, \ldots, x_{i_k})^{\c
 D_a}= \pi_a(x)$ for some $L$-formula $\phi(x_{i_0},\ldots, x_{i_k})$.  We
 have $\phi(x_{i_0},\ldots, x_{i_k})^{\c D_a}\in\Nr_{\mu}\c D_a)$.  Pick
 $f\in \phi(x_{i_0},\ldots, x_{i_k})^{\c D_a}$ and let $b=N_a(f(0),
 f(1), \ldots f_{n-1})\in\alpha$.  We will show that 
$b(x_0, x_1, \ldots x_{n-1})^{\c D_a}\subseteq
 \phi(x_{i_0},\ldots, x_{i_k})^{\c D_a}$.  Take any $g\in
b(x_0, x_1\ldots x_{n-1})^{\c D_a}$, 
so $N_a(g(0), g(1)\ldots g(n-1))=b$.  The map $\set{(f(0),
g(0)), (f(1), g(1))\ldots (f(n-1), g(n-1))}$ is 
a partial isomorphism of $N_a$.  By
 (\ref{eq:theta}) this extends to a finite partial isomorphism
 $\theta$ of $N_a$ whose domain includes $f(i_0), \ldots, f(i_k)$. Let
 $g'\in U_a$ be defined by
\[ g'(i) =\left\{\begin{array}{ll}\theta(i)&\mbox{if }i\in\dom(\theta)\\
g(i)&\mbox{otherwise}\end{array}\right.\] By (\ref{eq:bf}), $\c N_a,
g'\models\phi(x_{i_0}, \ldots, x_{i_k})$. Observe that
$g'(0)=\theta(0)=g(0)$ and similarly $g'(n-1)=g(n-1)$, so $g$ is identical
to $g'$ over $\mu$ and it differs from $g'$ on only a finite
set of coordinates.  Since $\phi(x_{i_0}, \ldots, x_{i_k})^{\c
\ D_a}\in\Nr_{\mu}(\c C)$ we deduce $\c N_a, g \models \phi(x_{i_0}, \ldots,
x_{i_k})$, so $g\in\phi(x_{i_0}, \ldots, x_{i_k})^{\c D_a}$.  This
proves that $b(x_0, x_1\ldots x_{\mu-1})^{\c D_a}\subseteq\phi(x_{i_0},\ldots,
x_{i_k})^{\c D_a}=\pi_a(x)$, and so 
$$\iota_a(b(x_0, x_1,\ldots x_{n-1})^{\c \ D_a})\leq
\iota_a(\phi(x_{i_0},\ldots, x_{i_k})^{\c D_a})\leq x\in\c
C\setminus\set0.$$  Hence every non-zero element $x$ of $\Nr_{n}\c C$ 
is above
a an atom $\iota_a(b(x_0, x_1\ldots n_1)^{\c D_a})$ (some $a, b\in
\alpha$) of $\Nr_{n}\c C$.  So
$\Nr_{n}\c C$ is atomic and $\alpha\cong\At\Nr_{n}\c C$ --- the isomorphism
is $b \mapsto (b(x_0, x_1,\dots x_{n-1})^{\c D_a}:a\in A)$.
\end{proof}

We can use such games to show that for $n\geq 3$, there is a representable $\A\in \CA_n$ 
with atom structure $\alpha$ such that $\forall$ can win the game $F^{n+2}(\alpha)$.
However \pe\ has a \ws\ in $H_n(\alpha)$, for any $n<\omega$.
It will follow that there a countable cylindric algebra $\c A'$ such that $\c A'\equiv\c
A$ and \pe\ has a \ws\ in $H(\c A')$.
So let $K$ be any class such that $\Nr_n\CA_{\omega}\subseteq K\subseteq S_c\Nr_n\CA_{n+2}$.
$\c A'$ must belong to $\Nr_n(\RCA_\omega)$, hence $\c A'\in K$.  But $\c A\not\in K$
and $\c A\preceq\c A'$. Thus $K$ is not elementary. From this it easily follows that the class of completely representable cylindric algebras
is not elementary, and that the class $\Nr_n\CA_{n+k}$ for any $k\geq 0$ is not elementary either. 
Furthermore the constructions works for many variants of cylindric algebras
like Halmos' polyadic equality algebras and Pinter's substitution algebras.
Formally we shall prove:
\begin{theorem}\label{r} Let $3\leq n<\omega$. Then the following hold:
\begin{enumroman}
\item Any $K$ such that $\Nr_n\CA_{\omega}\subseteq K\subseteq S_c\Nr_n\CA_{n+2}$ is not elementary.
\item The inclusions $\Nr_n\CA_{\omega}\subseteq S_c\Nr_n\CA_{\omega}\subseteq S\Nr_n\CA_{\omega}$ are all proper
\end{enumroman}
\end{theorem}




\begin{thebibliography}{100}

\bibitem{AGMNS} Andr\'eka, H., Givant, S., Mikulas, S. N\'emeti, I., Simon A., 
{\it Notions of density
that imply representability in algebraic logic.} 
Annals of Pure and Applied logic, {\bf 91}(1998) p. 93 --190. 



\bibitem{1} H. Andreka, M. Ferenczi, I. Nemeti (editors) {\it Cylindric-like algebras and algebraic logic} Bolyai Society, Mathematical Studies, Springer
(2013).

\bibitem{ANT} Andr\'eka, H., N\'emeti, I., Sayed Ahmed, T., 
{\it Omitting types for finite variable fragments and complete representations of algebras.}
Journal of Symbolic Logic {\bf 73}(1) (2008) p.65-89


\bibitem{CF} Casanovas, E., and Farre, R. {\it Omitting Types in Incomplete Theories.}  
Journal of Symbolic logic, Vol. 61, Number 1, p. 236-245.
March 1996. 



\bibitem{DM} Daigneault, A., and Monk,J.D., 
{\it Representation Theory for Polyadic algebras}. 
Fund. Math. {\bf 52}(1963) p.151-176.




\bibitem{HMT1}L. Henkin, D. Monk, A. Tarski {\it Cylindric algebras, part 1} 1970

\bibitem{HMT2}L. Henkin, D. Monk, A. Tarski {\it Cylindric algebras, part 2} 1985

\bibitem{r} R. Hirsch {\it Relation algebra reducts of cylindric algebras and complete representations}. Journal of Symbolic Logic
(72) (2007) p. 673-703

\bibitem{Hodkinson} Hodkinson{\it Atom structures of relation and cylindric algebas} Appals(1997)p. 117-148.

\bibitem {Hirsch} Hirsch and Hodkinson  {\it  Complete representations in algebraic logic} JSL (62)(1997) p. 816-647


\bibitem{R} Hirsch R. {\it Relation algebra reducts of cylindric algebras and complete representations} Journal of Symbolic Logic (72) (2007) 673-703

\bibitem{HHbook} Hirsch, Hodkinson {\it Relation lgebras by Games}  147 Studies in Logic and Foundations of Mathematics Elsevier 
North Holland 2002

\bibitem{strong} Hirsch, Hodkinson {\it Strongly representable  atom structures of cylindric algebra} Journal of Symbolic Logic (74)(2009) 811-828

\bibitem{HHM}  Hirsch, Hodkinson and Maddux {\it On provability with finitely many varibales} Bulletin of Symbolic Logic 
(8) (2002) p. 329-347.

\bibitem {HHK} Hirsch Hodkinson and Kurucz {\it On modal logics between $K\times K\times K$ and $S5\times S5\times S5$.} 
Journal of Symbolic Logic(67) (2002) 221-234.


\bibitem{HH} Hirsch and Hodkinson {\it Completions and complete representations}
In \cite{1} p. 61-90.

\bibitem{Shelah} S. Shelah {\it Classification Theory } 1978

\bibitem{m} Maddux {\it Non finite axiomatizability results for cylindric and relation algebras} JSL, 54 (1989) 951-974.
\bibitem{N} Newelski, L. {\it Omitting types and the real line.}
Jornal of Symbolic Logic, {\bf 52}(1987),  p.1020-1026.

\bibitem{IGPL} Sayed Ahmed {\it The class of neat reducts is not elementary} Logic Journal of $IGPL$(9) (2001) 593-628

\bibitem{MLQ} Sayed Ahmed {\it A model theoretic solutin to a problem of Tarski} 
Mathematical Logic quarterly (48) (2002) 343-355.

\bibitem{FM} T. Sayed Ahmed {\it The class of 2 dimensional neat reducts of polyadic algebras is not elementary}. 
Fundementa Mathematica (172) (2002) p.61-81

\bibitem{MLQ} T. Sayed Ahmed {\it A model theoretic solution to a problem of Tarski} Mathematical Logic Quarterly (48) (2002)
p.343-355

\bibitem{Bulletin}  Sayed Ahmed,T., {\it Algebraic Logic, where does it stand today?}
Bulletin of Symbolic Logic. {\bf 11}(4)(2005), p.465-516.

\bibitem{w} Sayed Ahmed {\it  Weakly representable atom structures that are not strongly representable with an application to first order logic}
Math Logic Quarterly (2008) p. 294-306


\bibitem{OTT} Sayed Ahmed T., {\it Completions, complete representations and omiting types} In \cite{1} p. 205-222

\bibitem{OTT2} Sayed Ahmed, T. and Samir B., {\it Omitting types for first order logic with infinitary predicates}
Mathematical Logic Quaterly  {\bf 53}(6) (2007) p.564-576.


\bibitem{KS} Khaled, Sayed Ahmed {\it On complete representations of algebras of logic} 
Logic Journal of $IGPL$, (2009)   p.267-272

\bibitem{IGPl} Sayed ahmed {\it  The class of neat reducts is not elementary} Logic Journal of IGPL (9) (2001)p. 593-628


\bibitem{OTT2} Sayed Ahmed, T. and Samir B., {\it Omitting types for first order logic with infinitary predicates}
Mathematical Logic Quaterly  {\bf 53}(6) (2007) p.564-576.

\bibitem{hv} Hodkinson Venema{ \it Canonivcal varities with no canonical axiomatizations} Tarns 357 (2005) 4579-4605
 
\bibitem{ghv} {\it Erdos graphs resolves Fine's canonicity problem} Bull 10(2) (2004) 186-208.

\bibitem{Kechris} A.S. Kechris, \emph{Classical Descriptive Set Theory}, Springer Verlag, New York, 1995.


\bibitem{sep} M. Assem,  \emph{Separating Models by Formulas and the Number of Countable Models}, submitted (2013).
\bibitem{BeckerKechris} H. Becker and A.S. Kechris, \emph{The Descriptive Set Theory of Polish Group Actions}, Cambridge University Press, 1996. CMP 97:

\end{thebibliography}
\end{document}